\documentclass[a4paper]{article}
\usepackage[utf8]{inputenc}

\usepackage{amsthm,amsmath,amssymb,amsfonts}
\usepackage{latexsym}
\usepackage{mathrsfs} 
\usepackage{mathtools}
\usepackage{url}
\usepackage{enumerate}
\usepackage[margin=29mm]{geometry}
\usepackage{authblk}
\usepackage{hyperref}

\theoremstyle{plain}
\newtheorem{Thm}{Theorem}[section]
\newtheorem{Lem}[Thm]{Lemma}
\newtheorem{Prop}[Thm]{Proposition}

\theoremstyle{definition}

\newtheorem{Assum}[Thm]{Assumption}
\newtheorem{Rem}[Thm]{Remark}
\theoremstyle{remark}

\theoremstyle{plain}
\newtheorem*{Thm*}{Theorem}
\newtheorem*{Lem*}{Lemma}
\newtheorem*{Prop*}{Proposition}
\newtheorem*{Cor*}{Corollary}
\theoremstyle{definition}
\newtheorem*{Def*}{Definition}
\newtheorem*{Rem*}{Remark}
\theoremstyle{remark}
\newtheorem*{Eg*}{Example}

\numberwithin{equation}{section}

\DeclareMathOperator\supp{supp}

\title{Central limit theorems for nonlinear stochastic wave equations in dimension three}
\author{Masahisa Ebina} 

\affil{Department of Mathematics, Graduate School of Science, Kyoto University, Kitashirakawa-Oiwakecho, Sakyo-ku, Kyoto 606-8502, Japan}
\affil{E-mail address: \texttt{ebina.masahisa.83m@st.kyoto-u.ac.jp}}
\date{}

\begin{document}

\maketitle
\begin{abstract}
    In this paper, we consider three-dimensional nonlinear stochastic wave equations driven by the Gaussian noise which is white in time and has some spatial correlations. 
    Using the Malliavin-Stein's method, we prove the Gaussian fluctuation for the spatial average of the solution under the Wasserstein distance in the cases where the spatial correlation is given by an integrable function and by the Riesz kernel. In both cases we also establish functional central limit theorems.
\end{abstract}

\vskip\baselineskip




\section{Introduction}
In this paper we consider the following nonlinear stochastic wave equation
\begin{align}
\label{SPDE}
\begin{cases}
    \left(\frac{\partial^2 }{\partial t^2} - \Delta \right)u(t,x) =  \sigma(u(t,x))\dot{W}(t,x),\\
    u(0,x) = 1,\\
    \frac{\partial u}{\partial t}(0,x) = 0,
\end{cases}
\end{align}
on $[0,T] \times \mathbb{R}^3$, where $T > 0$ is fixed, $\Delta$ is Laplacian on $\mathbb{R}^3$, $\sigma:\mathbb{R} \to \mathbb{R}$, and $\dot{W}(t,x)$ is the formal notation of centered Gaussian noise defined on a complete probability space $(\Omega, \mathscr{F}, P)$. The covariance of $\dot{W}(t,x)$ is given by
\begin{align}
\label{dot W covariance}
\mathbb{E}[\dot{W}(t,x)\dot{W}(s,y)] = \delta_0(t-s)\gamma(x-y).
\end{align}
In \eqref{dot W covariance}, $\delta_0$ denotes the Dirac delta function and  $\gamma$ is a spatial correlation function such that $\gamma(x)dx$ is a nonnegative, nonnegative definite tempered measure on $\mathbb{R}^3$. 
Then, $\gamma$ has to be the Fourier transform of nonnegative tempered measure $\mu$, \textit{i.e.} $\gamma = \mathcal{F}\mu \;\; \text{in  $\mathcal{S}^{\prime}(\mathbb{R}^3)$}$, where $\mathcal{S}^{\prime}(\mathbb{R}^3)$ denotes the space of tempered distributions. 
The measure $\mu$ is called the spectral measure of $\gamma$.
See Section \ref{subsection Stochastic integrals} for details.

The equations \eqref{SPDE} are interpreted in the sense of Dalang-Walsh (see \cite{Walsh,Dal99}).
That is, a real-valued jointly measurable stochastic process $\{u(t,x) \mid (t,x) \in [0,T]\times \mathbb{R}^3 \}$ is a random field solution to \eqref{SPDE} if it is adapted to the filtration generated by the noise $W$ and satisfies 
\begin{align}
\label{random field solution}
    u(t,x) = 1 + \int_0^t\int_{\mathbb{R}^3}G(t-s,x-y)\sigma(u(s,y))W(ds,dy), \quad \text{a.s.}, 
\end{align}
for all $(t,x) \in [0,T]\times \mathbb{R}^3$. 
Here $G$ denotes the fundamental solution of the three-dimensional wave equation. It is well-known (see \textit{e.g.} \cite[Chapter 4]{Mizohata}) that 
\begin{align}
\label{fundamental solution}
    G(t, dx) = \frac{1}{4\pi t}\sigma_t(dx) \quad &t>0,
\end{align}
where $\sigma_t$ denotes the uniform surface measure on $\partial B_t \coloneqq \{x \in \mathbb{R}^3 \mid |x|=t \}$ with total measure $4\pi t^2$. 
The stochastic integral on the right-hand side of \eqref{random field solution} will be defined in Section \ref{subsection Stochastic integrals}.

Throughout the paper, we assume the following assumptions.
\begin{Assum}
\label{assumption}
\begin{enumerate}[(1)]
    \item $\sigma:\mathbb{R} \to \mathbb{R}$ is Lipschitz continuous function with Lipschitz constant $L \in (0,\infty)$. 
    \item $\sigma$ is continuously differentiable function (\textit{i.e.} $\sigma \in C^1(\mathbb{R})$) and $\sigma(1) \neq 0$.
    \item The spectral measure $\mu$ of $\gamma$ satisfies the so-called Dalang's condition:
    \begin{align}
    \label{Dalang's condition}
        \int_{\mathbb{R}^3}\langle x \rangle^{-2}\mu(dx) < \infty,
    \end{align}
    where $\langle x \rangle \coloneqq \sqrt{1+|x|^2}$ and $|x|$ denotes the usual Euclidean norm of $x$ on $\mathbb{R}^3$.
\end{enumerate}
\end{Assum}

Under the assumptions (1) and (3), it is known (see \cite{Dal99, DalangQuel}) 
that there exists a unique random field solution $\{u(t,x) \mid (t,x) \in [0,T]\times \mathbb{R}^3 \}$ of \eqref{SPDE} and that the solution is $L^2(\Omega)$-continuous and satisfies for all $p \geqslant 1$, 
\begin{equation}
\label{solution uniform bound}
    \sup_{(t,x) \in [0,T]\times\mathbb{R}^3}\mathbb{E}[|u(t,x)|^p] < \infty.
\end{equation}
Note that we need the assumption $\sigma \in C^1(\mathbb{R})$ to use the tools of the Malliavin calculus, and the assumption $\sigma(1) \neq 0$ excludes the trivial case: $u(t,x) = 1$, for all $(t,x) \in [0,T]\times \mathbb{R}^3$.

We are interested in the asymptotic behavior of the centered spatial integral of the form 
\begin{align}
\label{spatial integral}
    F_R(t)\coloneqq\int_{B_R}(u(t,x) - 1)dx
\end{align}
as $R \to \infty$, where $t>0$, $B_R\coloneqq \{x \in \mathbb{R}^3 \mid |x| \leqslant R \}$, and $u(t,x)$ is the solution to \eqref{SPDE}.
Our goal in the present paper is to show that $F_R(t)$ with some normalization has the Gaussian fluctuation as $R \to \infty$.

Recently, there has been a lot of research studying asymptotic behavior of spatial averages of stochastic partial differential equations. In order to explain our motivation, let us  briefly recall some previous works.
In \cite{CLTforSHE}, Huang, Nualart, and Viitasaari study one-dimensional stochastic heat equations driven by a space-time white noise.
Using the Malliavin-Stein approach, which is a combination of the Malliavin calculus and Stein's method (see \cite{nourdin_peccati_2012}), they prove the quantitative and functional central limit theorems (CLTs) for the spatial average of the solution.
After that, by similar arguments in  \cite{CLTforSHE}, the authors of \cite{SHEriesznoise} consider the same equation with the Gaussian noise which is white in time and which has spatial correlation $\gamma(x) = |x|^{-\beta}$, and show that similar results also hold in arbitrary spatial dimensions $d \geqslant 1$. 
For more related results concerned with stochastic heat equations, see \cite{SFHECLT, CLTforSHEvia, MR4421618, spatialergodicitydeltainitialcondition,  spatialstationarity,timedependclt, PAMroughnoise,nualart2021quantitative, averaginggaussianfunctionals, SHEboundary}.

As for stochastic wave equations, fractional Gaussian noise with Hurst parameter $H\in [1/2,1)$ in spatial dimension $d=1$ is considered in \cite{SWE1dfracnoise}, and the cases that $d=2$,  $\gamma(x) = |x|^{-\beta} \: (0<\beta < 2)$ and that $d\in \{1,2\}$, $\gamma \in L^1(\mathbb{R})$ if $d=1$ and $\gamma \in L^1(\mathbb{R}^2) \cap L^s(\mathbb{R}^2)$ for some $s >1$ if $d=2$ are studied in \cite{Averaging2dSWE} and \cite{CLTforSWEindimension1and2}, respectively. 
See also \cite{balan2021hyperbolic} for the case where Gaussian noise is colored in time and space.
However, unlike the case of stochastic heat equations, the results about CLTs for the spatial average of stochastic wave equations have only been known for $d \leqslant 2$.

In spatial dimension $d = 3$, the recent work \cite{SWEstationarity} establishes that the solution $u(t,x)$ of \eqref{SPDE} is spatially ergodic if the spectral measure $\mu$ of $\gamma$ satisfies $\mu(\{0\}) =0$. 
Under this condition, the mean ergodic theorem implies that
\begin{align*}
    \frac{F_R(t)}{|B_R|} \xrightarrow{R\to \infty} 0 
\end{align*}
in $L^2(\Omega)$, where $|B_R|$ is the volume of $B_R$.
Taking into account the cases $d \leqslant 2$, it is natural to ask whether $F_R(t)$ with some normalization also has the Gaussian fluctuation as $R \to \infty$.
Motivated by \cite{SWEstationarity}, our aim is to give an affirmative answer to this question.

In order to establish the Gaussian fluctuation, we have to impose some additional conditions for the spatial correlation function $\gamma$. 
In this paper we prove the CLTs for $F_R(t)$ under the two different conditions:
\begin{enumerate}
    \item[(i)] $\gamma \in L^1(\mathbb{R}^3)$ such that $\gamma(x) > 0$ for all $x \in \mathbb{R}^3$.
    \item[(ii)] $\gamma(x) = |x|^{-\beta}$ for some $\beta \in (0,2)$. 
\end{enumerate}

\begin{Rem}
\begin{enumerate}[(1)]
    \item Under the Dalang's condition, it is known (\textit{cf.} \cite{spatialergodicityforSPDEsvia, SWEstationarity}) that the sufficient condition for spatial ergodicity $\mu(\{0\})=0$ is equivalent to 
    \begin{align*}
        \lim_{R\to \infty}\frac{1}{R^3}\int_{B_R}\gamma(x)dx = 0.
    \end{align*}
    Hence spatial ergodicity for the solution holds for the both cases (i) and (ii).
    \item The spatial correlation $\gamma$ in the case (ii) is called the Riesz kernel. In this case, it is known (\textit{cf.} \cite[Chapter V]{singularintegrals}) that $\gamma(x)dx = |x|^{-\beta}dx$ is nonnegative definite tempered measure on $\mathbb{R}^3$ and its spectral measure is $\mu(dx) = c_{\beta}|x|^{\beta -3}dx$, where $c_{\beta}$ is a constant depending on $\beta$. Notice that $\mu(dx)$ satisfies the Dalang's condition \eqref{Dalang's condition} if and only if $\beta \in (0,2)$.
\end{enumerate}
\end{Rem}

To state the main results, let us now fix some notations.
$F_R(t)$ denotes the spatial integral defined by \eqref{spatial integral}, and set $\sigma_{R}(t)\coloneqq \sqrt{\mathrm{Var}(F_R(t))}$.
Set $\mathscr{H} \coloneqq \{h:\mathbb{R} \to \mathbb{R} \; | \; \lVert h\rVert_{\textrm{Lip}} \leqslant 1\}$, where 
\begin{align*}
    \lVert h \rVert_{\mathrm{Lip}} = \sup_{\substack{x \neq y \\ x,y \in \mathbb{R}}} \frac{|h(x)-h(y)|}{|x-y|}.
\end{align*}
Recall that the Wasserstein distance between the laws of two integrable real-valued random variables $X$ and $Y$ is defined by 
\begin{align}
\label{Wasserstein distance def}
    d_{\mathrm{W}}(X,Y) = \sup_{h \in \mathscr{H}}|\mathbb{E}[h(X)]-\mathbb{E}[h(Y)]|.
\end{align}
We also write $d_{\mathrm{W}}(X,\mathcal{N}(0,1))$ for the Wasserstein distance between the law of $X$ and the standard normal law.

We are now ready to state the first main result of this paper.
\begin{Thm}
\label{main result1}
Assume that the spatial correlation function $\gamma$ satisfies one of the two conditions below:
\begin{enumerate}[\normalfont(i)]
    \item $\gamma \in L^1(\mathbb{R}^3)$ and $\gamma(x) > 0$ for all $x \in \mathbb{R}^3$.
    \item $\gamma(x) = |x|^{-\beta}$ for some $\beta \in (0,2)$. 
\end{enumerate}
Then, for any fixed $t \in (0,T]$, we have $\sigma_R(t) > 0$ for every $R >0$ and 
\begin{align}
    \lim_{R \to \infty}d_{\mathrm{W}}\left( \frac{F_R(t)}{\sigma_R(t)}, \mathcal{N}(0,1) \right) = 0. \label{distance convergence}
\end{align}
\end{Thm}

\begin{Rem}
\begin{enumerate}[(1)]
    \item Theorem \ref{main result1} also holds under the Kolmogorov distance $d_{\mathrm{Kol}}$ and the Fortet-Mourier distance $d_{\mathrm{FM}}$ because
    \begin{align*}
        d_{\mathrm{Kol}}(X,\mathcal{N}(0,1)) \leqslant 2\sqrt{d_{\mathrm{W}}(X,\mathcal{N}(0,1))}, \quad d_{\mathrm{FM}}(X,Y) \leqslant d_{\mathrm{W}}(X,Y),
    \end{align*}
    for any integrable real-valued random variables $X$ and $Y$. See \cite[Appendix C]{nourdin_peccati_2012}.
    \item Taking into account the results in \cite{SWE1dfracnoise, Averaging2dSWE, CLTforSWEindimension1and2}, we expect that the convergence rate of \eqref{distance convergence} is $R^{-\frac{3}{2}}$ under the condition (i) in Theorem \ref{main result1} and $R^{-\frac{\beta}{2}}$ under the condition (ii). However, it seems difficult to obtain such convergence rates by our method and we leave it as a future problem. 
\end{enumerate}
\end{Rem}

Let $C([0,T])$ denote the space of continuous functions on $[0,T]$.
Set
\begin{align}
\label{notation a}
    \tau_{\beta} = \int_{B^2_1} |x-y|^{-\beta} dxdy, \quad \eta(r) = \mathbb{E}[\sigma(u(r,0))].
\end{align}
Here is the second main result.
\begin{Thm}
\label{main result2}
\begin{enumerate}[\normalfont(1)]
    \item Let $\gamma \in L^1(\mathbb{R}^3)$. Then, as $R \to \infty$, the process $\{ R^{-\frac{3}{2}}F_R(t) \mid t \in [0,T] \}$ converges weakly in $C([0,T])$, and the limiting process is a centered Gaussian process $\{\mathcal{G}_1(t) \mid t \in [0,T] \}$ with covariance function 
    \begin{align*}
        \mathbb{E}[\mathcal{G}_1(t) \mathcal{G}_1(s)] = |B_1|\int_{\mathbb{R}^3}\mathrm{Cov}(u(t,x),u(s,0))dx.
    \end{align*}
    \item Let $\gamma(x) = |x|^{-\beta}$ for some $0< \beta <2$. Then , as $R \to \infty$, the process $\{ R^{\frac{\beta}{2}-3}F_R(t) \mid t \in [0,T] \}$ converges weakly in $C([0,T])$, and the limiting process is a centered Gaussian process $\{\mathcal{G}_2(t) \mid t \in [0,T] \}$ with covariance function 
    \begin{align*}
        \mathbb{E}[\mathcal{G}_2(t) \mathcal{G}_2(s)] = \tau_{\beta}\int_0^{t\land s}(t-r)(s-r)\eta^2(r)dr,
    \end{align*}
    where $t \land s \coloneqq \min\{t,s\}$.
\end{enumerate}
\end{Thm}

Let us now briefly sketch the strategy of the proof of Theorem \ref{main result1}.
To follow the strategy of \cite{SWE1dfracnoise, Averaging2dSWE,CLTforSWEindimension1and2}, we may need some following-type pointwise estimate:
\begin{align}
\label{derivative moment pointwise estimate}
    \lVert D_{s,y}u(t,x) \rVert_p \lesssim G(t-s,x-y),
\end{align}
where $D$ denotes the Malliavin derivative operator which will be defined in Section \ref{subsection Malliavin calculus and Malliavin-Stein bound}, $G$ is the fundamental solution, and $\lVert \cdot \rVert_p$ denotes the $L^p(\Omega)$-norm.
This type of inequality works well when we consider stochastic heat equations in arbitrary spatial dimensions $d \geqslant 1$ and stochastic wave equations in spatial dimensions $d\leqslant 2$, because the corresponding fundamental solution $G$ is a function.
However, in $d=3$, the fundamental solution of the wave equation (see \eqref{fundamental solution}) is not a function but the measure, and the above type of pointwise estimate does not make sense. 
This is one of the main difficulties in proving CLTs for three-dimensional stochastic wave equations.

To avoid this problem, we consider the following Picard iteration scheme: Set $u_0(t,x)=1$, and 
\begin{align}
\label{Picard iteration at introduction}
    u_{n+1}(t,x) \coloneqq 1 + \int_0^t\int_{\mathbb{R}^3} G_{n+1}(t-s,x-y)\sigma(u_n(s,y))W(ds,dy)
\end{align}
for all $n \geqslant 0$,
where $G_{n+1}$ is the regularization of $G$ which will be defined in Section \ref{subsection some basic estimates}.  
Note that $G_{n+1}$ is a function and is much easier to handle than $G$.
Then, instead of \eqref{derivative moment pointwise estimate} we make use of the following estimate
\begin{align}
\label{derivative estimate introduction}
    \lVert D_{s,y}u_n(t,x) \rVert_p \lesssim C_n \mathbf{1}_{\{|x-y| \leqslant T+1\}},
\end{align}
where $C_n$ is a constant depending on $n$, and $\mathbf{1}$ denotes the indicator function.
This estimate as well as the standard proof strategy of \cite{SWE1dfracnoise, Averaging2dSWE,CLTforSWEindimension1and2} allows us to prove that for any fixed $n \geqslant 1$,
\begin{align*}
    d\left( \frac{F_{n,R}(t)}{\sigma_{n,R}(t)}, \mathcal{N}(0,1) \right) \xrightarrow{R \to \infty} 0,
\end{align*}
where $d$ is some distance between probability measures and
\begin{align*}
    F_{n,R}(t) \coloneqq \int_{B_R}(u_n(t,x) - 1)dx, \quad \sigma_{n,R}^2(t) \coloneqq \mathrm{Var}(F_{n,R}(t)).
\end{align*}
Since we have by the triangle inequality that
\begin{align*}
    d\left( \frac{F_{R}(t)}{\sigma_{R}(t)}, \mathcal{N}(0,1) \right) \leqslant d\left( \frac{F_{R}(t)}{\sigma_{R}(t)},  \frac{F_{n,R}(t)}{\sigma_{n,R}(t)}\right) + d\left( \frac{F_{n,R}(t)}{\sigma_{n,R}(t)}, \mathcal{N}(0,1) \right),
\end{align*}
our task is now to derive the uniform convergence 
\begin{align*}
    \sup_{R} d\left( \frac{F_{R}(t)}{\sigma_{R}(t)},  \frac{F_{n,R}(t)}{\sigma_{n,R}(t)}\right) \xrightarrow{n \to \infty} 0.
\end{align*}
For this purpose, the Wasserstein distance is more suitable than the total variation distance, and this is why we use $d_{\mathrm{W}}$ in our analysis.
Using the Wasserstein distance $d_{\mathrm{W}}$ and combining various estimates, we can show the above uniform convergence, and Theorem \ref{main result1} follows.

The rest of the paper is organized as follows: In Section \ref{Preliminaries} we recall some basic facts of stochastic integral and the Malliavin calculus. Then we introduce some preliminary results and estimates which are required for other sections.
In Section \ref{The Picard Approximation and its Malliavin derivative} we show some properties of the Picard approximation sequence and prove the estimate \eqref{derivative estimate introduction} for the Malliavin derivative of the sequence.
Section \ref{CLT} is devoted to the proof of Theorem \ref{main result1}.
In Section \ref{FCLT} we prove the second main result, Theorem \ref{main result2}.
Finally, Section \ref{Appendix} is intended to collect some technical estimates used in the paper.  
\vskip\baselineskip

\noindent
\textbf{Notation.}
We introduce some notations which are used throughout the paper. 
\begin{itemize}
    \item For any $x,y \in \mathbb{R}^3$, $|x|$ denotes the usual Euclidean norm of $x$ on $\mathbb{R}^3$ and $x\cdot y$ denotes the standard inner product such that $|x|^2 = x\cdot x$. We set $\langle x \rangle \coloneqq \sqrt{1+  |x|^2}$. For any $a,b \in \mathbb{R}$, let $a \land b \coloneqq \min\{a,b\}$ and $a \lor b \coloneqq \max\{a,b\}$.
    \item The closed ball of center $x$ and radius $r >0$ is denoted by $B_r(x) \coloneqq \{y \in \mathbb{R}^3 \,:\, | x-y | \leqslant r \}$, and we write $B_r$ instead of $B_r(0)$. The Lebesgue measure of a measurable set $A$ is denoted by $|A|$. 
    \item $C_0^{\infty}(\mathbb{R}^3)$, $\mathcal{S}(\mathbb{R}^3)$, and $\mathcal{S}^{\prime}(\mathbb{R}^3)$ denote the space of smooth functions with compact support, the Schwartz space of rapidly decreasing functions, and the space of tempered distributions on $\mathbb{R}^3$, respectively. Let $\mathcal{F}$ denote the Fourier transform operator on $\mathbb{R}^3$. For an integrable function $f$, its Fourier transform is defined by $\mathcal{F}f(\xi) \coloneqq \int_{\mathbb{R}^3}e^{-2\pi \sqrt{-1} \xi \cdot x} f(x)dx$. 
    \item $\lVert X \rVert_p \coloneqq \mathbb{E}[|X|^p]^{\frac{1}{p}}$ denotes the $L^p(\Omega, \mathscr{F}, P)$-norm of a real-valued random variable $X$.
    \item We use the notation $X \lesssim Y$ or $Y \gtrsim X$ to denote the estimate $X \leqslant CY$ for some constant $C>0$ which does not depend on X and Y. 
    In certain cases, we write $X \lesssim_{\alpha} Y$ to emphasize the dependence of the constant $C$ on a parameter $\alpha$.
\end{itemize}

\noindent
\textbf{Acknowledgment.} The author would like to express his deep gratitude to Professor Seiichiro Kusuoka for his helpful advice and encouragement. This work was supported by JSPS KAKENHI Grant Number JP22J21604.

\section{Preliminaries}
\label{Preliminaries}
\subsection{Stochastic integrals}
\label{subsection Stochastic integrals}
Following \cite{Dal99, DalangQuel}, we formulate the stochastic integral which is used in this paper.
Recall that the spatial correlation function $\gamma$ is a nonnegative function on $\mathbb{R}^3$ such that $\gamma(x)dx$ is a nonnegative definite tempered measure on $\mathbb{R}^3$. That is, $\gamma(x) \geqslant 0$,
\begin{align*}
    \int_{\mathbb{R}^3}(\varphi*\tilde{\varphi})(x)\gamma(x)dx \geqslant 0 
\end{align*}
for all $\varphi \in \mathcal{S}(\mathbb{R}^3)$, and there exists $k > 0$ such that
\begin{align*}
    \int_{\mathbb{R}^3}\langle x \rangle^{-k} \gamma(x)dx < \infty.
\end{align*}
Here $*$ denotes the convolution in space and $\tilde{\varphi}(x) \coloneqq \varphi(-x)$.
Then, Bochner-Schwartz theorem (see \cite[Chapitre VII, Th\'{e}or\`{e}me XVIII]{Schwartz}) implies that $\gamma$ is the Fourier transform of a nonnegative tempered measure $\mu$ \textit{i.e.} $\gamma = \mathcal{F}\mu \;\; \text{in  $\mathcal{S}^{\prime}(\mathbb{R}^3)$}$. This measure $\mu$ is called the spectral measure of $\gamma$.  By the definition of the Fourier transform 
in $\mathcal{S}^{\prime}(\mathbb{R}^3)$, this means that 
\begin{align*}
    \int_{\mathbb{R}^3}\varphi(x)\gamma(x)dx = \int_{\mathbb{R}^3}\mathcal{F}\varphi(\xi)\mu(d\xi)
\end{align*}
for all $\varphi \in \mathcal{S}(\mathbb{R}^3)$.

Let $(\Omega,\mathscr{F}, P)$ be a complete probability space. We consider an $L^2(\Omega,\mathscr{F},P)$-valued mean zero Gaussian process $W = \{W(\varphi) \mid \varphi \in C_0^{\infty}([0,T]\times \mathbb{R}^3)\}$ with covariance 
\begin{align*}
    \mathbb{E}[W(\varphi)W(\psi)] = \int_{0}^T dt\int_{\mathbb{R}^3}(\varphi(t)*\tilde{\psi}(t))(x)\gamma(x)dx = \int_{0}^T dt\int_{\mathbb{R}^3}\mathcal{F}\varphi(t)(\xi)\overline{\mathcal{F}\psi(t)(\xi)}\mu(d\xi),
\end{align*}
where $\tilde{\psi}(t,x)\coloneqq\psi(t,-x)$. We assume that $\mathscr{F}$ is generated by $W$.

Let $\mathcal{H}$ denote the completion of the space $\mathcal{S}(\mathbb{R}^3)$ with respect to the norm $\lVert \cdot \rVert_{\mathcal{H}}$ induced by the inner product 
\begin{align*}
    \langle \varphi,\psi \rangle_{\mathcal{H}} \coloneqq \int_{\mathbb{R}^3}\mu(d\xi)\mathcal{F}\varphi(t)(\xi)\overline{\mathcal{F}\psi(t)(\xi)} = \int_{\mathbb{R}^3}(\varphi *\tilde{\psi})(x)\gamma(x)dx , \quad \varphi, \psi \in \mathcal{S}(\mathbb{R}^3).
\end{align*}
Here we identify two functions $\varphi$ and $\psi$ if $\lVert \varphi - \psi \rVert_{\mathcal{H}} = 0$.
Then $\mathcal{H}$ is a separable real Hilbert space.
Note that generally $\mathcal{H}$ contains some distributions. 
In fact, under the Dalang's condition \eqref{Dalang's condition}, the fundamental solution of three-dimensional wave equation \eqref{fundamental solution} belongs to $\mathcal{H}$ (see also \cite[Remark 2.3]{DalangQuel}).
Define $\mathcal{H}_T = L^2([0,T];\mathcal{H})$. The norm $\lVert \cdot \rVert _{\mathcal{H}_T}$ is given by 
\begin{align*}
    \lVert \varphi \rVert _{\mathcal{H}_T}^2 = \int_{0}^T\lVert \varphi(t) \rVert_{\mathcal{H}}^2dt.
\end{align*}
We need the following lemma. See \cite[Lemma 2.4]{DalangQuel} for the proof.
\begin{Lem}
\label{dense lemma}
$C_0^{\infty}([0,T]\times \mathbb{R}^3)$ is dense in $(\mathcal{H}_T, \lVert \cdot \rVert _{\mathcal{H}_T})$.
\end{Lem}

Because $\varphi \mapsto W(\varphi)$ is a linear isometry from $(C_0^{\infty}([0,T]\times \mathbb{R}^3), \lVert \cdot \rVert _{\mathcal{H}_T})$ to $L^2(\Omega,\mathscr{F},P)$, it follows from Lemma \ref{dense lemma} that we can define $W(\varphi)$ for all $\varphi \in \mathcal{H}_T$ by extending the isometry. 
By an approximation argument, the space $\mathcal{H}$ contains indicator functions of bounded Borel sets (see Lemma \ref{Lemma fourier transform}).
Set $W_t(A)\coloneqq W(\mathbf{1}_{[0,t]}\mathbf{1}_{A})$ for all $t \geqslant 0$ and $A \in \mathcal{B}_b(\mathbb{R}^3)$, where $\mathcal{B}_b(\mathbb{R}^3)$ denotes the bounded Borel sets of $\mathbb{R}^3$.
Let $\mathscr{F}_t^0$ denote $\sigma$-field generated by the family of random variables $\{W_s(A) \mid A \in \mathcal{B}_b(\mathbb{R}^3), 0 \leqslant s \leqslant t\}$ and the $P$-null sets, and define $\mathscr{F}_t \coloneqq \cap_{s > t} \mathscr{F}_s^0$ for $t \in [0,T)$ and $\mathscr{F}_T \coloneqq \mathscr{F}_T^0$.
Then, the process $\{W_t(A), \mathscr{F}_t, t \in [0,T], A \in \mathcal{B}_b(\mathbb{R}^3)\}$ is a worthy martingale measure (see \cite{Walsh}), and its covariance measure and dominating measure are given by 
\begin{align*}
    \langle W(A), W(B) \rangle_t = t\int_{\mathbb{R}^6}\mathbf{1}_A(x)\mathbf{1}_B(y)\gamma(x-y)dxdy.
\end{align*}

Let $X = \{ X(t,x) \mid (t,x) \in [0,T] \times \mathbb{R}^3 \}$ be a random field on the probability space $(\Omega,\mathscr{F},P)$.
We say that $X$ is $(\mathscr{F}_t)$-adapted if $X(t,x)$ is $\mathscr{F}_t$-measurable for every $(t,x) \in [0,T] \times \mathbb{R}^3$. $X$ is stochastically continuous if it is continuous in probability at any point $(t,x) \in [0,T] \times \mathbb{R}^3$.
Let $\mathcal{P}$ denote the predictable $\sigma$-field on $[0,T] \times \mathbb{R}^3 \times \Omega$ with respect to the filtration $(\mathscr{F}_t)$. (\textit{cf.} \cite{Walsh}.) $X$ is predictable if it is measurable with respect to $\mathcal{P}$. 

It is known (see \cite[Proposition B.1]{MR4017124}) that any stochastically continuous and $(\mathscr{F}_t)$-adapted random field $X(t,x)$ has a predictable modification.
Then, for such a random field $X(t,x)$ satisfying
\begin{align}
\label{X c1}
    \mathbb{E}\left[ \int_0^T\int_{\mathbb{R}^6}|X(t,x)||X(t,y)|\gamma(x-y)dxdydt \right] < \infty,
\end{align}
the stochastic integral 
\begin{align*}
    Y_t = \int_0^t\int_{\mathbb{R}^3}X(s,x)W(ds,dx)
\end{align*}
is well-defined Walsh integral and its quadratic variation is given by 
\begin{align*}
    \langle Y \rangle_t = \int_0^t\int_{\mathbb{R}^6}X(s,x)X(s,y)\gamma(x-y)dxdyds.
\end{align*}
Note that jointly measurable deterministic functions are predictable.

If a random field $\{Z(t,x) \mid (t,x) \in [0,T]\times \mathbb{R}^3\}$ is a jointly measurable with respect to $\mathcal{B}([0,T]\times \mathbb{R}^3) \times \mathscr{F}$ such that for any fixed $t \geqslant 0$,
\begin{align}
\label{Z c1}
    \mathbb{E}[Z(t,x)Z(t,y)] = \mathbb{E}[Z(t,x-y)Z(t,0)], \quad x ,y \in \mathbb{R}^3, 
\end{align}
and 
\begin{align}
\label{Z c2}
    \sup_{(t,x)\in [0,T] \times \mathbb{R}^3} \mathbb{E}[|Z(t,x)|^2] < \infty,
\end{align}
then, owing to \cite{Dal99}, there is a nonnegative tempered measure $\mu_t^{Z}$ on $\mathbb{R}^3$ such that 
\begin{align*}
    \mathbb{E}[Z(t,\cdot)Z(t,0)]\gamma(\cdot) = \mathcal{F}\mu_t^{Z}
\end{align*}
in $\mathcal{S}^{\prime}(\mathbb{R}^3)$.
Under Assumption \ref{assumption}, it is known (see \textit{e.g.} \cite{Dal99,SWEstationarity}) that a unique random field solution $u(t,x)$ to \eqref{SPDE} is strictly stationary in space variable. That is, the finite-dimensional distributions of the process $\{u(t,x+y) \mid x \in \mathbb{R}^3 \}$ are independent of $y \in \mathbb{R}^3$. Therefore, by \eqref{solution uniform bound}, $Z(t,x) \coloneqq \sigma(u(t,x))$ satisfies \eqref{Z c1} and \eqref{Z c2}, and there is a nonnegative tempered measure $\mu_t^{\sigma(u)}$ such that 
\begin{align}
\label{mu^sigma(u)}
    \mathbb{E}[\sigma(u(t,\cdot))\sigma(u(t,0))]\gamma(\cdot) = \mathcal{F}\mu_t^{\sigma(u)}
\end{align}
in $\mathcal{S}^{\prime}(\mathbb{R}^3)$ for given $t$.

In order for the stochastic integral on the right-hand side of \eqref{random field solution} to have exact meaning, we need the following result due to Dalang. See \cite[Theorem 2]{Dal99} for the proof. 
Note that the fundamental solution of the three-dimensional wave equation $G(t)$ is the nonnegative distribution with rapid decrease on $\mathbb{R}^3$ for all $t >0$.

\begin{Prop}
\label{dalang prop}
Let $t \mapsto S(t)$ be a deterministic function with values in the space of nonnegative distributions on $\mathbb{R}^3$ with rapid decrease, such that 
\begin{align*}
    \int_0^T dt\int_{\mathbb{R}^3}\mu(d\xi)|\mathcal{F}S(t)(\xi)|^2 < \infty.
\end{align*}
Let $\{Z(t,x) \mid (t,x) \in [0,T]\times \mathbb{R}^3\}$ be a predictable random field which satisfies \eqref{Z c1} and \eqref{Z c2}.
Then the stochastic integral 
\begin{align*}
    \int_0^t\int_{\mathbb{R}^3}S(s,x)Z(s,x)W(ds,dx)
\end{align*}
is well-defined in the sense of Dalang and 
\begin{align}
    \mathbb{E}\left[\left( \int_0^t\int_{\mathbb{R}^3}S(s,x)Z(s,x)W(ds,dx)\right)^2 \right]
    &= \int_0^t ds \int_{\mathbb{R}^3} |\mathcal{F}S(s)(\xi)|^2\mu_s^Z(d\xi) \label{isometry fourier}\\
    &\leqslant \int_0^t ds \left( \sup_{x\in \mathbb{R}^3} \mathbb{E}[|Z(s,x)|^2] \right)\int_{\mathbb{R}^3}\mu(d\xi)|\mathcal{F}S(s)(\xi)|^2. \label{isometry fourier bound}
\end{align}
\end{Prop}

To estimate the $L^p(\Omega)$-norm of stochastic integrals, we introduce the following version of the Burkholder-Davis-Gundy inequality (BDG inequality). See \textit{e.g.} \cite[Appendix B]{AnalysisofSPDEs} for more details.
\begin{Lem}[Burkholder-Davis-Gundy inequality]
\label{BDG inequality}
Let $\{X(t,x) \mid (t,x) \in [0,T]\times \mathbb{R}^3\}$ be a predictable random field which satisfies \eqref{X c1}.
Then, for every $p \in [2,\infty)$ and $t \in [0,T]$, 
\begin{align*}
    \left \lVert \int_0^t\int_{\mathbb{R}^3}X(s,x)W(ds,dx) \right \rVert_p^2 \leqslant 4p \left\lVert \int_0^t\int_{\mathbb{R}^6}X(s,x)X(s,y)\gamma(x-y)dxdyds \right\rVert_{\frac{p}{2}}.
\end{align*}
\end{Lem}

\subsection{The Malliavan calculus and Malliavin-Stein bound}
\label{subsection Malliavin calculus and Malliavin-Stein bound}
In this section, we first recall some basic facts of the Malliavin calculus based on the Gaussian process $\{W(\varphi) \mid \varphi \in \mathcal{H}_T\}$ defined in Section 2.1.
Then we introduce some preliminary results of Malliavin-Stein's method needed for our proofs.
For details, the reader is referred to \cite{nourdin_peccati_2012,Nualartbook}.

Let $C_{\mathrm{p}}^{\infty}(\mathbb{R}^m)$ denote the space of smooth functions $f:\mathbb{R}^m \to \mathbb{R}$ such that all their partial derivatives have at most polynomial growth. The space of all smooth random variables of the form $F=f(W(\varphi_1),\ldots W(\varphi_m))$, where $m\geqslant 1$, $f \in C_{\mathrm{p}}^{\infty}(\mathbb{R}^m)$, and $\varphi_i \in \mathcal{H}_T, \; i = 1,2,\ldots,m$
is denoted by $\mathcal{S}$.
For a smooth random variable $F$ of the form above, its Malliavin derivative is given by
\begin{align*}
    DF = \sum_{i=1}^m\frac{\partial f}{\partial x_i}(W(\varphi_1),\ldots, W(\varphi_m))\varphi_i.
\end{align*}
Clearly, $DF$ is the $\mathcal{H}_T$-valued random variable. 

For any $p\in[1,\infty)$, let $\mathbb{D}^{1,p}$ denote the closure of $\mathcal{S}$ with respect to the norm 
\begin{align*}
    \lVert F \rVert_{1,p} = ( \mathbb{E}[|F|^p] + \mathbb{E}[\lVert DF \rVert_{\mathcal{H}_T}^p] )^{\frac{1}{p}}.
\end{align*}
The Malliavin derivative operator $D:L^p(\Omega) \to L^p(\Omega;\mathcal{H}_T)$, initially defined on $\mathcal{S}$, is closable and can be extended to $\mathbb{D}^{1,p}$. The closure of $D$ is again denoted by $D$. 
When $F \in \mathbb{D}^{1,p}$ and $DF$ is a random function valued in $\mathcal{H}_T$, we write this function as $D_{t,x}F, \ \  (t,x) \in [0,T] \times \mathbb{R}^3$.
For instance, if $F = f(W(\varphi_1),...,W(\varphi_m))$ for some $f \in C_{\mathrm{p}}^{\infty}(\mathbb{R}^m)$ and $\varphi_i \in C_0^{\infty}([0,T] \times \mathbb{R}^3), \; i = 1,2,\ldots,m$, then $DF$ is a random function and 
\begin{align*}
    D_{t,x}F = \sum_{i=1}^m\frac{\partial f}{\partial x_i}(W(\varphi_1),\ldots, W(\varphi_m))\varphi_i(t,x).
\end{align*}

The operator $D$ satisfies the following chain rule: Let $p \geqslant 1$. Suppose that $F \in \mathbb{D}^{1,p}$ and $\psi:\mathbb{R}\to \mathbb{R}$ is a continuously differentiable function with bounded derivative.
Then $\psi(F) \in \mathbb{D}^{1,p}$ with 
\begin{align*}
    D(\psi(F)) = \psi^{\prime}(F)DF.
\end{align*}

Let $\delta$ and $\mathrm{Dom}(\delta)$ denote the adjoint operator of $D:\mathbb{D}^{1,2} \to L^2(\Omega; \mathcal{H}_T)$ and its domain, respectively. The relationship between $D$ and $\delta$ is characterized by the duality formula 
\begin{align}
\label{duality formula}
    \mathbb{E}[\delta(u)F] = \mathbb{E}[\langle u, DF \rangle_{\mathcal{H}_T}],
\end{align}
where $u \in \mathrm{Dom}(\delta) \subset L^2(\Omega;\mathcal{H}_T)$ and $F \in \mathbb{D}^{1,2}$. In particular, $\mathbb{E}[\delta(u)] = 0$ for all $u \in \mathrm{Dom}(\delta)$ because $D1 = 0$. 
In our setting, it is known that the operator $\delta$ coincides with the stochastic integral defined in Section 2.1. 
That is, for any predictable random field $X = \{ X(t,x) \mid (t,x) \in [0,T] \times \mathbb{R}^3 \}$ which satisfies \eqref{X c1}, we have $X \in \mathrm{Dom}(\delta)$ and 
\begin{align}
\label{skorokhod integral}
    \delta(X) = \int_0^T\int_{\mathbb{R}^3}X(t,x)W(dt,dx), \quad \text{a.s.}.
\end{align}

One of the important results in the Malliavin calculus is the following proposition, which is known as the Clark-Ocone formula. See, for instance, \cite[Proposition 6.3]{spatialergodicityforSPDEsvia} for the proof.
\begin{Prop}
Suppose $F \in \mathbb{D}^{1,2}$. Then, $F$ can be represented as a stochastic integral
\begin{align*}
    F = \mathbb{E}[F] + \int_0^T\int_{\mathbb{R}^3}\mathbb{E}[D_{t,x}F | \mathscr{F}_t]W(dt,dx), \quad \text{a.s.}.
\end{align*}
\end{Prop}

If $DF$ and $DG$ are random functions, we obtain by applying the Clark-Ocone formula and the isometry property of stochastic integral that
\begin{align}
\label{Poincare inequality}
    |\mathrm{Cov}(F,G)| \leqslant \int_0^T\int_{\mathbb{R}^6}\lVert D_{t,x}F \rVert_2 \lVert D_{t,y}G \rVert_2 \gamma (x-y)dxdydt.
\end{align}
This inequality is usually called the Poincar\'{e} inequality.

Let us now introduce some preliminary results of Malliavin-Stein's method.
Stein's method is a probabilistic technique to allow one to get some bounds for the distance between two probability measures. 
Combining Stein's method and the Malliavin calculus, we can obtain very useful estimate to prove quantitative central limit theorems (see \cite{nourdin_peccati_2012}).
The next proposition is needed in the proof of Theorem \ref{main result1}.
\begin{Prop}
\label{Proposition Wasserstein bound}
Let $F = \delta(v)$ for some $v \in \mathrm{Dom}(\delta)$. Suppose that $\mathbb{E}[F^2] = 1$ and $F \in \mathbb{D}^{1,2}$. Then we have
\begin{align*}
     d_{W}(F, \mathcal{N}(0,1)) &\leqslant \sqrt{\frac{2}{\pi}} \sqrt{\mathrm{Var} \langle DF,v \rangle_{\mathcal{H}_T}}.
\end{align*}
\end{Prop}
For the proof of this proposition, see \textit{e.g.} \cite[Proposition 2.2]{CLTforSHE} and \cite[Theorem 8.2.1]{nualart_nualart_2018}. See also \cite[Theorem 5.1.3]{nourdin_peccati_2012} for an analogous result.

In order to establish the functional CLT in Theorem \ref{main result2}, we also need the following multivariate counterpart of Proposition \ref{Proposition Wasserstein bound}, which is a version of \cite[Theorem 6.1.2]{nourdin_peccati_2012}. See \cite[Proposition 2.3]{CLTforSHE} for the proof.

\begin{Prop}
\label{Proposition multivariate stein bound}
Fix $m \geqslant 2$, and let $F = (F_1, \ldots , F_m)$ be a random vector such that for every $i = 1,2,\ldots m$, $F_i = \delta(v_i)$ for some $v_i \in \mathrm{Dom}(\delta)$ and $F_i \in \mathbb{D}^{1,2}$. Let $Z$ be an m-dimensional centered Gaussian vector with covariance matrix $(C_{i,j})_{1 \leqslant i,j \leqslant m}$.
Then, for any twice continuously differentiable function $h: \mathbb{R}^m \to \mathbb{R}$ with bounded second partial derivatives, we have
\begin{align*}
    |\mathbb{E}[h(F)] - \mathbb{E}[h(Z)]| \leqslant \frac{m}{2}\lVert \nabla^2 h \rVert _{\infty}\sqrt{\sum_{i,j=1}^{m}\mathbb{E}[(C_{i,j}-\langle DF_i, v_j \rangle_{\mathcal{H}_T})^2]},
\end{align*}
where 
\begin{align*}
    \lVert \nabla^2 h \rVert _{\infty} = \max_{1 \leqslant i,j \leqslant m} \sup_{x \in \mathbb{R}^m} \left| \frac{\partial^2h}{\partial x_i \partial x_j}(x) \right|.
\end{align*}
\end{Prop}

\subsection{Some other basic results}
\label{subsection some basic estimates}
Recall that the fundamental solution of the three-dimensional wave equation is denoted by $G$.
It follows from \eqref{fundamental solution} that 
\begin{align}
\label{G total measure}
    G(t, \mathbb{R}^3) \coloneqq \int_{\mathbb{R}^3}G(t,dx) = t, \quad t > 0.
\end{align}
For convenience, we set $G(t, dx) = 0$ for $t \leqslant 0$.

It is well-known (see \textit{e.g.} \cite[Chapter 4]{Mizohata}) that the Fourier transform of $G(t)$ is given by
\begin{align*}
    \mathcal{F}G(t)(\xi) = \frac{\sin(2\pi t|\xi|)}{2\pi |\xi|}, \quad t \geqslant 0.
\end{align*}
From this we obtain, for any $t,s \in [0,T]$,  
\begin{align}
    |\mathcal{F}G(t)(\xi)| \leqslant t \leqslant T \label{FG bound}
\end{align}
and
\begin{align}
    |\mathcal{F}G(t)(\xi)-\mathcal{F}G(s)(\xi)| 
    &= \frac{1}{\pi |\xi|}|\cos(\pi(t+s)|\xi|)\sin(\pi(t-s)|\xi|)| \nonumber\\
    &\leqslant \frac{1}{\pi |\xi|}|\sin(\pi(t-s)|\xi|)| \nonumber \\ 
    &= 2\left|\mathcal{F}G\left(\frac{|t-s|}{2}\right)(\xi)\right| \label{FG minus FG bound by FG} \\ 
    &\leqslant |t-s|. \label{FG minus FG bound}
\end{align}
Moreover, simple estimates show that for $t \in [0,T]$, 
\begin{align}
    |\mathcal{F} G(t)(\xi)|^2 &= \left|\frac{\sin{(2\pi t|\xi|)}}{2\pi |\xi|}\right|^2 
    \leqslant \frac{1}{4{\pi}^2|\xi|^2}\mathbf{1}_{\{|\xi|>1\}} + t^2\mathbf{1}_{\{|\xi|\leqslant 1\}} \nonumber\\
    &\leqslant \frac{1}{2{\pi}^2}\frac{1}{1+|\xi|^2}\mathbf{1}_{\{|\xi|>1\}} + \frac{2t^2}{1+|\xi|^2}\mathbf{1}_{\{|\xi|\leqslant 1\}} \nonumber\\
    &\leqslant (1+2T^2)\langle \xi \rangle^{-2}. \label{fourier G bound}
\end{align}
Consequently, it follows from Dalang's condition \eqref{Dalang's condition} that  
\begin{align}
\label{muFG bound}
    \int_{\mathbb{R}^3}\mu(d\xi)|\mathcal{F} G(t)(\xi)|^2 \leqslant (1+2T^2)\int_{\mathbb{R}^3} \langle \xi \rangle^{-2} \mu(d\xi) < \infty.
\end{align}

We next introduce the regularization of $G(t)$.
Set 
\begin{align*}
    \rho(x)\coloneqq 
    \begin{cases}
        c\exp{(-\frac{1}{1-|x|^2})} &(|x|<1),\\
        0 &(|x| \geqslant 1),
    \end{cases}
\end{align*}
where $c>0$ is a normalization constant such that $\int_{\mathbb{R}^3} \rho(x)dx = 1$.
Let $\{a_n\}_{n=1}^{\infty}$ be a fixed monotone increasing positive sequence such that $\sum_{n = 1}^{\infty}\frac{1}{a_n} =1$. For $n \geqslant 1$, we define 
$\rho_n(x) = (a_n)^3\rho(a_nx)$ for all $x \in \mathbb{R}^3$ and 
\begin{align*}
    G_n(t,x) = \int_{\mathbb{R}^3}\rho_n(x-y)G(t,dy).
\end{align*}
Here are some elementary properties of $G_n$.
\begin{Lem}
\label{G_n property}
We have $G_n(t,\cdot) \in C_0^{\infty}(\mathbb{R}^3) \subset  \mathcal{S}(\mathbb{R}^3)$ for all $t \in [0,T]$ and $G_n(\cdot,x) \in C^{\infty}((0,T))\cap C([0,T])$ for all $x \in \mathbb{R}^3$. In particular, $[0,T] \times \mathbb{R}^3 \ni (t,x) \mapsto G_n(t,x) \in \mathbb{R}$ is continuous.  Moreover, $G_n$ satisfies the following properties:
\begin{enumerate}[\normalfont(i)]
    \item $\supp G_n(t,\cdot) \subset B_{t+\frac{1}{a_n}}$.
    \item $\sup_{(t,x) \in [0,T]\times\mathbb{R}^3}|G_n(t,x)| \leqslant T\lVert \rho_n\rVert_{\infty}$. 
    \item $|\mathcal{F} G_n(t)(\xi)| = |\mathcal{F}\rho_n(\xi)||\mathcal{F} G(t)(\xi)| \leqslant |\mathcal{F} G(t)(\xi)|$.
\end{enumerate}
\end{Lem}
\begin{proof}
The three properties easily follow from the definition of $G_n$.
Because
\begin{align*}
    G_n(t,x) = \frac{t}{4\pi}\int_{\partial B_1}\rho_n(x-ty)\sigma_1(dy) \qquad t>0,
\end{align*}
it is also easy to check that the rest of statements hold true (see, \textit{e.g.} Lemma 3.4 of Chapter 6 in \cite{SteinShakarchi}).
\end{proof}

Lemma \ref{G_n property} together with \eqref{muFG bound} implies the following uniform bound:
\begin{align}
\label{muFG_n uniform bound}
    \sup_{n} \int_{\mathbb{R}^3}G_n(t,x)G_n(t,y)\gamma(x-y)dxdy 
    \leqslant (1+2T^2)\int_{\mathbb{R}^3}\langle \xi \rangle^{-2} \mu(d\xi) < \infty.
\end{align}

Following the notation of \cite{Averaging2dSWE, CLTforSWEindimension1and2}, we define 
\begin{align*}
    \varphi_{t,R}(s,x) = \int_{\mathbb{R}^3}\mathbf{1}_{B_R}(x-y)G(t-s,dy), \quad
    \varphi_{n,t,R}(s,x) = \int_{\mathbb{R}^3}\mathbf{1}_{B_R}(x-y)G_n(t-s,y)dy.
\end{align*}

\begin{Lem}
\label{varphi property}
For any $s \leqslant t \leqslant T$, we have $\varphi_{n,t,R}(s,\cdot) \in C_0^{\infty}(\mathbb{R}^3)$, $\varphi_{t,R}(s,\cdot) \in L^1(\mathbb{R}^3) \cap L^{\infty}(\mathbb{R}^3)$ and $\varphi_{n,t,R}(s,x) \leqslant \varphi_{t,R+1}(s,x)$ for all $x \in \mathbb{R}^3$.
\end{Lem}
\begin{proof}
The fact $G_n(t-s, \cdot) \in C_0^{\infty}(\mathbb{R}^3)$ implies $\varphi_{n,t,R}(s,\cdot) \in C_0^{\infty}(\mathbb{R}^3)$, and 
it is easy to check that 
\begin{align*}
    \lVert \varphi_{t,R}(s, \cdot)\rVert_{L^1(\mathbb{R}^3)} \leqslant T|B_1|R^3 \quad \text{and} \quad
    \lVert \varphi_{t,R}(s, \cdot)\rVert_{L^{\infty}(\mathbb{R}^3)} \leqslant T.
\end{align*}
The last assertion follows from the Fubini-Tonelli theorem and $(\mathbf{1}_{B_R}*\rho_n)(x)\leqslant \mathbf{1}_{B_{R+1}}(x)$.
This completes the proof.
\end{proof}

Recall that since $\gamma = \mathcal{F}\mu$ in $\mathcal{S}^{\prime}(\mathbb{R}^3)$, we have
\begin{align*}
    \int_{\mathbb{R}^3}(\varphi*\tilde{\psi})(x)\gamma(x)dx =
    \int_{\mathbb{R}^6}\varphi(x)\psi(y)\gamma(x-y)dxdy = \int_{\mathbb{R}^3}\mathcal{F}\varphi(\xi)\overline{\mathcal{F}\psi(\xi)}\mu(d\xi)
\end{align*}
for any $\varphi, \psi \in \mathcal{S}(\mathbb{R}^3)$.
However, in some cases we want to use the Fourier transform for functions which do not belong to $\mathcal{S}(\mathbb{R}^3)$. 
For this purpose, the next lemma is often used in the sections below.

\begin{Lem}
\label{Lemma fourier transform}
Let $f,g$ be bounded measurable functions with compact supports on $\mathbb{R}^3$. Then, $f,g \in \mathcal{H}$ and it holds that 
\begin{align}
\label{alpha3}
    \langle f,g \rangle_{\mathcal{H}} = \int_{\mathbb{R}^6}f(x)g(y)\gamma(x-y)dxdy = \int_{\mathbb{R}^3}\mathcal{F}f(\xi)\overline{\mathcal{F}g(\xi)}\mu(d\xi).
\end{align}
\end{Lem}
\begin{proof}
Fix $\alpha \in C_0^{\infty}(\mathbb{R}^3)$ such that $\alpha \geqslant 0$, $\supp \alpha \subset B_1$, and $\int_{\mathbb{R}^3}\alpha(x)dx=1$. Set for all $n \geqslant 1$, $\alpha_n(x)=n^3\alpha(nx)$ for all $x\in \mathbb{R}^3$.
We define $f_n = f*\alpha_n$ and $g_n = g*\alpha_n$.
Then $f_n$ and $g_n$ belong to $C_0^{\infty}(\mathbb{R}^3)$ and thus to $\mathcal{H}$. Hence 
\begin{align}
\label{alpha2}
    \langle f_n, g_n \rangle_{\mathcal{H}} = \int_{\mathbb{R}^6}f_n(x)g_n(y)\gamma(x-y)dxdy &= \int_{\mathbb{R}^3}\mathcal{F}f_n(\xi)\overline{\mathcal{F}g_n(\xi)}\mu(d\xi).
\end{align}
Furthermore, by the construction of $f_n$, there is a compact set $K \subset \mathbb{R}^3$ such that $\supp f_n \subset K$ for all $n\geqslant1$.
Then, it follows from Fatou's lemma that
\begin{align}
    \int_{\mathbb{R}^3}|\mathcal{F}f(\xi)|^2\mu(d\xi) &\leqslant \liminf_{n\to \infty}\int_{\mathbb{R}^3}|\mathcal{F}f_n(\xi)|^2\mu(d\xi)
    = \liminf_{n\to \infty}\int_{\mathbb{R}^6}f_n(x)f_n(y)\gamma(x-y)dxdy\nonumber\\
    &\leqslant \lVert f \rVert_{L^{\infty}(\mathbb{R}^3)}^2\int_{K \times K}\gamma(x-y)dxdy < \infty, \label{alpha1}
\end{align}
where the last inequality follows because $\gamma(x)dx$ is tempered measure.
Because \eqref{alpha1} holds and $|1 - \mathcal{F}\alpha_n(\xi)|$ is bounded and converges to 0 as $n\to \infty$, we conclude from the Lebesgue dominated convergence theorem that
\begin{align*}
    \lim_{n \to \infty}\int_{\mathbb{R}^3}\mu(d\xi)|\mathcal{F}f(\xi) - \mathcal{F}f_n(\xi)|^2
    =\lim_{n \to \infty}\int_{\mathbb{R}^3}\mu(d\xi)|\mathcal{F}f(\xi)|^2|1 - \mathcal{F}\alpha_n(\xi)|^2 = 0.
\end{align*}
This implies that $f \in \mathcal{H}$. 
The same arguments also work for $g$. Finally, \eqref{alpha3} is checked by \eqref{alpha2} and a standard limiting argument.
\end{proof}

For any jointly measurable function $f(t,x,\omega)$, define 
\begin{align*}
    \lVert f \rVert_{+}^2 &= \mathbb{E}\left[\int_0^T\int_{\mathbb{R}^6}|f(t,x)||f(t,y)|\gamma(x-y)dxdydt \right],\\
    \lVert f \rVert_{0}^2 &= \mathbb{E}\left[\int_0^T\int_{\mathbb{R}^6}f(t,x)f(t,y)\gamma(x-y)dxdydt \right],
\end{align*}
and let $\mathcal{J}$ denote the set of all jointly measurable function $f$ such that $\lVert f \rVert_{+} < \infty$. 
Here we do not require the predictability of $f$.
We identify two functions $f,g \in \mathcal{J}$ if $\lVert f-g \rVert_{0} = 0$. Then $\lVert \cdot \rVert_{0}$ indeed defines the norm on $\mathcal{J}$.

\begin{Lem}
\label{Lemma L2(Omega,H_T)}
There is a linear isometry $\iota : \mathcal{J} \to L^2(\Omega; \mathcal{H}_T)$ such that for any $f,g \in \mathcal{J}$, 
\begin{align*}
    \langle \iota f, \iota g \rangle_{\mathcal{H}_T} =  \int_0^T\int_{\mathbb{R}^6}f(t,x)g(t,y)\gamma(x-y)dxdydt, \quad \text{a.s.}.
\end{align*}
\end{Lem}
\begin{proof}
For any $f \in \mathcal{J}$ and integer $n \geqslant 1$, define $f_n(t,x,\omega) = f(t,x,\omega)\mathbf{1}_{\{|f| \leqslant n\}}(t,x,\omega)\mathbf{1}_{B_n}(x)$. 
Then $f_n$ is jointly measurable (with respect to $\mathcal{B}([0,T] \times \mathbb{R}^3) \times \mathscr{F}$), and for every $t \in [0,T]$ and $\omega \in \Omega$, $f_n(t,\cdot, \omega) \in \mathcal{H}$ by Lemma \ref{Lemma fourier transform}.
Because $\mathcal{S}(\mathbb{R}^3)$ is dense in $\mathcal{H}$, and for every $\varphi \in \mathcal{S}(\mathbb{R}^3)$,
\begin{align*}
    \langle f_n(t, \cdot, \omega), \varphi \rangle_{\mathcal{H}} = \int_{\mathbb{R}^6}f_n(t,x,\omega)\varphi(y)\gamma(x-y)dxdy
\end{align*}
is measurable with respect to $\mathcal{B}([0,T]) \times \mathscr{F}$, the function $f_n : [0,T] \times \Omega \to \mathcal{H}$ is measurable by the Pettis measurability theorem (\textit{cf.} \cite[Theorem 1.1.6]{MR3617205}).
Moreover, $|f_n| \leqslant |f|$ implies 
\begin{align*}
    \lVert f_n \rVert_{L^2([0,T] \times \Omega; \mathcal{H})} = \lVert f_n \rVert_{0} \leqslant \lVert f \rVert_{+} < \infty,
\end{align*}
and thus $f_n \in L^2([0,T] \times \Omega; \mathcal{H})$. 
Since $\lim_{n \to \infty} f_n(t,x,\omega) = f(t,x,\omega)$ for every $(t,x,\omega)$, we deduce from the Lebesgue dominated convergence theorem that $\lim _{n \to \infty}\lVert f_n - f \rVert_{+} = 0$, and it follows that $\{f_n\}$ is Cauchy sequence in $L^2([0,T] \times \Omega; \mathcal{H})$.  
For simplicity of notation, we identify two spaces $L^2([0,T] \times \Omega; \mathcal{H}) \simeq L^2(\Omega;\mathcal{H}_T)$, and a version of $f_n$ which belongs to $L^2(\Omega;\mathcal{H}_T)$ is again denoted by $f_n$.
With this identification, $\{f_n\}$ is Cauchy sequence in $L^2(\Omega;\mathcal{H}_T)$ and there exists $\tilde{f} \in L^2(\Omega;\mathcal{H}_T)$ such that 
\begin{align*}
    \lim_{n \to \infty}\lVert f_n - \tilde{f} \rVert_{L^2(\Omega;\mathcal{H}_T)} = 0.
\end{align*}
Let us define a map $\iota : \mathcal{J} \to L^2(\Omega; \mathcal{H}_T)$ by $\iota f = \tilde{f}$.
Then it is easy to check that $\iota$ is the linear isometry. 
For any $f,g \in \mathcal{J}$, applying Lemma \ref{Lemma fourier transform} and Cauchy–Schwarz inequality, we see that 
\begin{align*}
    &\mathbb{E}\left[ \int_0^T\int_{\mathbb{R}^6}|f(t,x,\omega)||g(t,y,\omega)|\gamma(x-y)dxdydt \right]\\
    &\leqslant \liminf_{n \to \infty}\mathbb{E}\left[ \int_0^T\langle |f_n(t,\cdot,\omega)|, |g_n(t,\cdot,\omega)|\rangle_{\mathcal{H}} dt \right]\\
    &\leqslant \liminf_{n \to \infty}\left(\mathbb{E} \int_0^T\lVert |f_n(t,\cdot,\omega)| \rVert_{\mathcal{H}}^2 dt \right)^{\frac{1}{2}}\left(\mathbb{E} \int_0^T\lVert |g_n(t,\cdot,\omega)| \rVert_{\mathcal{H}}^2 dt \right)^{\frac{1}{2}}\\
    &\leqslant \lVert f \rVert_{+}\lVert g \rVert_{+} < \infty,
\end{align*}
where the third inequality follows from $|f_n| \leqslant |f|$ and $|g_n| \leqslant |g|$.
Therefore, we obtain from the Lebesgue dominated convergence theorem that for almost surely
\begin{align*}
    \langle \iota f, \iota g \rangle_{\mathcal{H}_T} 
    &= \lim_{n \to \infty} \langle f_n, g_n \rangle_{\mathcal{H}_T}\\
    &= \int_0^T\int_{\mathbb{R}^6}f(t,x,\omega)g(t,y,\omega)\gamma(x-y)dxdydt.
\end{align*}
This completes the proof.
\end{proof}
From now on, in order to simplify the notation, we identify $\iota f$ with $f$. With this identification, we will write $\mathcal{J} \subset L^2(\Omega;\mathcal{H}_T)$.

\section{The Picard Approximation and its Malliavin derivative}
\label{The Picard Approximation and its Malliavin derivative}
In this section, we consider the Picard approximation sequence defined in Section 1. 
We first collect some properties of the sequence in Section \ref{section The Picard Approximation}, and then prove the moment estimates for its Malliavin derivative in Section \ref{section The Malliavin derivative of the Picard approxximation}.
\subsection{The Picard approximation}
\label{section The Picard Approximation}
Recall that we set the Picard iteration scheme in \eqref{Picard iteration at introduction} as follows:
\begin{align}
    u_0(t,x) &= 1, \nonumber\\
    u_{n+1}(t,x) &= 1 + \int_0^t\int_{\mathbb{R}^3} G_{n+1}(t-s,x-y)\sigma(u_n(s,y))W(ds,dy). \label{u_n+1}
\end{align}

\begin{Prop}
\label{Proposition u_n property}
For every integer $n \geqslant 0$, we have the following properties:
\begin{enumerate}[\normalfont(i)]
    \item $u_n(t,x)$ has a predictable modification.
    \item $\sup_{(t,x) \in [0,T]\times\mathbb{R}^3}\lVert u_n(t,x)\rVert_2 < \infty$.
    \item $u_n(t,x)$ is strictly stationary in space variable: The finite-dimensional distributions of the process $\{u_n(t,x+y) \mid x \in \mathbb{R}^3 \}$ are independent of $y \in \mathbb{R}^3$. 
    \item $(t,x) \mapsto u_n(t,x)$ is $L^2(\Omega)$-continuous.
    \item For every $t \in [0,T]$, 
    \begin{align*}
        \lim_{h \to 0}\sup_{x \in \mathbb{R}^3} \lVert u_n(t+h,x) - u_n(t,x)\rVert_2 = 0.
    \end{align*}
\end{enumerate}
\end{Prop}

\begin{proof}
It is clear that $u_0(t,x) = 1$ satisfies all five properties above. Assume by induction that $u_n(t,x)$ satisfies the first three properties in the proposition: (i), (ii), and (iii).
The Lipschitz continuity of $\sigma$ and the property (ii) for $u_n(t,x)$ yield 
\begin{align}
\label{sigmau_n bound}
    \sup_{(t,x) \in [0,T]\times\mathbb{R}^3}\lVert \sigma(u_n(t,x))\rVert_2 < \infty.
\end{align}
By taking a predictable modification of $u_n(t,x)$, we see from \eqref{sigmau_n bound} and \eqref{muFG_n uniform bound} that $G_{n+1}(t-s,x-y)\sigma(u_n(s,y))$ is also predictable and the stochastic integral on the right-hand side of \eqref{u_n+1} is well-defined.
Hence $u_{n+1}(t,x)$ is also well-defined and $(\mathscr{F}_t)$-adapted.
Moreover, from \eqref{sigmau_n bound} we have 
\begin{align*}
    &\lVert u_{n+1}(t,x) \rVert_2^2\\
    &\leqslant 2 + 2\int_0^t\int_{\mathbb{R}^6}G_{n+1}(t-s,x-y)G_{n+1}(t-s,x-z)\gamma(y-z)\mathbb{E}[\sigma(u_n(s,y))\sigma(u_n(s,z))]dydzds\\
    &\leqslant 2 + 2\sup_{(t,x) \in [0,T]\times\mathbb{R}^3}\lVert \sigma(u_n(t,x))\rVert_2^2\int_0^t\int_{\mathbb{R}^6}G_{n+1}(t-s,x-y)G_{n+1}(t-s,x-z)\gamma(y-z)dydzds\\
    &\leqslant 2+2T(1+2T^2)\sup_{(t,x) \in [0,T]\times\mathbb{R}^3}\lVert \sigma(u_n(t,x))\rVert_2^2\int_{\mathbb{R}^3}\langle \xi \rangle^{-2}\mu(d\xi),
\end{align*}
where the third inequality follows from \eqref{muFG_n uniform bound}.
From this inequality, it follows that $u_{n+1}(t,x)$ satisfies the property (ii).

Strict stationarity of $u_{n+1}(t,x)$ follows from the same arguments in \cite[Lemma 18]{Dal99}. See also \cite[Lemma 7.1]{spatialergodicityforSPDEsvia} for similar arguments. 

Now observe that the Fourier transform of $G_{n+1}$ satisfies
\begin{gather*}
    \mathcal{F} G_{n+1}(t)(\xi) = \mathcal{F} G(t)(\xi) \mathcal{F} \rho_{n+1}(\xi),\\
    |\mathcal{F} G_{n+1}(t+h)(\xi) - \mathcal{F} G_{n+1}(t)(\xi)| \leqslant |\mathcal{F} G(t+h)(\xi) - \mathcal{F} G(t)(\xi)|,
\end{gather*}
for all $\xi \in \mathbb{R}^3$, $t \in [0,T]$, and $h > 0$.
Because $G$ satisfies the Hypothesis C in \cite{Dalang1999Corrections}, $G_{n+1}$ also satisfies it. Owing to the proof of Lemma 19 in \cite{Dalang1999Corrections}, we obtain that $x \mapsto u_{n+1}(t,x)$ is $L^2(\Omega)$-continuous for fixed $t \in [0,T]$, and $t \mapsto u_{n+1}(t,x)$ is $L^2(\Omega)$-equicontinuous for $x \in \mathbb{R}^3$. 
From these it follows that $(t,x) \mapsto u_{n+1}(t,x)$ is $L^2(\Omega)$-continuous. Hence $u_{n+1}(t,x)$ satisfies the properties (iv) and (v). 

Moreover, adaptedness and $L^2(\Omega)$-continuity imply that $u_{n+1}(t,x)$ has a predictable modification.
Therefore, $u_{n+1}(t,x)$ satisfies all five properties in Proposition \ref{Proposition u_n property} and the proof is completed.
\end{proof}

The properties (ii) and (v) in Proposition \ref{Proposition u_n property} also hold uniformly in $n$.

\begin{Prop}
\label{Proposition u_n uniform property}
For all $p \geqslant 1$, we have
\begin{align}
    \sup_{n}\sup_{(t,x) \in [0,T]\times\mathbb{R}^3}\lVert u_n(t,x)\rVert_p < \infty, \label{iteration uniform bound}\\
    \lim_{h \to 0}\sup_n \sup_{x \in \mathbb{R}^3} \lVert u_n(t+h,x) - u_n(t,x)\rVert_p = 0. \label{iteration uniform equicontinuous}
\end{align}
\end{Prop}
\begin{proof}
We first prove \eqref{iteration uniform bound}. Let $p \geqslant 2$. By the BDG inequality (Lemma \ref{BDG inequality}) and Minkowski's inequality, we obtain
\begin{align*}
    &\lVert u_{n+1}(t,x)\rVert_p^2\\
    &\leqslant 2 + 8p\int_0^tdr\int_{\mathbb{R}^6}dzdz'G_{n+1}(t-r,x-z)G_{n+1}(t-r,x-z')\gamma(z-z')\lVert \sigma(u_{n}(r,z))\sigma(u_{n}(r,z'))\rVert_{\frac{p}{2}}\\
    &\leqslant 2 + 8p\int_0^tdr\int_{\mathbb{R}^6}dzdz'G_{n+1}(t-r,x-z)G_{n+1}(t-r,x-z')\gamma(z-z')\lVert \sigma(u_{n}(r,z))\rVert_p^2\\
    &\leqslant 2 + 8p\int_0^tdr\int_{\mathbb{R}^6}dzdz'G_{n+1}(t-r,x-z)G_{n+1}(t-r,x-z')\gamma(z-z')\left(2\sigma(0)^2 + 2L^2\lVert u_{n}(r,z)\rVert_p^2\right)\\
    &\leqslant 2 + 16p\sigma(0)^2\int_0^tdr\int_{\mathbb{R}^3}\mu(d\xi)|\mathcal{F} G_{n+1}(t-r)(\xi)|^2\\
    &\qquad + 16pL^2\int_0^tdr \sup_{\eta \in \mathbb{R}^3}\lVert u_{n}(r,\eta)\rVert_p^2 \int_{\mathbb{R}^3}\mu(d\xi)|\mathcal{F} G_{n+1}(t-r)(\xi)|^2,
\end{align*}
where the second inequality follows from an elementary inequality:
\begin{align*}
    \lVert \sigma(u_{n}(r,z))\sigma(u_{n}(r,z'))\rVert_{\frac{p}{2}} \leqslant \frac{1}{2}(\lVert\sigma(u_{n}(r,z)) \rVert_p^2 + \lVert \sigma(u_{n}(r,z'))\rVert_p^2 ).
\end{align*}
It follows from \eqref{muFG_n uniform bound} that
\begin{align}
\label{H_n inequality}
    H_{n+1}(t) \leqslant c_1 + c_2\int_0^t H_n(r)dr,
\end{align}
where 
\begin{gather*}
    H_n(t) = \sup_{(\theta, \eta) \in [0,t] \times \mathbb{R}^3}\lVert u_{n}(\theta, \eta)\rVert_p^2,\\
    c_1 = 2 + 16p\sigma(0)^2T(1+2T^2)\int_{\mathbb{R}^3}\langle \xi \rangle^{-2}\mu(d\xi), \quad  c_2 = 16pL^2(1+2T^2)\int_{\mathbb{R}^3}\langle \xi \rangle^{-2} \mu(d\xi).
\end{gather*}
Then, iterating the inequality \eqref{H_n inequality}, we have
\begin{align*}
    H_n(t) \leqslant c_1e^{c_2t} \leqslant c_1e^{c_2T},
\end{align*}
which proves \eqref{iteration uniform bound}.

Next we prove \eqref{iteration uniform equicontinuous}. 
Let $t \in [0,T]$, $x \in \mathbb{R}^3$, and $h \in \mathbb{R}$. 
Since $\sigma(u_{n}(t,x))$ is strictly stationary by Proposition \ref{Proposition u_n property} and satisfies 
\begin{align}
    \sup_{n}\sup_{(t,x) \in [0,T]\times\mathbb{R}^3}\lVert \sigma(u_n(t,x))\rVert_2^2 \leqslant (2\sigma(0)^2 + 2L^2\sup_{n}\sup_{(t,x) \in [0,T]\times\mathbb{R}^3}\lVert u_n(t,x)\rVert_2^2)  < \infty \label{sigma u_n uniform bound}
\end{align}
by \eqref{iteration uniform bound}, there is a nonnegative tempered measure $\mu_t^{\sigma(u_n)}$ such that for each $t$,
\begin{align*}
    \mathbb{E}[\sigma(u_n(t,\cdot))\sigma(u_n(t,0))]\gamma(\cdot) = \mathcal{F}\mu_t^{\sigma(u_n)}
\end{align*}
in $\mathcal{S}^{\prime}(\mathbb{R}^3)$.
From \eqref{u_n+1} and (iii) in Lemma \ref{G_n property}, we have 
\begin{align*}
    &\lVert u_{n+1}(t+h,x)-u_{n+1}(t,x)\rVert_2^2\\
    &= \mathbb{E}\left[\left|\int_0^T\int_{\mathbb{R}^3}(G_{n+1}(t+h-s,x-y) - G_{n+1}(t-s,x-y) )\sigma(u_n(s,y))W(ds,dy) \right|^2\right]\\
    &= \int_0^Tds\int_{\mathbb{R}^3}\mu_s^{\sigma(u_n)}(d\xi)\left|e^{-2\pi\sqrt{-1}\xi\cdot x }\right|^2|\mathcal{F}G_{n+1}(t+h-s)(\xi)-\mathcal{F}G_{n+1}(t-s)(\xi)|^2\\
    &\leqslant \int_0^{(t+h) \land t}ds\int_{\mathbb{R}^3}\mu_s^{\sigma(u_n)}(d\xi)|\mathcal{F}G(t+h-s)(\xi)-\mathcal{F}G(t-s)(\xi)|^2\\
    &\quad + \int_{(t+h)\land t}^{(t+h)\lor t}ds\int_{\mathbb{R}^3}\mu_s^{\sigma(u_n)}(d\xi)|\mathcal{F}G((t+h-s) \lor (t-s))(\xi)|^2,
\end{align*}
and the integrals on the right-hand side of the inequality do not depend on $x$.
We deduce from \eqref{isometry fourier bound}, \eqref{sigma u_n uniform bound}, and \eqref{muFG bound} that the second integral has limit zero as $h \to 0$ uniformly in $n$.
Moreover, using \eqref{FG minus FG bound by FG}, \eqref{isometry fourier bound}, and \eqref{sigma u_n uniform bound}, we have
\begin{align*}
    &\int_0^{(t+h)\land t}ds\int_{\mathbb{R}^3}\mu_s^{\sigma(u_n)}(d\xi)|\mathcal{F}G(t+h-s)(\xi)-\mathcal{F}G(t-s)(\xi)|^2\\
    &\leqslant 2\int_0^{(t+h)\land t}ds\int_{\mathbb{R}^3}\mu_s^{\sigma(u_n)}(d\xi)\left|\mathcal{F}G\left( \frac{|h|}{2} \right)(\xi)\right|^2\\
    &\leqslant 2T\sup_{n}\sup_{(t,x) \in [0,T]\times\mathbb{R}^3}\lVert \sigma(u_n(t,x))\rVert_2^2 \int_{\mathbb{R}^3}\mu(d\xi)\left|\mathcal{F}G\left( \frac{|h|}{2} \right)(\xi)\right|^2.
\end{align*}
Because $\lim_{h \to 0}|\mathcal{F}G(|h|)(\xi)| = 0$ and \eqref{fourier G bound} and \eqref{muFG bound} hold, we conclude from the Lebesgue dominated convergence theorem that
\begin{align*}
    \lim_{h \to 0} \sup_{n}\int_0^{(t+h) \land t}ds\int_{\mathbb{R}^3}\mu_s^{\sigma(u_n)}(d\xi)|\mathcal{F}G(t+h-s)(\xi)-\mathcal{F}G(t-s)(\xi)|^2=0.
\end{align*}
Therefore, \eqref{iteration uniform equicontinuous} holds for $p=2$. Using H\"{o}lder's inequality and \eqref{iteration uniform bound}, we can deduce \eqref{iteration uniform equicontinuous} for all $p \geqslant 1$ from the case $p=2$. This completes the proof.
\end{proof}

Finally, we show that the sequence $u_n(t,x)$ converges to the solution $u(t,x)$ of \eqref{SPDE} in $L^p(\Omega)$.

\begin{Prop}
For all $p \geqslant 1$, we have
\begin{align}
    \lim_{n \to \infty}\sup_{(t,x) \in [0,T]\times\mathbb{R}^3} \lVert u_n(t,x)-u(t,x) \rVert_p = 0. \label{iteration uniform convergence}
\end{align}
\end{Prop}

\begin{proof}
From \eqref{u_n+1}, we have 
\begin{align*}
    \lVert u_{n+1}(t,x)-u(t,x)\rVert_2^2
    &\leqslant 2 \left \lVert \int_0^t\int_{\mathbb{R}^3}G_{n+1}(t-r,x-z)(\sigma(u_n(r,z)) - \sigma(u(r,z)))W(dr,dz) \right \rVert _2^2\\
    & \quad + 2 \left \lVert \int_0^t\int_{\mathbb{R}^3}(G_{n+1}(t-r,x-z)-G(t-r,x-z))\sigma(u(r,z))W(dr,dz) \right \rVert _2^2\\
    &\eqqcolon \mathbf{P_1} + \mathbf{P_2}.
\end{align*}
Using the isometry property of stochastic integral, \eqref{muFG_n uniform bound}, and Lipschitz continuity of $\sigma$, we obtain
\begin{align*}
    \mathbf{P_1} &\lesssim \int_0^tdr\int_{\mathbb{R}^6}dzdz'G_{n+1}(t-r,x-z)G_{n+1}(t-r,x-z')\gamma(z-z')\\
    &\qquad \times \lVert \sigma(u_n(r,z)) - \sigma(u(r,z')\rVert_2\lVert \sigma(u_n(r,z)) - \sigma(u(r,z')\rVert_2\\
    &\lesssim_T \int_0^tdr \sup_{(\theta, \eta) \in [0,r] \times \mathbb{R}^3} \lVert u_n(\theta,\eta) - u(\theta, \eta)\rVert_2^2.
\end{align*}
For the term $\mathbf{P_2}$, \eqref{isometry fourier} and \eqref{mu^sigma(u)} imply that
\begin{align*}
    \mathbf{P_2} &= 2\int_0^tdr\int_{\mathbb{R}^3}\mu_r^{\sigma(u)}(d\xi)|\mathcal{F} G_{n+1}(t-r)(\xi) - \mathcal{F} G(t-r)(\xi)|^2\\
    &= 2\int_0^tdr\int_{\mathbb{R}^3}\mu_r^{\sigma(u)}(d\xi)|\mathcal{F} \rho_{n+1}(\xi)-1|^2|\mathcal{F} G(t-r)(\xi)|^2.
\end{align*}
Note that $G_{n+1}(t) - G(t)$ is in general not a nonnegative distribution as in Proposition \ref{dalang prop}. However, both $G_{n+1}(t)$ and $G(t)$ are indeed nonnegative distributions with rapid decrease and the isometry property \eqref{isometry fourier} still holds. 
Because $|\mathcal{F} \rho_{n+1}(\xi)-1|^2$ converges pointwise to 0 as $n \to \infty$ and
\begin{gather*}
    |\mathcal{F} \rho_{n+1}(\xi)-1|^2|\mathcal{F} G(t-r)(\xi)|^2 \leqslant 4|\mathcal{F} G(t-r)(\xi)|^2,\\
    \int_0^t dr\int_{\mathbb{R}^3}\mu_r^{\sigma(u)}(d\xi)|\mathcal{F} G(t-r)(\xi)|^2
    \lesssim \int_0^tdr\int_{\mathbb{R}^3}\mu(d\xi)|\mathcal{F} G(t-r)(\xi)|^2 < \infty
\end{gather*}
by \eqref{solution uniform bound}, \eqref{isometry fourier bound}, and \eqref{muFG bound}, we conclude from the Lebesgue dominated convergence theorem that \\ $\lim_{n \to \infty} \mathbf{P_2} = 0$. 
Hence we obtain 
\begin{gather*}
    \sup_{(\theta, \eta) \in [0,t] \times \mathbb{R}^3} \lVert u_{n+1}(\theta,\eta) - u(\theta, \eta)\rVert_2^2 \leqslant C_{n+1} + C\int_0^tdr \sup_{(\theta, \eta) \in [0,r] \times \mathbb{R}^3} \lVert u_n(\theta,\eta) - u(\theta, \eta)\rVert_2^2,
\end{gather*}
where $C_n$ is a constant such that $\lim_{n \to \infty}C_n = 0$, and $C$ is a constant independent of $n$.
By iterating this inequality, we can derive \eqref{iteration uniform convergence} for the case $p=2$. 
Using H\"{o}lder's inequality, \eqref{solution uniform bound}, and \eqref{iteration uniform bound}, we can deduce \eqref{iteration uniform convergence} for all $p \geqslant 1$ from the case $p=2$.
\end{proof}

\subsection{The Malliavin derivative of the Picard approximation}
\label{section The Malliavin derivative of the Picard approxximation}
By a standard induction argument (see \textit{e.g.} \cite{millet}), we can show that for any $n \geqslant 0$, and for any $(t,x) \in [0,T] \times \mathbb{R}^3$, we have $u_n(t,x) \in \mathbb{D}^{1,p}$ for any $p \in [1,\infty)$ and $Du_n(t,x)$ satisfies 
\begin{align}
    Du_{n+1}(t,x) &= G_{n+1}(t-\cdot,x-\star)\sigma(u_n(\cdot,\star)) \nonumber\\ 
    &\quad + \int_0^t\int_{\mathbb{R}^3}G_{n+1}(t-r,x-z)\sigma^{\prime}(u_n(r,z))Du_n(r,z)W(dr,dz), \label{Du_n iteration}
\end{align}
where the stochastic integral on the right-hand side of \eqref{Du_n iteration} is the $\mathcal{H}_T$-valued stochastic integral (\textit{cf.} \cite[Section 2.6]{DalangQuel}).
For instance, we have $Du_0(t,x) = 0$ and $ Du_1(t,x) = G_1(t-\cdot, x- \star)\sigma(1)$ in $L^p(\Omega;\mathcal{H}_T)$.
Recall that $Du_n(t,x)$ is defined as the $\mathcal{H}_T$-valued random variable. 
Since $\mathcal{H}_T$ generally contains some distributions, it is not obvious at first glance that $Du_n(t,x)$ has a version $D_{s,y}u_n(t,x)$, which is a random function on $[0,T] \times \mathbb{R}^3$.
In order to prove the pointwise moment estimate \eqref{derivative estimate introduction} for $Du_n(t,x)$, we first have to show that, for any fixed $(t,x) \in [0,T] \times \mathbb{R}^3$ and $n$, there is a sufficiently nice version of $Du_n(t,x)$.

To do this, we introduce stochastic processes $\{M_n(t,x,s,y) \mid (t,x,s,y) \in ([0,T] \times \mathbb{R}^3)^2\}_{n\geqslant 1}$ as follows:
\begin{align}
    M_1(t,x,s,y) &= G_1(t-s, x-y)\sigma(1), \nonumber\\
    M_{n+1}(t,x,s,y) &= G_{n+1}(t-s,x-y)\sigma(u_n(s,y)) \nonumber\\
    &\quad + \int_0^t\int_{\mathbb{R}^3}G_{n+1}(t-r,x-z)\sigma^{\prime}(u_n(r,z))M_n(r,z,s,y)W(dr,dz), \label{M_(n+1)}
\end{align}
where we use the convention $G(t,dx) = 0$ for $t \leqslant 0$. Moreover, we introduce the following notations for convenience:
\begin{align*}
    \Lambda(T,p) &\coloneqq \sup_{n}\sup_{(t,x) \in [0,T]\times\mathbb{\mathbb{R}}^3}\lVert \sigma(u_n(t,x))\rVert_p, \qquad 
    \Theta(T,n) \coloneqq T\lVert\rho_n\rVert_{L^{\infty}(\mathbb{R}^3)},\\
    J_T^2 &\coloneqq (1+2T^2)\int_{\mathbb{R}^3}\langle \xi \rangle^{-2}\mu(d\xi).
\end{align*}
Clearly, $\Theta(T,n)< \infty$.
Lipschitz continuity of $\sigma$, \eqref{iteration uniform bound}, and \eqref{Dalang's condition} imply that $\Lambda(T,p) < \infty$ and $J_T^2 < \infty$.

\begin{Prop}
\label{Proposition M_n property}
For every integer $n \geqslant 1$, we have the following properties:
\begin{enumerate}[\normalfont(i)]
    \item $(t,x,s,y;\omega) \mapsto M_n(t,x,s,y;\omega)$ has a jointly measurable modification and $M_n(t,x,s,y)$ is $\mathscr{F}_t$-measurable  for any fixed $(t,x,s,y) \in ([0,T] \times \mathbb{R}^3)^2$.
    \item $\sup_{(t,x,s,y) \in ([0,T] \times \mathbb{R}^3)^2}\lVert M_n(t,x,s,y) \rVert_2 < \infty$.
    \item $M_n(t,x,s,y) = 0$ if $s \geqslant t$.
    \item $(t,x,s,y) \mapsto M_n(t,x,s,y)$ is $L^2(\Omega)$-continuous.
\end{enumerate}
\end{Prop}

\begin{proof}
The proof is by induction on $n$. Because $(t,x) \mapsto G_1(t,x)$ is continuous bounded function by Lemma \ref{G_n property} and $M_1(t,x,s,y)$ is deterministic, it is easy to check that $M_1(t,x,s,y)$ satisfies all four properties above.

Now we assume that $M_n(t,x,s,y)$ satisfies all properties in the proposition.
By Proposition \ref{Proposition u_n property} and the induction assumption,  $(t,x;\omega) \mapsto \sigma^{\prime}(u_n(t,x;\omega))M_n(t,x,s,y;\omega)$ is stochastically continuous and jointly measurable for any fixed $(s,y)$, and is $\mathscr{F}_t$-measurable for every $(t,x)$. Moreover, 
\begin{align*}
    \sup_{(t,x) \in [0,T] \times \mathbb{R}^3}\lVert \sigma^{\prime}(u_n(t,x))M_n(t,x,s,y) \rVert_2 \leqslant L \sup_{(\theta,\eta,\lambda,\kappa) \in ([0,T] \times \mathbb{R}^3)^2}\lVert M_n(\theta,\eta,\lambda,\kappa) \rVert_2 < \infty.
\end{align*}
Therefore, the stochastic integral in \eqref{M_(n+1)} is well-defined and $M_{n+1}(t,x,s,y)$ is $\mathscr{F}_t$-measurable for every $(t,x,s,y)$.
Then, we have from (ii) of Lemma \ref{G_n property} that
\begin{align*}
    &\lVert M_{n+1}(t,x,s,y) \rVert_2^2\\ 
    &\leqslant 2\lVert G_{n+1}(t-s, x-y)\sigma(u_n(s,y)) \rVert_2^2\\
    &\quad + 2 \left \lVert \int_0^t\int_{\mathbb{R}^3}G_{n+1}(t-r,x-z)\sigma^{\prime}(u_n(r,z))M_n(r,z,s,y)W(dr,dz) \right \rVert _2^2\\
    &\leqslant 2\Theta(T,n)^2\Lambda(T,2)^2\\
    &\quad + 2L^2\int_0^tdr\int_{\mathbb{R}^6}dzdz'G_{n+1}(t-r,x-z)G_{n+1}(t-r,x-z')\gamma(z-z')\\
    &\qquad \qquad \qquad \times \lVert M_n(r,z,s,y) \rVert_2\lVert M_n(r,z',s,y) \rVert_2\\
    &\leqslant 2\Theta(T,n)^2\Lambda(T,2)^2 + 2L^2T J_T^2\sup_{(\theta,\eta,\lambda,\kappa) \in ([0,T] \times \mathbb{R}^3)^2}\lVert M_n(\theta,\eta,\lambda,\kappa) \rVert_2^2,
\end{align*}
where the third inequality follows from \eqref{muFG_n uniform bound}. From this inequality, it follows that $M_{n+1}(t,x,s,y)$ satisfies the property (ii). The property (iii) for $M_{n+1}$ directly follows from \eqref{M_(n+1)}.

Next we show $L^2(\Omega)$-continuity. 
Because $G_{n+1}(t,x)$ is continuous in $(t,x)$ by Lemma \ref{G_n property} and $\sigma(u_n(s,y))$ is $L^2(\Omega)$-continuous in $(s,y)$ by Proposition \ref{Proposition u_n property}, we conclude that $(t,x,s,y) \mapsto G_{n+1}(t-s,x-y)\sigma(u_n(s,y))$ is $L^2(\Omega)$-continuous. 
To simplify notation, we write $N_{n+1}(t,x,s,y)$ instead of the stochastic integral in \eqref{M_(n+1)}.
Then, we have
\begin{align*}
    &\lVert N_{n+1}(t',x',s',y') - N_{n+1}(t,x,s,y) \rVert_2^2\\
    &\leqslant 2\left\lVert \int_0^T\int_{\mathbb{R}^3}( G_{n+1}(t'-r,x'-z) - G_{n+1}(t-r,x-z) )\sigma^{\prime}(u_n(r,z))M_n(r,z,s'y')W(dr,dz) \right\rVert_2^2\\
    &\quad + 2\left\lVert \int_0^T\int_{\mathbb{R}^3} G_{n+1}(t-r,x-z)\sigma^{\prime}(u_n(r,z))(M_n(r,z,s'y')-M_n(r,z,s,y))W(dr,dz) \right\rVert_2^2\\
    &\eqqcolon \mathbf{S_1} + \mathbf{S_2}.
\end{align*}
Using the isometry property, we have 
\begin{align*}
    \mathbf{S_1} 
    &\leqslant 2L^2\sup_{(\theta,\eta,\lambda,\kappa) \in ([0,T] \times \mathbb{R}^3)^2}\lVert M_n(\theta,\eta,\lambda,\kappa) \rVert_2^2\\
    &\quad \times \int_0^Tdr\int_{\mathbb{R}^6}dzdz'\gamma(z-z')|G_{n+1}(t'-r,x'-z)-G_{n+1}(t-r,x-z)|\\
    &\qquad \qquad \qquad \times |G_{n+1}(t'-r,x'-z')-G_{n+1}(t-r,x-z')|.
\end{align*}
We are going to prove that 
\begin{align}
\label{lim S_1}
    \lim_{(t',x',s',y') \to (t,x,s,y)}\mathbf{S_1} = 0,
\end{align}
and it is harmless to assume that $|x'-x| < 1$. Continuity and compact support property of $G_{n+1}$ imply that for every $(r,z) \in [0,T] \times \mathbb{R}^3$,
\begin{align*}
    \lim_{(t',x') \to (t,x)} |G_{n+1}(t'-r,x'-z) - G_{n+1}(t-r,x-z)| = 0,\\
    |G_{n+1}(t'-r,x'-z) - G_{n+1}(t-r,x-z)| \leqslant 2\Theta(T,n+1)\mathbf{1}_{B_{T+2}}(x-z).
\end{align*}
Because 
\begin{align*}
    \int_0^Tdr\int_{\mathbb{R}^6}dzdz'\gamma(z-z')\mathbf{1}_{B_{T+2}}(x-z)\mathbf{1}_{B_{T+2}}(x-z') < \infty,
\end{align*}
\eqref{lim S_1} follows from the Lebesgue dominated convergence theorem.
For the term $\mathbf{S_2}$, we have
\begin{align*}
    \mathbf{S_2} 
    &\leqslant 2L^2\int_0^T dr\int_{\mathbb{R}^6}dzdz'G_{n+1}(t-r,x-z)G_{n+1}(t-r,x-z')\gamma(z-z')\\
    &\qquad \qquad \times \lVert M_n(r,z,s',y') - M_n(r,z,s,y) \rVert_2\lVert M_n(r,z',s',y') - M_n(r,z',s,y) \rVert_2.
\end{align*}
Since, by the properties (ii) and (iv) for $M_n$, we have for every $(r,z) \in [0,T] \times \mathbb{R}^3$,
\begin{align*}
    \lim_{(s',y') \to (s,y)}\lVert M_n(r,z,s',y') - M_n(r,z,s,y) \rVert_2 = 0
\end{align*}
and
\begin{align*}
    &G_{n+1}(t-r, x-z)\lVert M_n(r,z,s',y') - M_n(r,z,s,y) \rVert_2\\
    &\leqslant 2\sup_{(\theta,\eta,\lambda,\kappa) \in ([0,T] \times \mathbb{R}^3)^2}\lVert M_n(\theta,\eta,\lambda,\kappa) \rVert_2\Theta(T,n+1)\mathbf{1}_{B_{T+1}}(x-z),
\end{align*}
we can apply the Lebesgue dominated convergence theorem again and obtain
\begin{align}
\label{lim S_2}
    \lim_{(t',x',s',y') \to (t,x,s,y)}\mathbf{S_2} = 0.
\end{align}
Combining \eqref{lim S_1} and \eqref{lim S_2}, we have
\begin{align*}
    \lim_{(t',x',s',y') \to (t,x,s,y)}\lVert N_{n+1}(t',x',s',y') - N_{n+1}(t,x,s,y) \rVert_2 = 0,
\end{align*}
and it follows that $(t,x,s,y) \mapsto M_{n+1}(t,x,s,y)$ is $L^2(\Omega)$-continuous. 

Finally, $L^2(\Omega)$-continuity implies that $M_{n+1}(t,x,s,y)$ has a jointly measurable modification. Therefore, $M_{n+1}$ satisfies all four properties in Proposition \ref{Proposition M_n property} and the proof is completed.
\end{proof}

Using the fact that $G_{n}$ has compact support on $[0,T] \times \mathbb{R}^3$, we can derive the moment estimate for $M_n(t,x,s,y)$.
Let us define $I_k$ by
\begin{align*}
    I_0(t,x,s,y,n) &= G_{n+1}(t-s,x-y)\sigma(u_n(s,y)),\\
    I_k(t,x,s,y,n) &= \int_s^t\int_{\mathbb{R}^3}G_{n+1}(t-r,x-z)\sigma^{\prime}(u_n(r,z))I_{k-1}(r,z,s,y,n-1)W(dr,dz), \qquad (k \geqslant 1).
\end{align*}
Then, it follows from \eqref{M_(n+1)} that
\begin{align*}
    M_{n+1}(t,x,s,y) = \sum_{k=0}^n I_k(t,x,s,y,n),
\end{align*}
and, by the triangle inequality, we have
\begin{align*}
    \lVert M_{n+1}(t,x,s,y)\rVert_p \leqslant \sum_{k=0}^n \lVert I_k(t,x,s,y,n)\rVert_p.
\end{align*}

Now we are in position to state the main result in this section.
\begin{Thm}
\label{Theorem moment estimate}
Let $p \geqslant 2$. Then, we have for any $n \geqslant 1$ and $(t,x,s,y) \in ([0,T] \times \mathbb{R}^3)^2$,
\begin{align*}
    \lVert M_{n}(t,x,s,y)\rVert_p &\leqslant C\Theta(T,n)\mathbf{1}_{B_{t-s+1}}(x-y)\\
    &\leqslant C\Theta(T,n)\mathbf{1}_{B_{T+1}}(x-y).
\end{align*}
Here, the above constant $C$ is given by
\begin{align*}
    C\coloneqq \Lambda(T,p)\sum_{k=0}^{\infty}\frac{(2\sqrt{pT}LJ_T)^k}{\sqrt{k!}} < \infty.
\end{align*}
\end{Thm}
\begin{proof}
When $k = 0$, it is clear that
\begin{align}
\label{k=0}
    \lVert I_0(t,x,s,y,n-1)\rVert_p &= G_{n}(t-s,x-y)\lVert \sigma(u_{n-1}(s,y))\rVert_p \nonumber\\
    &\leqslant \Lambda(T,p)\Theta(T,n)\mathbf{1}_{B_{t-s+\frac{1}{a_{n}}}}(x-y).
\end{align}
When $k=1$, using the BDG inequality, Minkowski's inequality, and the Cauchy–Schwarz inequality, we obtain
\begin{align*}
    \lVert I_1(t,x,s,y,n-1)\rVert_p^2 &\leqslant 4p \biggl \lVert \int_s^tdr\int_{\mathbb{R}^6}dzdz'G_{n}(t-r,x-z)G_{n}(t-r,x-z')\gamma(z-z')\\
    &\qquad \quad \;\; \times \sigma^{\prime}(u_{n-1}(r,z))I_{0}(r,z,s,y,n-2)\sigma^{\prime}(u_{n-1}(r,z'))I_{0}(r,z',s,y,n-2) \biggr \rVert_{\frac{p}{2}}\\
    &\leqslant 4pL^2 \int_s^tdr\int_{\mathbb{R}^6}dzdz'G_{n}(t-r,x-z)G_{n}(t-r,x-z')\gamma(z-z')\\
    &\qquad \qquad \times \lVert I_{0}(r,z,s,y,n-2) \rVert_p \lVert I_{0}(r,z',s,y,n-2) \rVert_{p}.
\end{align*}
Then applying \eqref{k=0}, (i) of Lemma \ref{G_n property}, and \eqref{muFG_n uniform bound}, we have
\begin{align*}
    &\lVert I_1(t,x,s,y,n-1)\rVert_p^2\\
    &\leqslant 4pL^2\Lambda(T,p)^2\Theta(T,n-1)^2\\
    &\quad \times \int_s^tdr\int_{\mathbb{R}^6}dzdz'G_{n}(t-r,x-z)G_{n}(t-r,x-z')\gamma(z-z')\mathbf{1}_{B_{r-s+\frac{1}{a_{n-1}}}}(z-y)\mathbf{1}_{B_{r-s+\frac{1}{a_{n-1}}}}(z'-y)\\
    &\leqslant 4pL^2\Lambda(T,p)^2\Theta(T,n-1)^2\mathbf{1}_{B_{t-s+\frac{1}{a_{n-1}}+\frac{1}{a_{n}}}}(x-y)\\
    &\qquad \times \int_s^tdr\int_{\mathbb{R}^6}G_{n}(t-r,x-z)G_{n}(t-r,x-z')\gamma(z-z')dzdz'\\
    &\leqslant 4pL^2\Lambda(T,p)^2\Theta(T,n-1)^2J_T^2(t-s)\mathbf{1}_{B_{t-s+\frac{1}{a_{n-1}}+\frac{1}{a_{n}}}}(x-y),
\end{align*}
and therefore we get
\begin{align*}
    \lVert I_1(t,x,s,y,n-1)\rVert_p \leqslant (2\sqrt{p}LJ_T)\Lambda(T,p) \Theta(T,n-1)(t-s)^{\frac{1}{2}}\mathbf{1}_{B_{t-s+\frac{1}{a_{n-1}}+\frac{1}{a_{n}}}}(x-y).
\end{align*}
Next we assume for $k \geqslant 1$ that 
\begin{align}
\label{indhyp}
    \lVert I_k(t,x,s,y,n-1)\rVert_p \leqslant (2\sqrt{p}LJ_T)^k\Lambda(T,p) \Theta(T,n-k)(t-s)^{\frac{k}{2}}(k!)^{-\frac{1}{2}}\mathbf{1}_{B_{t-s+\sum_{i=n-k}^{n}\frac{1}{a_i}}}(x-y)
\end{align}
for any $n \geqslant 1$ and $(t,x,s,y) \in ([0,T] \times \mathbb{R}^3)^2$.
Then we have
\begin{align*}
    &\lVert I_{k+1}(t,x,s,y,n-1)\rVert_p^2\\ 
    &\leqslant 4pL^2\int_s^tdr\int_{\mathbb{R}^6}dzdz'G_{n}(t-r,x-z)G_{n}(t-r,x-z')\gamma(z-z')\\
    &\qquad \qquad \times \lVert I_k(r,z,s,y,n-2)\rVert_p\lVert I_k(r,z,s,y,n-2)\rVert_{p}\\
    &\leqslant 4pL^2 (2\sqrt{p}LJ_T)^{2k}\Lambda(T,p)^2\Theta(T,n-k-1)^2(k!)^{-1}\\
    &\qquad \times \int_s^tdr(r-s)^k\int_{\mathbb{R}^6}dzdz'G_{n}(t-r,x-z)G_{n}(t-r,x-z')\gamma(z-z')\\
    &\qquad \qquad \qquad \qquad \qquad \times \mathbf{1}_{B_{r-s+\sum_{i=n-k-1}^{n-1}\frac{1}{a_i}}}(z-y)\mathbf{1}_{B_{r-s+\sum_{i=n-k-1}^{n-1}\frac{1}{a_i}}}(z'-y)\\
    &\leqslant 4pL^2(2\sqrt{p}LJ_T)^{2k}\Lambda(T,p)^2\Theta(T,n-k-1)^2(k!)^{-1}\mathbf{1}_{B_{t-s+\sum_{i=n-k-1}^{n}\frac{1}{a_i}}}(x-y)\\
    &\qquad \times \int_s^tdr(r-s)^k\int_{\mathbb{R}^6}dzdz'G_{n}(t-r,x-z)G_{n}(t-r,x-z')\gamma(z-z')\\
    &\leqslant (2\sqrt{p}LJ_T)^{2(k+1)}\Lambda(T,p)^2\Theta(T,n-k-1)^2((k+1)!)^{-1}(t-s)^{k+1}\mathbf{1}_{B_{t-s+\sum_{i=n-k-1}^{n}\frac{1}{a_i}}}(x-y).
\end{align*}
Hence we obtain
\begin{align*}
    &\lVert I_{k+1}(t,x,s,y,n-1)\rVert_p\\
    &\leqslant (2\sqrt{p}LJ_T)^{k+1}\Lambda(T,p) \Theta(T,n-k-1)(t-s)^{\frac{k+1}{2}}((k+1)!)^{-\frac{1}{2}}\mathbf{1}_{B_{t-s+\sum_{i=n-k-1}^{n}\frac{1}{a_i}}}(x-y),
\end{align*}
and \eqref{indhyp} holds for $k+1$.
Therefore, by induction, \eqref{indhyp} holds for all $k \geqslant 0$.
Since $\Theta(T,n) = T\lVert\rho_n\rVert_{\infty}$ is monotone increasing in $n$, 
we finally have
\begin{align*}
    \lVert M_{n}(t,x,s,y)\rVert_p &\leqslant \sum_{k=0}^{n-1} \lVert I_k(t,x,s,y,n-1)\rVert_p\\
    &\leqslant \mathbf{1}_{B_{t-s+\sum_{i=1}^{n}\frac{1}{a_i}}}(x-y)\Lambda(T,p)\sum_{k=0}^{n-1}\frac{(2\sqrt{pT}LJ_T)^k\Theta(T,n-k)}{\sqrt{k!}}\\
    &\leqslant C\Theta(T,n)\mathbf{1}_{B_{t-s+1}}(x-y),
\end{align*}
and the proof is completed.
\end{proof}

Finally, we show that $M_n(t,x, \cdot, \star)$ is indeed a version of $Du_n(t,x)$.
\begin{Prop}
\label{Proposition M_nDu_n}
Let $(t,x) \in [0,T] \times \mathbb{R}^3$ and $p \geqslant 1$. For every $n \geqslant 1$,  we have
\begin{align}
\label{M_nDu_n}
    M_n(t,x,\cdot, \star) = Du_n(t,x)
\end{align}
in $L^p(\Omega; \mathcal{H}_T)$.
\end{Prop}

\begin{proof}
Since $Du_n(t,x) \in L^p(\Omega; \mathcal{H}_T)$ for every $(t,x) \in [0,T] \times \mathbb{R}^3$ and $p \geqslant 1$, we only need to show that \eqref{M_nDu_n} holds when $p = 2$.
It is clear that 
\begin{align*}
    M_1(t,x,\cdot, \star) = G_1(t-\cdot, x-\star)\sigma(1) = Du_1(t,x).
\end{align*}
Assume by induction that \eqref{M_nDu_n} holds for $n$ when $p = 2$. 
We first observe that $M_{n+1}(t,x,\cdot,\star) \in L^2(\Omega;\mathcal{H}_T)$. Indeed, by Proposition \ref{Proposition M_n property} and Theorem \ref{Theorem moment estimate}, $(s,y;\omega) \mapsto M_{n+1}(t,x,s,y;\omega)$ is jointly measurable and 
\begin{align*}
    &\mathbb{E}\left[\int_0^T \int_{\mathbb{R}^6}|M_{n+1}(t,x,s,y)||M_{n+1}(t,x,s,z)|\gamma(y-z)dydzds\right]\\
    &\leqslant \int_0^T\int_{\mathbb{R}^6}\lVert M_{n+1}(t,x,s,y) \rVert_2 \lVert M_{n+1}(t,x,s,z) \rVert_2 \gamma(y-z)dydzds\\
    &\lesssim_{n} \int_0^T\int_{B_{T+1}^2}\gamma(y-z)dydzds < \infty.
\end{align*}
Therefore, $M_{n+1}(t,x,\cdot,\star) \in \mathcal{J} \subset L^2(\Omega;\mathcal{H}_T)$.

Let $\{f_k\}_k \subset C_0^{\infty}([0,T] \times \mathbb{R}^3)$ be a complete orthonormal system of $\mathcal{H}_T$ (\textit{cf.} Lemma \ref{dense lemma}). Then, the construction of the $\mathcal{H}_T$-valued stochastic integral implies that
\begin{align*}
    \langle Du_{n+1}(t,x), f_k \rangle_{\mathcal{H}_T} 
    &= \langle G_{n+1}(t-\cdot,x-\star)\sigma(u_n(\cdot,\star)), f_k \rangle_{\mathcal{H}_T}\\
    &\quad + \int_0^t\int_{\mathbb{R}^3}G_{n+1}(t-r,x-z)\sigma^{\prime}(u_n(r,z))\langle Du_n(r,z), f_k \rangle_{\mathcal{H}_T}W(dr,dz).
\end{align*}
In order to prove $M_{n+1}(t,x,\cdot, \star) = Du_{n+1}(t,x)$ in $L^2(\Omega; \mathcal{H}_T)$, it is sufficient to show that 
\begin{align}
\label{MnDn c1}
    \langle M_{n+1}(t,x,\cdot, \star), f_k \rangle_{\mathcal{H}_T} = \langle Du_{n+1}(t,x), f_k \rangle_{\mathcal{H}_T}, \quad \text{a.s.}, 
\end{align}
for every $k$. Applying Lemma \ref{Lemma L2(Omega,H_T)}, we have almost surely 
\begin{align*}
    &\langle M_{n+1}(t,x,\cdot, \star), f_k \rangle_{\mathcal{H}_T} 
    = \int_0^T\int_{\mathbb{R}^6}M_{n+1}(t,x,s,y)f_k(s,z)\gamma(y-z)dydzds\\
    &= \langle G_{n+1}(t-\cdot, x-\star)\sigma(u_n(\cdot, \star)), f_k \rangle_{\mathcal{H}_T}\\
    &\quad + \int_0^T\int_{\mathbb{R}^3}\left( \int_0^t\int_{\mathbb{R}^3}G_{n+1}(t-r,x-z)\sigma^{\prime}(u_n(r,z))M_n(r,z,s,y)W(dr,dz) \right)(f_k(s)*\gamma)(y)dyds\\
    &= \langle G_{n+1}(t-\cdot, x-\star)\sigma(u_n(\cdot, \star)), f_k \rangle_{\mathcal{H}_T}\\
    &\quad +\int_0^t\int_{\mathbb{R}^3}G_{n+1}(t-r,x-z)\sigma^{\prime}(u_n(r,z))\left(\int_0^T\int_{\mathbb{R}^3}M_n(r,z,s,y)(f_k(s)*\gamma)(y)dyds\right)W(dr,dz).
\end{align*}
Here, in the last line above, we use Fubini's theorem, \eqref{skorokhod integral}, and the duality formula \eqref{duality formula} to interchange the stochastic integral and the Lebesgue integral, which follows from the similar argument as in \cite[p.10]{Averaging2dSWE}.
Note that the signed measure $(f_k(s)*\gamma)(y)dyds$ on $[0,T] \times \mathbb{R}^3$ is in general not finite as in the setting in \cite[p.10]{Averaging2dSWE}. However, using Theorem \ref{Theorem moment estimate}, we can indeed apply Fubini's theorem.
By induction assumption and Lemma \ref{Lemma L2(Omega,H_T)}, we have for every $(t,x) \in [0,T] \times \mathbb{R}^3$,
\begin{align*}
    \langle Du_n(t,x), f_k \rangle_{\mathcal{H}_T} 
    &= \langle M_n(t,x,\cdot, \star), f_k \rangle_{\mathcal{H}_T}\\
    &= \int_0^T\int_{\mathbb{R}^3}M_n(r,z,s,y)(f_k(s)*\gamma)(y)dyds, \quad \text{a.s.}.
\end{align*}
Therefore, \eqref{MnDn c1} holds and the proof is completed.
\end{proof}

By Theorem \ref{Theorem moment estimate} and Proposition \ref{Proposition M_nDu_n}, we have proved the pointwise moment estimate \eqref{derivative estimate introduction} for $Du_n(t,x)$.

\section{Central Limit Theorems}
\label{CLT}
We prove Theorem \ref{main result1} in this section. 
Section \ref{section Nondegeneracy of the variance} is devoted to proving the nondegeneracy of the variance in Theorem \ref{main result1}. 
We show the rest of statement \eqref{distance convergence} in Section \ref{section Proof of Theorem}.
\subsection{Nondegeneracy of the variance}
\label{section Nondegeneracy of the variance}
Let us first recall some notations. We set in Section 1 that
\begin{align*}
    F_{n,R}(t) = \int_{B_R} (u_n(t,x) - 1)dx, \quad F_{R}(t) = \int_{B_R} (u(t,x) - 1)dx, 
\end{align*}
and 
\begin{align*}
    \sigma_{n,R}^2(t) = \mathrm{Var}(F_{n,R}(t)), \quad \sigma_{R}^2(t) = \mathrm{Var}(F_{R}(t)).
\end{align*}
Notice that $\mathbb{E}[F_{n,R}(t)] = \mathbb{E}[F_{R}(t)] = 0$.
Set
\begin{align*}
    V_{n,t,R}(s,y)= \varphi_{n,t,R}(s,y)\sigma(u_{n-1}(s,y)), \quad V_{t,R}(s,y)= \varphi_{t,R}(s,y)\sigma(u(s,y)).
\end{align*}
Then, we have
\begin{align}
     F_{n,R}(t) &= \int_0^t\int_{\mathbb{R}^3}\varphi_{n,t,R}(s,y)\sigma(u_{n-1}(s,y))W(ds,dy) = \delta(V_{n,t,R}),  \label{fn delta}\\
    F_{R}(t) &= \int_0^t\int_{\mathbb{R}^3}\varphi_{t,R}(s,y)\sigma(u(s,y))W(ds,dy) = \delta(V_{t,R}), \label{f delta}
\end{align}
almost surely, where $\delta$ is the adjoint operator of the Malliavin derivative operator $D$ (see Section \ref{subsection Malliavin calculus and Malliavin-Stein bound}).
Indeed, \eqref{fn delta} and \eqref{f delta} follow from the similar argument in p.10 of \cite{Averaging2dSWE}.

To prove the nondegeneracy, we first consider the case $\gamma \in L^1(\mathbb{R}^3)$.
\begin{Lem}
\label{Lemma L1 nondegeneracy}
Let $\gamma \in L^1(\mathbb{R}^3)$ such that $\gamma(x) > 0$ for all $x \in \mathbb{R}^3$. For $t \in (0,T]$, we have $\sigma_R^2(t) >0$, $\sigma_{n,R}^2(t)>0$ for any $R>0$ and $n \geqslant 1$. Moreover, if $R\geqslant 2(T+1)$, then we have 
\begin{gather}
\label{l1 nondegeneracy}
    \sigma^2_R(t) \gtrsim R^3, \quad \sigma^2_{n,R}(t) \gtrsim R^3 \quad ( n \geqslant 1).
\end{gather}
In particular, implicit constants in \eqref{l1 nondegeneracy} are independent of $n$ and $R$.
\end{Lem}
\begin{proof}
We first prove the claim for $\sigma_R^2(t)$.
By the isometry property of stochastic integral and Lemma \ref{Lemma L2(Omega,H_T)}, we have
\begin{align*}
    \sigma^2_R(t) &= \int_0^tdr\int_{\mathbb{R}^6}dzdz'\varphi_{t,R}(r,z)\varphi_{t,R}(r,z')\gamma(z-z')\mathbb{E}[\sigma(u(r,z))\sigma(u(r,z'))]\\
    &= \int_0^t\mathbb{E}[\lVert V_{t,R}(r,\cdot)\rVert_{\mathcal{H}}^2]dr.
\end{align*}
Now observe that $t \mapsto u(t,x)$ is $L^2(\Omega)$-equicontinuous for $x \in \mathbb{R}^3$ by \eqref{iteration uniform equicontinuous} and \eqref{iteration uniform convergence}.
Since $u(0,x) = 1$ for all $x \in \mathbb{R}^3$ and $\sigma$ is the Lipschitz continuous function, we deduce from this equicontinuity and \eqref{solution uniform bound} that 
\begin{align*}
    &\lim_{r \to 0} \sup_{z \in \mathbb{R}^3} \sup_{z' \in \mathbb{R}^3}|\mathbb{E}[\sigma(u(r,z))\sigma(u(r,z'))] - \sigma(1)^2|\\
    &= \lim_{r \to 0} \sup_{z \in \mathbb{R}^3} \sup_{z' \in \mathbb{R}^3}|\mathbb{E}[\sigma(u(r,z))\sigma(u(r,z'))] - \mathbb{E}[\sigma(u(0,z))\sigma(u(0,z'))]| = 0.
\end{align*}
This together with $\sigma(1) \neq 0$ (see Assumption \ref{assumption}) implies that we can find small enough $\kappa >0$ such that
\begin{align*}
    \mathbb{E} [\sigma(u(r,z))\sigma(u(r,z'))] \geqslant \frac{|\sigma(1)|^2}{2} > 0
\end{align*}
for all $(r,z,z') \in [0,\kappa] \times \mathbb{R}^3 \times \mathbb{R}^3$ and $\kappa <t$.
This leads to the estimate
\begin{align*}
    \sigma^2_R(t) &=\int_0^t\mathbb{E}[\lVert V_{t,R}(r,\cdot)\rVert_{\mathcal{H}}^2]dr
    \geqslant \int_0^{\kappa}\mathbb{E}[\lVert V_{t,R}(r,\cdot)\rVert_{\mathcal{H}}^2]dr\\
    &\geqslant \frac{|\sigma(1)|^2}{2}\int_0^{\kappa}dr\int_{\mathbb{R}^6}dzdz'\varphi_{t,R}(r,z)\varphi_{t,R}(r,z')\gamma(z-z')\\
    &= \frac{|\sigma(1)|^2}{2}\int_0^{\kappa}dr\int_{\mathbb{R}^6}G(t-r,dx)G(t-r,dx')\left(\int_{\mathbb{R}^6}dzdz'\mathbf{1}_{B_R}(z-x)\mathbf{1}_{B_R}(z'-x')\gamma(z-z')\right).
\end{align*}
As $\gamma(x) >0$ for all $x \in \mathbb{R}^3$, we get $\sigma^2_R(t) > 0$ for $R>0$. 
Now we assume $R\geqslant 2(T+1)$. Recall (see \eqref{fundamental solution}) that the support of the measure $G(t-r, dx)$ is $\partial B_{t-r} = \{x\in \mathbb{R}^3 \mid |x| = t-r\}$ .
Because 
\begin{align*}
    \mathbf{1}_{B_{R-T}}(z) \leqslant \mathbf{1}_{B_{R}}(z-x), \quad \mathbf{1}_{B_{R-T}}(z') \leqslant \mathbf{1}_{B_{R}}(z'-x')
\end{align*}
for any fixed $x,x' \in \partial B_{t-r}$, we obtain
\begin{align*}
    \sigma^2_R(t)
    &\geqslant \frac{|\sigma(1)|^2}{2}\int_0^{\kappa}dr\int_{\mathbb{R}^6}G(t-r,dx)G(t-r,dx')\left(\int_{\mathbb{R}^6}dzdz'\mathbf{1}_{B_{R-T}}(z)\mathbf{1}_{B_{R-T}}(z')\gamma(z-z')\right)\\
    &\geqslant \frac{|\sigma(1)|^2}{2}\int_0^{\kappa}dr\int_{\mathbb{R}^6}G(t-r,dx)G(t-r,dx')\int_{\mathbb{R}^3}\mathbf{1}_{B_{\frac{R-T}{2}}}(z)\left(\int_{\mathbb{R}^3}\mathbf{1}_{B_{R-T}}(z')\gamma(z-z')dz'\right)dz\\
    &= \frac{|\sigma(1)|^2}{2}\int_0^{\kappa}dr(t-r)^2\int_{\mathbb{R}^3}\mathbf{1}_{B_{\frac{R-T}{2}}}(z)\left(\int_{\mathbb{R}^3}\mathbf{1}_{B_{R-T}}(z-z')\gamma(z')dz'\right)dz,
\end{align*}
where the last equality follows from \eqref{G total measure}. 
Similarly, since $\mathbf{1}_{B_{\frac{R-T}{2}}}(z') \leqslant \mathbf{1}_{B_{R-T}}(z-z')$ for any fixed $z \in B_{\frac{R-T}{2}}$, 
we have
\begin{align*}
    \sigma^2_R(t)
    &\geqslant \frac{|\sigma(1)|^2}{2}\int_0^{\kappa}dr(t-r)^2\int_{\mathbb{R}^3}\mathbf{1}_{B_{\frac{R-T}{2}}}(z)dz\int_{\mathbb{R}^3}\mathbf{1}_{B_{\frac{R-T}{2}}}(z')\gamma(z')dz'\\
    &\geqslant \left( \frac{|\sigma(1)|^2\kappa(t-\kappa)^2|B_1|}{16}\left(1- \frac{T}{R}\right)^3\int_{B_{\frac{R-T}{2}}}\gamma(x)dx\right) R^3\\
    &\geqslant \left(\frac{|\sigma(1)|^2\kappa(t-\kappa)^2|B_1|}{128} \int_{B_{\frac{T}{2}}}\gamma(x)dx \right)R^3,
\end{align*}
and $\sigma^2_R(t) \gtrsim R^3$ is proved.
Similar argument is also applied to the case $\sigma^2_{n,R}(t)$ and we only sketch the proof.
Since $u_n(0,x)=1$ for all $x \in \mathbb{R}^3$ and \eqref{iteration uniform equicontinuous} holds, we have
\begin{align*}
    \lim_{r \to 0} \sup_n \sup_{z \in \mathbb{R}^3} \sup_{z' \in \mathbb{R}^3}|\mathbb{E}[\sigma(u_n(r,z))\sigma(u_n(r,z'))] - \sigma(1)^2|=0.
\end{align*}
From this we can find independently of $n$ small enough $\kappa \in (0,t)$ such that
\begin{align*}
    \mathbb{E} [\sigma(u_n(r,z))\sigma(u_n(r,z'))] \geqslant \frac{|\sigma(1)|^2}{2} > 0
\end{align*}
for every $n \geqslant 1$ and $(r,z,z') \in [0,\kappa] \times \mathbb{R}^3 \times \mathbb{R}^3$. Hence
\begin{align*}
    \sigma^2_{n,R}(t) &\geqslant \frac{|\sigma(1)|^2}{2}\int_0^{\kappa}dr\int_{\mathbb{R}^6}G_n(t-r,dx)G_n(t-r,dx')\left(\int_{\mathbb{R}^6}dzdz'\mathbf{1}_{B_{R}}(z-x)\mathbf{1}_{B_{R}}(z'-x')\gamma(z-z')\right),
\end{align*}
and we get $\sigma^2_{n,R}(t) > 0$ for every $n  \geqslant 1$ and $R>0$. Now assume that $R\geqslant 2(T+1)$. Since $\lVert G_n(t-r) \rVert_{L^1(\mathbb{R}^3)} = G(t-r,\mathbb{R}^3) = t-r$, we obtain 
\begin{align*}
    \sigma^2_{n,R}(t) \geqslant \left(\frac{|\sigma(1)|^2\kappa(t-\kappa)^2|B_1|}{128} \int_{B_{\frac{T+1}{2}}}\gamma(x)dx \right)R^3,
\end{align*}
by the same argument above. Therefore $\sigma^2_{n,R}(t) \gtrsim R^3 \quad (n \geqslant 1)$ and the proof is completed.
\end{proof}
Next we prove similar results for $\gamma(x) = |x|^{-\beta}$.
\begin{Lem}
\label{Lemma Riesz nondegeneracy}
Let $\gamma(x) = |x|^{-\beta}$ for some $0< \beta <2$. For $t\in(0,T]$, we have $\sigma_R^2(t) >0$, $\sigma_{n,R}^2(t)>0$ for any $R>0$ and $n \geqslant 1$. Furthermore, if $R \geqslant 1$, then we have 
\begin{gather}
\label{Riesz nondegeneracy}
    \sigma^2_R(t) \gtrsim R^{6-\beta}, \quad \sigma^2_{n,R}(t) \gtrsim R^{6-\beta} \quad ( n \geqslant 1).
\end{gather}
In particular, implicit constants in \eqref{Riesz nondegeneracy} are independent of $n$ and $R$.
\end{Lem}

\begin{proof}
Applying the same arguments in the proof of Lemma \ref{Lemma L1 nondegeneracy}, we can take $\kappa >0$ small enough to obtain
\begin{align*}
    \sigma^2_R(t) 
    &\geqslant \frac{|\sigma(1)|^2}{2}\int_0^{\kappa}dr\int_{\mathbb{R}^6}G(t-r,dx)G(t-r,dx')\left(\int_{\mathbb{R}^6}dzdz'\mathbf{1}_{B_R}(z-x)\mathbf{1}_{B_R}(z'-x')|z-z'|^{-\beta}\right)\\
    &\geqslant \frac{|\sigma(1)|^2}{2}\int_0^{\kappa}dr\int_{\mathbb{R}^6}G(t-r,dx)G(t-r,dx')\left(\int_{\mathbb{R}^6}dzdz'\mathbf{1}_{B_R}(z)\mathbf{1}_{B_R}(z')\langle x-x'+z-z' \rangle^{-\beta} \right),
\end{align*}
which proves $\sigma_R^2(t) >0$ for $R>0$.
Now we assume $R \geqslant 1$.
Using Peetre's inequality (see Proposition \ref{Peetre inequality}) twice, we get 
\begin{align*}
    \langle x-x'+z-z' \rangle^{-\beta} 
    &\geqslant 2^{-\frac{\beta}{2}}\langle z-z' \rangle^{-\beta}\langle x-x' \rangle^{-\beta}\\
    &\geqslant 2^{-\beta}\langle z-z' \rangle^{-\beta}\langle x \rangle^{-\beta}\langle x' \rangle^{-\beta},
\end{align*}
and we have
\begin{align*}
    \sigma^2_R(t)
    &\geqslant \frac{|\sigma(1)|^2}{2^{1+\beta}}\int_0^{\kappa}\left(\int_{\mathbb{R}^3}\langle x \rangle^{-\beta}G(t-r,dx)\right)^2dr\int_{B^2_R}\langle z-z' \rangle^{-\beta}dzdz'\\
    &\geqslant \frac{|\sigma(1)|^2|B_1|^2}{2^{1+\beta}5^{\frac{\beta}{2}}}\langle T \rangle^{-2\beta}\left(\int_0^{\kappa}\left(\int_{\mathbb{R}^3}G(t-r,dx)\right)^2dr\right) R^{6-\beta}\\
    &\geqslant \frac{|\sigma(1)|^2\kappa(t-\kappa)^2|B_1|^2}{2^{1+\beta}5^{\frac{\beta}{2}}}\langle T \rangle^{-2\beta} R^{6-\beta}.
\end{align*}
Therefore, we get $\sigma^2_R(t) \gtrsim R^{6-\beta}$.
The proof for $\sigma_{n,R}^2(t)$ is similar and we omit details.
\end{proof}

To end this section, we note that the following two conditions are equivalent in our setting:
\begin{enumerate}[(i)]
    \item $\sigma(1) \neq 0$.
    \item $\sigma_R(t) >0$ and $\sigma_{n,R}(t) > 0$ for every $R >0$, $t > 0$, and $n \geqslant 1$. 
\end{enumerate}
Indeed, if $\sigma(1) =0$, then the Picard iteration scheme \eqref{u_n+1} leads to that $u_n(t,x) = u(t,x) = 1$ for every $(t,x)$. Hence (ii) implies (i). The opposite direction follows from Lemmas \ref{Lemma L1 nondegeneracy} and \ref{Lemma Riesz nondegeneracy}. See \cite[Lemma 3.4]{SWE1dfracnoise} for similar arguments.

\subsection{Proof of Theorem \ref{main result1}}
\label{section Proof of Theorem}
We first prepare the following lemma.
\begin{Lem}
\label{Lemma lim sup norm zero}
Let $\gamma \in L^1(\mathbb{R}^3)$ such that $\gamma(x) >0$ for all $x \in \mathbb{R}^3$, or $\gamma(x) = |x|^{-\beta}$ for some $\beta \in (0,2)$. Then, for any fixed $t \in (0,T]$, we have
\begin{align}
\label{lim sup norm zero}
    \lim_{n\to \infty}\left( \sup_{R \geqslant 2(T+1)} \left \lVert \frac{F_{n,R}(t)}{\sigma_{n,R}(t)} - \frac{F_R(t)}{\sigma_{R}(t)} \right \rVert _2 \right) = 0.
\end{align}
\end{Lem}
\begin{proof}
Let $R \geqslant 2(T+1)$. We deduce from Lemmas \ref{Lemma L1 nondegeneracy} and \ref{Lemma Riesz nondegeneracy} that $\sigma_{n,R}(t)>0$, $\sigma_{R}(t)>0$.
Since $\sigma_{n,R}^2(t) = \lVert F_{n,R}(t) \rVert_2^2$ and $\sigma_{R}^2(t) = \lVert F_{R}(t) \rVert_2^2$, the triangle inequality yields that
\begin{align*}
    \left \lVert \frac{F_{n,R}(t)}{\sigma_{n,R}(t)} - \frac{F_R(t)}{\sigma_R(t)} \right \rVert _2^2 
    &\leqslant \frac{2}{{\sigma^2_R}(t)}\lVert F_R(t) - F_{n,R}(t)\rVert_2^2 + 2\frac{|\sigma_R(t) - \sigma_{n,R}(t)|^2}{{\sigma^2_R}(t){\sigma^2_{n,R}}(t)}\lVert F_{n,R}(t)\rVert_2^2\\
    &= \frac{2}{{\sigma^2_R}(t)} \left \{ \lVert F_R(t) - F_{n,R}(t)\rVert_2^2 +  \left | \lVert F_R(t)\rVert_2 - \lVert F_{n,R}(t)\rVert_2 \right|^2 \right\} \\
    &\leqslant \frac{4}{{\sigma^2_R}(t)}\lVert F_R(t) - F_{n,R}(t)\rVert_2^2.
\end{align*}
Taking into account \eqref{fn delta} and \eqref{f delta}, we have $\lVert F_R(t) - F_{n,R}(t) \rVert_2^2 \leqslant 2 \mathbf{F_1} + 2\mathbf{F_2}$, where
\begin{align*}
    \mathbf{F_1} &= \int_0^tdr\int_{\mathbb{R}^6}dzdz'\varphi_{n,t,R}(r,z)\varphi_{n,t,R}(r,z')\gamma(z-z')\\
    & \qquad \times \mathbb{E} [(\sigma(u_n(r,z)) - \sigma(u(r,z)))(\sigma(u_n(r,z')) - \sigma(u(r,z')))],\\
    \mathbf{F_2} &= \int_0^tdr\int_{\mathbb{R}^6}dzdz'(\varphi_{n,t,R}(r,z) - \varphi_{t,R}(r,z))(\varphi_{n,t,R}(r,z') - \varphi_{t,R}(r,z'))\\
    &\qquad \times \gamma(z-z') \mathbb{E}[\sigma(u(r,z))\sigma(u(r,z'))].\\
\end{align*}
Using the Lipschitz continuity of $\sigma$, the Fourier transform, and \eqref{FG bound}, we get
\begin{align}
    \mathbf{F_1} 
    &\leqslant L^2 \int_0^tdr\int_{\mathbb{R}^6}dzdz'\varphi_{n,t,R}(r,z)\varphi_{n,t,R}(r,z')\gamma(z-z')\nonumber\\
    & \qquad \times \lVert u_n(r,z) - u(r,z)\rVert_2\lVert u_n(r,z') - u(r,z')\rVert_2\nonumber\\
    &\lesssim \sup_{(\theta,\eta) \in [0,T]\times\mathbb{R}^3}\lVert u_n(\theta,\eta) - u(\theta,\eta)\rVert_2^2 \int_0^tdr \int_{\mathbb{R}^3}|\mathcal{F} \mathbf{1}_{B_R}(\xi)|^2 |\mathcal{F} G(t-r)(\xi)|^2 |\mathcal{F} \rho_n(\xi)|^2 \mu(d\xi)\nonumber\\
    &\lesssim_T \sup_{(\theta,\eta) \in [0,T]\times\mathbb{R}^3}\lVert u_n(\theta,\eta) - u(\theta,\eta)\rVert_2^2 \int_{\mathbb{R}^3}|\mathcal{F} \mathbf{1}_{B_R}(\xi)|^2\mu(d\xi)\nonumber\\
    &= \sup_{(\theta,\eta) \in [0,T]\times\mathbb{R}^3}\lVert u_n(\theta,\eta) - u(\theta,\eta)\rVert_2^2 \int_{{B^2_R}}\gamma(z-z')dzdz'.\label{c1}
\end{align}
where the last equality follows from Lemma \ref{Lemma fourier transform}. 
This together with \eqref{iteration uniform convergence} implies that for any $\varepsilon > 0$, we can find $n$ large enough such that
\begin{align}
\label{F_1 estimate}
    \frac{1}{\sigma_{R}^2(t)}\mathbf{F_1} \lesssim \varepsilon \frac{1}{\sigma_{R}^2(t)} \int_{{B^2_R}}\gamma(z-z')dzdz'.
\end{align}
For the term $\mathbf{F_2}$, the Fourier transform and \eqref{fourier G bound} yield that 
\begin{align*}
    \mathbf{F_2} 
    &= \int_0^tdr\int_{\mathbb{R}^3}\mu_{r}^{\sigma(u)}(d\xi)|\mathcal{F}\varphi_{n,t,R}(r)(\xi) - \mathcal{F}\varphi_{t,R}(r)(\xi)|^2\\
    &\leqslant \int_0^tdr \int_{\mathbb{R}^3}\mu_{r}^{\sigma(u)}(d\xi)|\mathcal{F} \mathbf{1}_{B_R}(\xi)|^2 |\mathcal{F} G(t-r)(\xi)|^2 |\mathcal{F} \rho_n(\xi) - 1|^2 \\
    &\lesssim_T \int_0^tdr\int_{\mathbb{R}^3}\mu_{r}^{\sigma(u)}(d\xi)|\mathcal{F} \mathbf{1}_{B_R}(\xi)|^2 |\mathcal{F} \rho_n(\xi) - 1|^2 \langle \xi \rangle^{-2},
\end{align*}
where $\mu_{r}^{\sigma(u)}$ is defined in \eqref{mu^sigma(u)}. 
Let 
\begin{align*}
    K_{\varepsilon} \coloneqq \left \{ \xi \in \mathbb{R}^3 \bigg| \: \langle \xi \rangle^{-2} \geqslant \varepsilon \right\}
\end{align*}
for any given $0 < \varepsilon < 1 $. Then we have 
\begin{align}
    &\int_0^tdr\int_{\mathbb{R}^3}|\mathcal{F} \mathbf{1}_{B_R}(\xi)|^2 |\mathcal{F} \rho_n(\xi) - 1|^2 \langle \xi \rangle^{-2}\mu_{r}^{\sigma(u)}(d\xi)\nonumber\\
    &\leqslant 4\varepsilon \int_0^tdr\int_{\mathbb{R}^3\backslash K_{\varepsilon}}|\mathcal{F} \mathbf{1}_{B_R}(\xi)|^2\mu_{r}^{\sigma(u)}(d\xi) + \int_0^tdr\int_{K_{\varepsilon}} |\mathcal{F} \mathbf{1}_{B_R}(\xi)|^2 |\mathcal{F} \rho_n(\xi) - 1|^2 \langle \xi \rangle^{-2} \mu_{r}^{\sigma(u)}(d\xi)\nonumber\\
    &\leqslant \left(4 \varepsilon +  \sup_{x \in K_{\varepsilon}} |\mathcal{F} \rho_n(x) - 1|^2   \right) \int_0^tdr\int_{\mathbb{R}^3} |\mathcal{F} \mathbf{1}_{B_R}(\xi)|^2 \mu_{r}^{\sigma(u)}(d\xi)\nonumber\\
    &= \left(4 \varepsilon +  \sup_{x \in K_{\varepsilon}} \left|\mathcal{F} \rho\left(\frac{x}{a_n}\right) - 1\right|^2   \right) \int_0^tdr\int_{{B^2_R}}\gamma(z-z')\mathbb{E}[\sigma(u(r,z-z'))\sigma(u(r,0))]dzdz'\nonumber\\
    &\lesssim_T \left(4 \varepsilon +  \sup_{x \in K_{\varepsilon}} \left|\mathcal{F} \rho\left(\frac{x}{a_n}\right) - 1\right|^2 \right) \int_{{B^2_R}}\gamma(z-z')dzdz',\label{c2}
\end{align}
where the last step follows from H\"{o}lder's inequality and \eqref{solution uniform bound}.
Since $\lim_{y \to 0}\mathcal{F}\rho(y) =1$, we have the uniform convergence on the compact set $K_{\varepsilon}$;
\begin{align*}
    \lim_{n \to \infty} \sup_{x \in K_{\varepsilon}} \left|\mathcal{F} \rho\left(\frac{x}{a_n}\right) - 1\right|^2 = 0.
\end{align*}
Therefore, we can take $n$ large enough to get 
\begin{align}
\label{F_2 estimate}
    \frac{1}{\sigma_{R}^2(t)}\mathbf{F_2} \lesssim \varepsilon \frac{1}{\sigma_{R}^2(t)} \int_{{B^2_R}}\gamma(z-z')dzdz'.
\end{align}
Combining \eqref{F_1 estimate}, \eqref{F_2 estimate}, and Lemmas \ref{Lemma L1 nondegeneracy} and \ref{Lemma Riesz nondegeneracy}, we obtain
\begin{align*}
    \left \lVert \frac{F_{n,R}(t)}{\sigma_{n,R}(t)} - \frac{F_R(t)}{\sigma_{R}(t)} \right \rVert _2^2 \lesssim 
    \begin{cases}
    \displaystyle
        \varepsilon \frac{1}{R^3}\int_{{B^2_R}}\gamma(z-z')dzdz'  &(\gamma \in L^1(\mathbb{R}^3))\\
    \displaystyle
        \varepsilon \frac{1}{R^{6-\beta}}\int_{{B^2_R}}\gamma(z-z')dzdz'  &(\gamma(x) = |x|^{-\beta} \quad (0 < \beta <2))
    \end{cases}
\end{align*}
for large enough $n$.
It is easy to check 
\begin{align}
\label{gamma ball order}
    \int_{{B^2_R}}\gamma(z-z')dzdz' \lesssim 
    \begin{cases}
        R^3  &(\gamma \in L^1(\mathbb{R}^3))\\
        R^{6-\beta}  &(\gamma(x) = |x|^{-\beta} \quad (0 < \beta <2))
    \end{cases},
\end{align}
and so 
\begin{align*}
    \sup_{R \geqslant 2(T+1)} \left \lVert \frac{F_{n,R}(t)}{\sigma_{n,R}(t)} - \frac{F_R(t)}{\sigma_{R}(t)} \right \rVert _2^2 \lesssim \varepsilon, 
\end{align*}
which gives \eqref{lim sup norm zero}. This completes the proof.
\end{proof}

The following lemma is the key in proving Theorems \ref{main result1} and \ref{main result2}.
\begin{Lem}
\label{key}
Let $t_1,t_2 \in (0,T]$, $n \geqslant 1$, and $R \geqslant T+1$. Then, the following statements hold true:
\begin{enumerate}
    \item[(1)] When $\gamma \in L^1(\mathbb{R}^3)$, then,
    \begin{align}
    \label{y7}
        \mathrm{Var}(\langle DF_{n,R}(t_1),V_{n,t_2,R}\rangle_{\mathcal{H}_T})\lesssim \Theta(T,n-1)^2R^{3}.
    \end{align}
    \item[(2)] When $\gamma(x) = |x|^{-\beta}$ for some $0<\beta<2$, then,
    \begin{align}
    \label{y8}
        \mathrm{Var}(\langle DF_{n,R}(t_1),V_{n,t_2,R}\rangle_{\mathcal{H}_T}) \lesssim \Theta(T,n-1)^2R^{12-3\beta}.
    \end{align}
\end{enumerate}
\end{Lem}

\begin{proof}
The proof is similar to \cite[Section 3.2]{Averaging2dSWE} and we only give the sketch of the proof.
Let $p \geqslant 1$.
Because $u_n(t,x) \in \mathbb{D}^{1,p}$ for all $(t,x) \in [0,T] \times \mathbb{R}^3$, and \eqref{Du_n iteration} and \eqref{M_nDu_n} hold, we can show that $F_{n,R}(t) \in \mathbb{D}^{1,p}$ for all $t \in [0,T]$ and that it satisfies 
\begin{align*}
    DF_{n,R}(t) = V_{n,t,R}(\cdot,\star) + \int_0^t\int_{\mathbb{R}^3}\varphi_{n,t,R}(r,z)\sigma^{\prime}(u_{n-1}(r,z)) M_{n-1}(r,z,\cdot,\star)W(dr,dz).
\end{align*}
From this and Lemma \ref{Lemma L2(Omega,H_T)}, we see that
\begin{align*}
    &\langle DF_{n,R}(t_1), V_{n,t_2,R}\rangle_{\mathcal{H}_T}\\
    &= \langle V_{n,t_1,R}, V_{n,t_2,R}\rangle_{\mathcal{H}_T}
    + \left \langle \int_0^{t_1}\int_{\mathbb{R}^3}\varphi_{n,t_1,R}(r,z)\sigma^{\prime}(u_{n-1}(r,z)) M_{n-1}(r,z,\cdot,\star)W(dr,dz), V_{n,t_2,R} \right \rangle _{\mathcal{H}_T}\\
    &= \int_0^{t_1 \land t_2}ds\int_{\mathbb{R}^6}dydy'\varphi_{n,t_1,R}(s,y)\varphi_{n,t_2,R}(s,y')\gamma(y-y')\sigma(u_{n-1}(s,y))\sigma(u_{n-1}(s,y'))\\
    &\quad + \int_0^{t_1 \land t_2}ds\int_{\mathbb{R}^6}dydy'\left( \int_s^{t_1}\int_{\mathbb{R}^3}\varphi_{n,t_1,R}(r,z)\sigma^{\prime}(u_{n-1}(r,z)) M_{n-1}(r,z,s,y)W(dr,dz) \right)\\
    &\qquad \qquad \qquad \qquad \times \varphi_{n,t_2,R}(s,y')\sigma(u_{n-1}(s,y'))\gamma(y-y')\\
    &\eqqcolon \mathbf{V_1} + \mathbf{V_2}.
\end{align*}
Thus we obtain $\mathrm{Var}(\langle DF_{n,R}(t_1), V_{n,t_2,R}\rangle_{\mathcal{H}_T}) \leqslant 2\mathrm{Var}(\mathbf{V_1}) + 2\mathrm{Var}(\mathbf{V_2})$.
Using the following inequality
\begin{align*}
    \mathrm{Var}\left(\int_{\mathbb{R}_{+}}X_sds\right) \leqslant \left(\int_{\mathbb{R}_{+}}\sqrt{\mathrm{Var}(X_s)}ds\right)^2,
\end{align*}
which holds for every measurable process $X = \{X(s,\omega)\}_{s \geqslant 0}$ such that $\int_{\mathbb{R}_{+}}\sqrt{\mathrm{Var}(X_s)}ds < \infty$,
we get
\begin{align}
    &\mathrm{Var}(\mathbf{V_1}) \nonumber\\
    &\leqslant \left( \int_0^{t_1 \land t_2}\left \{ \mathrm{Var}\left(\int_{\mathbb{R}^6}dydy'\varphi_{n,t_1,R}(s,y)\varphi_{n,t_2,R}(s,y')\gamma(y-y')\sigma(u_{n-1}(s,y))\sigma(u_{n-1}(s,y'))\right) \right\}^{\frac{1}{2}}ds \right)^2, \label{y3}
\end{align}
where
\begin{align}
    &\mathrm{Var}\left(\int_{\mathbb{R}^6}dydy'\varphi_{n,t_1,R}(s,y)\varphi_{n,t_2,R}(s,y')\gamma(y-y')\sigma(u_{n-1}(s,y))\sigma(u_{n-1}(s,y'))\right)\nonumber\\
    &=\int_{\mathbb{R}^{12}}dydy'dzdz'\varphi_{n,t_1,R}(s,y)\varphi_{n,t_2,R}(s,y')\gamma(y-y')\varphi_{n,t_1,R}(s,z)\varphi_{n,t_2,R}(s,z')\gamma(z-z')\nonumber\\
    &\quad \qquad \times \mathrm{Cov}(\sigma(u_{n-1}(s,y))\sigma(u_{n-1}(s,y')), \sigma(u_{n-1}(s,z))\sigma(u_{n-1}(s,z'))).\label{y1}
\end{align}
Since 
\begin{align*}
    D\sigma(u_{n-1}(s,y))\sigma(u_{n-1}(s,y')) &= \sigma^{\prime}(u_{n-1}(s,y))Du_{n-1}(s,y)\sigma(u_{n-1}(s,y'))\\
    &\quad + \sigma(u_{n-1}(s,y))\sigma^{\prime}(u_{n-1}(s,y'))Du_{n-1}(s,y')
\end{align*}
by the chain rule of the Malliavin derivative, 
we deduce from the Poincar\'{e} inequality \eqref{Poincare inequality}, Cauchy–Schwarz inequality, \eqref{iteration uniform bound}, and Theorem \ref{Theorem moment estimate} that
\begin{align}
    &\mathrm{Cov}(\sigma(u_{n-1}(s,y))\sigma(u_{n-1}(s,y')), \sigma(u_{n-1}(s,z))\sigma(u_{n-1}(s,z')))\nonumber\\
    &\lesssim \int_0^sdr\int_{\mathbb{R}^6}dwdw'\left( \lVert M_{n-1}(s,y,r,w)\rVert_4 + \lVert M_{n-1}(s,y',r,w)\rVert_4 \right)\nonumber\\
    &\qquad \qquad \qquad \times \left( \lVert M_{n-1}(s,z,r,w')\rVert_4 + \lVert M_{n-1}(s,z',r,w')\rVert_4 \right)\gamma(w-w')\nonumber\\
    &\lesssim \Theta(T,n-1)^2\int_0^sdr\int_{\mathbb{R}^6}dwdw'\left( \mathbf{1}_{B_{T+1}}(y-w) + \mathbf{1}_{B_{T+1}}(y'-w) \right)\nonumber\\
    &\qquad \qquad \qquad \qquad \qquad \times \left( \mathbf{1}_{B_{T+1}}(z-w') + \mathbf{1}_{B_{T+1}}(z'-w') \right)\gamma(w-w'). \label{y2}
\end{align}
Taking into account \eqref{y3}, \eqref{y1}, and \eqref{y2}, we only need to estimate
\begin{align*}
    &\Theta(T,n-1)^2\int_0^sdr\int_{\mathbb{R}^{12}}dydy'dzdz'\varphi_{n,t_1,R}(s,y)\varphi_{n,t_2,R}(s,y')\gamma(y-y')\varphi_{n,t_1,R}(s,z)\varphi_{n,t_2,R}(s,z')\\
    &\qquad\qquad \qquad \qquad \times \gamma(z-z')\left(\int_{\mathbb{R}^6}\mathbf{1}_{B_{T+1}}(y'-w)\mathbf{1}_{B_{T+1}}(z'-w')\gamma(w-w')dwdw'\right)
\end{align*}
because the other terms appearing from \eqref{y2} can be estimated in the same way with the same bound.
Applying Lemma \ref{Lemma estimate} with $F(y'-z') \coloneqq  \int_{B^2_{T+1}}\gamma(y'-z'+w-w')dwdw'$, we easily check that
\begin{align}
\label{y5}
    \mathrm{Var}(\mathbf{V_1}) \lesssim \Theta(T,n-1)^2 \times
    \begin{cases}
        R^3  &(\gamma \in L^1(\mathbb{R}^3))\\
        R^{12-3\beta}  &(\gamma(x) = |x|^{-\beta} \quad (0<\beta<2))
    \end{cases}.
\end{align}
Next we consider the term $\mathrm{Var}(\mathbf{V_2})$. Following the argument in \cite{Averaging2dSWE} again, we only need to estimate 
\begin{align}
    &\Theta(T,n-1)^2 \int_s^{t_1}dr\int_{\mathbb{R}^{18}}dzdz'dydy'd\tilde{y}d\tilde{y}'\gamma(y-y') \gamma(\tilde{y}-\tilde{y}')\varphi_{n,t_2,R}(s,y')\varphi_{n,t_2,R}(s,\tilde{y}')\nonumber\\
    &\qquad \qquad \qquad \qquad \qquad  \times \varphi_{n,t_1,R}(r,z+y)\varphi_{n,t_1,R}(r,z'+\tilde{y})\mathbf{1}_{B_{T+1}}(z)\mathbf{1}_{B_{T+1}}(z')\gamma(y-\tilde{y}+z-z'). \label{y4}
\end{align}
Because 
\begin{align*}
    \varphi_{n,t_1,R}(r,z+y) \leqslant \varphi_{n,t_1,R+T+1}(r,y), \quad \varphi_{n,t_2,R}(s,y') \leqslant \varphi_{n,t_2,R+T+1}(s,y')
\end{align*}
for every $|z|, |z'| < T+1$, \eqref{y4} is bounded from above by
\begin{align*}
    &\Theta(T,n-1)^2\int_s^{t_1}dr
    \int_{\mathbb{R}^{12}}dydy'd\tilde{y}d\tilde{y}'\varphi_{n,t_1,R+T+1}(r,y)\varphi_{n,t_2,R+T+1}(s,y')\gamma(y-y')\\
    &\qquad \qquad \qquad \qquad \times \varphi_{n,t_1,R+T+1}(r,\tilde{y})\varphi_{n,t_2,R+T+1}(s,\tilde{y}')\gamma(\tilde{y}-\tilde{y}')\left(\int_{B^2_{T+1}}\gamma(y-\tilde{y}+z-z')dzdz'\right)
\end{align*}
Then, using Lemma \ref{Lemma estimate} with $F(y-\tilde{y}) \coloneqq  \int_{B^2_{T+1}}\gamma(y-\tilde{y}+z-z')dzdz'$, we obtain
\begin{align*}
    \mathrm{Var}(\mathbf{V_2}) \lesssim \Theta(T,n-1)^2 \times
    \begin{cases}
        (R+T+1)^3  &(\gamma \in L^1(\mathbb{R}^3))\\
        (R+T+1)^{12-3\beta}  &(\gamma(x) = |x|^{-\beta},\quad (0<\beta<2))
    \end{cases}.
\end{align*}
In particular, when $R \geqslant T+1$, we have
\begin{align}
\label{y6}
    \mathrm{Var}(\mathbf{V_2}) \lesssim \Theta(T,n-1)^2 \times
    \begin{cases}
        R^3  &(\gamma \in L^1(\mathbb{R}^3))\\
        R^{12-3\beta}  &(\gamma(x) = |x|^{-\beta},\quad (0<\beta<2))
    \end{cases}.
\end{align}
Finally, combining \eqref{y5} and \eqref{y6}, we obtain \eqref{y7} and \eqref{y8}.
Thus Lemma \ref{key} holds.
\end{proof}

Finally, we are now able to prove Theorem \ref{main result1}.\\
\noindent
\textit{Proof of Theorem \ref{main result1}.} The nondegeneracy of the variance in Theorem \ref{main result1} follows from Lemmas \ref{Lemma L1 nondegeneracy} and \ref{Lemma Riesz nondegeneracy}.
Let $R \geqslant 2(T+1)$. By the triangle inequality, we see that
\begin{align}
    d_{\mathrm{W}}\left(\frac{F_R(t)}{\sigma_R(t)}, \mathcal{N}(0,1) \right) 
    &\leqslant d_{\mathrm{W}}\left( \frac{F_R(t)}{\sigma_R(t)}, \frac{F_{n,R}(t)}{\sigma_{n,R}(t)} \right) + d_{\mathrm{W}}\left( \frac{F_{n,R}(t)}{\sigma_{n,R}(t)}, \mathcal{N}(0,1) \right) \nonumber\\
    &\leqslant \sup_{R \geqslant 2(T+1)} \left \lVert \frac{F_{n,R}(t)}{\sigma_{n,R}(t)} - \frac{F_R(t)}{\sigma_{R}(t)} \right \rVert _2 + d_{\mathrm{W}}\left( \frac{F_{n,R}(t)}{\sigma_{n,R}(t)}, \mathcal{N}(0,1) \right), \label{beta1}
\end{align}
where the last inequality is easily checked by the definition of $d_{\mathrm{W}}$ (see \eqref{Wasserstein distance def}). Now we apply Lemma \ref{Lemma lim sup norm zero}. For any $\varepsilon>0$, we can find large enough $n$ such that 
\begin{align}
\label{beta2}
    \sup_{R \geqslant 2(T+1)} \left \lVert \frac{F_{n,R}(t)}{\sigma_{n,R}(t)} - \frac{F_R(t)}{\sigma_{R}(t)} \right \rVert _2 < \varepsilon.
\end{align}
From now on, we fix such $n$.
It remains to prove that the second term on the right-hand side of \eqref{beta1} converges to 0 as $R \to \infty$.
Thanks to Proposition \ref{Proposition Wasserstein bound}, we have
\begin{align*}
    d_{\mathrm{W}}\left( \frac{F_{n,R}(t)}{\sigma_{n,R}(t)}, \mathcal{N}(0,1) \right) 
    \leqslant \sqrt{\frac{2}{\pi}}\frac{\sqrt{\mathrm{Var} \langle DF_{n,R}(t),V_{n,t,R}\rangle_{\mathcal{H}_T}}}{\sigma^2_{n,R}(t)}.
\end{align*}
Therefore, it follows from Lemmas \ref{Lemma L1 nondegeneracy}, \ref{Lemma Riesz nondegeneracy}, and \ref{key} that
\begin{align*}
    d_{\mathrm{W}}\left(\frac{F_{n,R}(t)}{\sigma_{n,R}(t)}, \mathcal{N}(0,1) \right) \lesssim \Theta(T,n-1) \times
    \begin{cases}
        R^{-\frac{3}{2}}  &(\gamma \in L^1(\mathbb{R}^3))\\
        R^{-\frac{\beta}{2}}  &(\gamma(x) = |x|^{-\beta} \quad (0<\beta<2))
    \end{cases}.
\end{align*}
This together with \eqref{beta2} implies \eqref{distance convergence}. Thus Theorem \ref{main result1} follows. \qed

\section{Functional Central Limit Theorems}
\label{FCLT}
In this section we prove Theorem \ref{main result2}. It is sufficient to show the convergence of the finite-dimensional distributions and the tightness. The former is proved in Section \ref{section Convergence of finite dimensional distributions}  and the latter in Section \ref{section Tightness}.
In Section \ref{section Limits of the covariance functions} we first determine the limit of the covariance of $F_R(t)$ and $F_{n,R}(t)$ as $R \to \infty$.
\subsection{Limits of the covariance functions}
\label{section Limits of the covariance functions}
Let us first consider the case $\gamma \in L^1(\mathbb{R}^3)$.
\begin{Prop}
\label{Proposition L1 limit}
Let $\gamma \in L^1(\mathbb{R}^3)$. Then, for any $t_1, t_2 \in [0,T]$, 
\begin{align}
    \lim_{R \to \infty} R^{-3}\mathbb{E}[F_{n,R}(t_1)F_{n,R}(t_2)] &= |B_1| \int_{\mathbb{R}^3}\mathrm{Cov}(u_n(t_1,x),u_n(t_2,0))dx, \label{l1 n limit}\\
        \lim_{R \to \infty} R^{-3}\mathbb{E}[F_{R}(t_1)F_{R}(t_2)] &= |B_1| \int_{\mathbb{R}^3}\mathrm{Cov}(u(t_1,x),u(t_2,0))dx. \label{l1 limit}
\end{align}
\end{Prop}
\begin{proof}
The strict stationarity of $u_n(t,x)$ leads to 
\begin{align*}
    \mathbb{E}[F_{n,R}(t_1)F_{n,R}(t_2)] &= \int_{{B^2_R}}\mathbb{E}[(u_n(t_1,x)-1)(u_n(t_2,y)-1)]dxdy\\
    &= \int_{{B^2_R}}\mathrm{Cov}(u_n(t_1,x-y),  u_n(t_2,0))dxdy\\
    &= R^{3}\int_{B_1}\left( \int_{B_R(-Ry)}\mathrm{Cov}(u_n(t_1,x),u_n(t_2,0))dx \right) dy.
\end{align*}
Because, by Lemma \ref{G_n property} and \eqref{iteration uniform bound},
\begin{align*}
    &\int_{\mathbb{R}^3} |\mathrm{Cov}(u_n(t_1,x),u_n(t_2,0))|dx\\
    &\leqslant \int_{\mathbb{R}^3}dx\int_0^{t_1 \land t_2}dr \int_{\mathbb{R}^6}dzdz'G_n(t_1-r,x-z)G_n(t_2-r,-z') \gamma(z-z')\\
    &\qquad \qquad \times 
    |\mathbb{E}[\sigma(u_{n-1}(r,z))\sigma(u_{n-1}(r,z'))]|\\
    &\lesssim_T \int_0^{t_1 \land t_2}dr \int_{\mathbb{R}^6}dzdz'G_n(t_2-r,-z') \gamma(z-z')\\
    &\lesssim_T \lVert \gamma \rVert_{L^1(\mathbb{R}^3)} < \infty,
\end{align*}
we conclude from the Lebesgue dominated convergence theorem that \eqref{l1 n limit} holds true.
Furthermore, by \eqref{iteration uniform bound} and \eqref{iteration uniform convergence} we apply Fatou's lemma to obtain 
\begin{align*}
    \int_{\mathbb{R}^3}|\mathrm{Cov}(u(t_1,x),u(t_2,0))|dx \leqslant \liminf_{n \to \infty} \int_{\mathbb{R}^3}|\mathrm{Cov}(u_n(t_1,x),u_n(t_2,0))|dx \lesssim_T  \lVert \gamma \rVert_{L^1(\mathbb{R}^3)} < \infty.
\end{align*}
Therefore, \eqref{l1 limit} also follows from the strict stationarity of $u(t,x)$ and the Lebesgue dominated convergence theorem.
\end{proof}

Next we consider the case $\gamma(x) = |x|^{-\beta}$.
Let $\gamma_0(x) \coloneqq \gamma(x)\mathbf{1}_{B_1}(x)$ for all $x \in \mathbb{R}^3$. Then we have
\begin{align}
\label{riesz bound}
    \gamma(x) =|x|^{-\beta} = \gamma_0(x) + |x|^{-\beta}\mathbf{1}_{\mathbb{R}^3 \setminus B_1}(x) \leqslant \gamma_0(x) + 2\langle x \rangle^{-\beta}.
\end{align}
Notice that $\gamma_0 \in L^1(\mathbb{R}^3)$. 
In order to obtain the limit of the covariance, we prepare some lemmas.
\begin{Lem}
\label{Lemma Cov n estimate}
Let $\gamma(x) = |x|^{-\beta}$ for some $0<\beta<2$. Then, there is an integrable function $g \in L^1(\mathbb{R}^3)$ such that
\begin{align*}
    |\mathrm{Cov}(\sigma(u_n(t,x)), \sigma(u_n(s,y)))| \lesssim_{T,\beta} \Theta(T,n)^2 \{g(x-y) + \langle x-y \rangle^{-\beta}\}
\end{align*}
\end{Lem}
\begin{proof}
By the Poincar\'{e} inequality \eqref{Poincare inequality}, Theorem \ref{Theorem moment estimate}, and \eqref{riesz bound}, we have
\begin{align*}
    &|\mathrm{Cov}(\sigma(u_n(t,x)), \sigma(u_n(s,y)))|\\
    &\lesssim \int_0^{t \land s} dr \int_{\mathbb{R}^6} \lVert M_n(t,x,r,z)\rVert_2 \lVert M_n(s,y,r,z')\rVert_2\gamma(z-z')dzdz'\\
    &\lesssim \Theta(T,n)^2 \int_{\mathbb{R}^6} \mathbf{1}_{B_{T+1}}(x-z) \mathbf{1}_{B_{T+1}}(y-z')\gamma(z - z')dzdz'\\
    &\lesssim \Theta(T,n)^2 \bigg \{ \int_{\mathbb{R}^6} \mathbf{1}_{B_{T+1}}(x-z) \mathbf{1}_{B_{T+1}}(y-z')\gamma_0(z - z')dzdz'\\
    &\qquad \qquad \qquad + \int_{\mathbb{R}^6} \mathbf{1}_{B_{T+1}}(x-z) \mathbf{1}_{B_{T+1}}(y-z')\langle z-z' \rangle^{-\beta}dzdz' \bigg \}.
\end{align*}
Let 
\begin{align*}
    g(x-y) \coloneqq \int_{\mathbb{R}^6} \mathbf{1}_{B_{T+1}}(x-z) \mathbf{1}_{B_{T+1}}(y-z')\gamma_0(z - z')dzdz'.
\end{align*}
Then, it is easy to check that $g \in L^1(\mathbb{R}^3)$ and $\lVert g \rVert_{L^1(\mathbb{R}^3)} \lesssim_T \lVert \gamma_0 \rVert_{L^1(\mathbb{R}^3)} < \infty$.
Furthermore, by Peetre's inequality (Lemma \ref{Peetre inequality}), we obtain
\begin{align*}
    &\int_{\mathbb{R}^6} \mathbf{1}_{B_{T+1}}(x-z) \mathbf{1}_{B_{T+1}}(y-z')\langle z-z' \rangle^{-\beta}dzdz'\\
    &\lesssim_{\beta} \langle x-y \rangle^{-\beta} \left (\int_{\mathbb{R}^3} \mathbf{1}_{B_{T+1}}(x-z)\langle x-z \rangle^{\beta}dz \right) \left (\int_{\mathbb{R}^3} \mathbf{1}_{B_{T+1}}(y-z')\langle y-z' \rangle^{\beta}dz' \right)\\
    &\lesssim_{T,\beta} \langle x-y \rangle^{-\beta}.
\end{align*}
Thus, Lemma \ref{Lemma Cov n estimate} holds.
\end{proof}

\begin{Lem}
\label{Lemma varphi cov order}
Let $\gamma(x) = |x|^{-\beta}$ for some $0<\beta<2$ and $R \geqslant 2$. Then, 
\begin{align}
    &\int_0^{t_1 \land t_2}dr\int_{\mathbb{R}^6}\varphi_{t_1,R}(r,z)\varphi_{t_2,R}(r,z')\gamma(z-z')|\mathrm{Cov}(\sigma(u_n(r,z)), \sigma(u_n(r,z')))|dzdz' \label{varphi cov}\\
    &\lesssim \Theta(T,n)^2 \times
    \begin{cases}
        R^{6-2\beta}  &(0<\beta < \frac{3}{2})\\
        R^3\log R  &(\beta = \frac{3}{2})\\
        R^3  &(\beta > \frac{3}{2})
    \end{cases}. \label{varphi cov order}
\end{align}
\end{Lem}
\begin{proof}
By \eqref{riesz bound}, \eqref{varphi cov} is bounded from above by
\begin{align*}
    &\int_0^{t_1 \land t_2}dr\int_{\mathbb{R}^6}\varphi_{t_1,R}(r,z)\varphi_{t_2,R}(r,z')\gamma_0(z-z')|\mathrm{Cov}(\sigma(u_n(r,z)), \sigma(u_n(r,z')))|dzdz'\\
    &\quad + 2\int_0^{t_1 \land t_2}dr\int_{\mathbb{R}^6}\varphi_{t_1,R}(r,z)\varphi_{t_2,R}(r,z')\langle z-z' \rangle^{-\beta}|\mathrm{Cov}(\sigma(u_n(r,z)), \sigma(u_n(r,z')))|dzdz'\\
    &\eqqcolon \mathbf{R_1} + 2\mathbf{R_2}.
\end{align*}
From \eqref{iteration uniform bound} and Lemma \ref{varphi property}, we obtain that
\begin{align}
    \mathbf{R_1} &\lesssim \int_0^{t_1 \land t_2}dr\int_{\mathbb{R}^6}\varphi_{t_1,R}(r,z)\varphi_{t_2,R}(r,z')\gamma_0(z-z')dzdz'\nonumber\\
    &\lesssim_T \int_0^{t_1 \land t_2}dr\int_{\mathbb{R}^6}\varphi_{t_2,R}(r,z')\gamma_0(z-z')dzdz'\nonumber\\
    &\lesssim_T R^3. \label{R_1 estimate}
\end{align}
For the term $\mathbf{R_2}$, Lemma \ref{Lemma Cov n estimate} implies that there is a function $g \in L^1(\mathbb{R}^3)$ such that
\begin{align*}
    \mathbf{R_2} &\lesssim \Theta(T,n)^2 \int_0^{t_1 \land t_2}dr\int_{\mathbb{R}^6}\varphi_{t_1,R}(r,z)\varphi_{t_2,R}(r,z')\langle z-z' \rangle^{-\beta} g(z-z') dzdz'\\
    &\quad + \Theta(T,n)^2 \int_0^{t_1 \land t_2}dr\int_{\mathbb{R}^6}\varphi_{t_1,R}(r,z)\varphi_{t_2,R}(r,z')\langle z-z' \rangle^{-2\beta}dzdz'\\
    &\eqqcolon \mathbf{R_{21}} + \mathbf{R_{22}}.
\end{align*}
Since $g \in L^1(\mathbb{R}^3)$, $\mathbf{R_{21}}$ is estimated by 
\begin{align}
    \mathbf{R_{21}} &\lesssim \Theta(T,n)^2\int_0^{t_1 \land  t_2}dr\int_{\mathbb{R}^6}\varphi_{t_1,R}(r,z)\varphi_{t_2,R}(r,z')g(z-z') dzdz'\nonumber\\
    &\lesssim_T \Theta(T,n)^2\int_0^{t_1 \land  t_2}dr\int_{\mathbb{R}^6}\varphi_{t_2,R}(r,z')g(z-z') dzdz'\nonumber\\
    &\lesssim_T R^3.\label{R_21 estimate}
\end{align}
For the term $\mathbf{R_{22}}$, we apply Peetre's inequality to get that
\begin{align*}
    \mathbf{R_{22}} &\lesssim \Theta(T,n)^2\int_0^{t_1 \land  t_2}dr\int_{\mathbb{R}^6}\varphi_{t_1,R}(r,z)\varphi_{t_2,R}(r,z')\langle z-z' \rangle^{-2\beta}dzdz'\\
    &=\Theta(T,n)^2\int_0^{t_1 \land  t_2}dr\int_{\mathbb{R}^6}G(t_1-r,dx)G(t_2-r,dx')\int_{B^2_R}\langle x-x'+z-z' \rangle^{-2\beta}dzdz'\\
    &\lesssim_T \left(\int_{\mathbb{R}^3}\langle x \rangle^{2\beta}G(t_1-r, dx)\right)\left(\int_{\mathbb{R}^3}\langle x' \rangle^{2\beta}G(t_2-r, dx')\right)\int_{{B^2_R}}\langle z-z' \rangle^{-2\beta}dzdz'\\
    &\lesssim_{T,\beta} \int_{{B^2_R}}\langle z-z' \rangle^{-2\beta}dzdz',
\end{align*}
where the last step follows since the support of $G$ is compact (see \eqref{fundamental solution}).
Then, using the simple estimate
\begin{align}
\label{gamma ball order alpha}
    \int_{{B^2_R}}\langle x-y \rangle^{\alpha}dxdy &\lesssim 
    \begin{cases}
        R^{6+\alpha}  &(\alpha > -3)\\
        R^3\log R  &(\alpha = -3)\\
        R^3  &(\alpha < -3)
    \end{cases},
\end{align}
we obtain 
\begin{align}
    \mathbf{R_{22}} \lesssim \Theta(T,n)^2 \times
    \begin{cases}
        R^{6-2\beta}  &(0<\beta < \frac{3}{2})\\
        R^3\log R  &(\beta = \frac{3}{2})\\
        R^3  &(\beta > \frac{3}{2})
    \end{cases}. \label{R_22 estimate}
\end{align}
Combining \eqref{R_1 estimate}, \eqref{R_21 estimate}, and \eqref{R_22 estimate}, we obtain \eqref{varphi cov order}. Thus Lemma \ref{Lemma varphi cov order} holds.
\end{proof}

In the next lemma, \eqref{1 limit} is proved in \cite[Lemma 2.2]{Averaging2dSWE} when the spatial dimension is two. Exactly the same arguments also work in spatial dimension three.
Recall that $\tau_{\beta}$ is defined in \eqref{notation a}.
\begin{Lem}
\label{Lemma riesz limit}
Let $\gamma(x) = |x|^{-\beta}$ for some $0<\beta<2$. Then, 
\begin{align}
    &\frac{1}{R^{6-\beta}}\int_{\mathbb{R}^{6}}\varphi_{t_1,R}(r,y)\varphi_{t_2,R}(r,z)|y-z|^{-\beta}dydz \xrightarrow{R \to \infty} \tau_{\beta}(t_1-r)(t_2-r), \label{1 limit}\\
    &\frac{1}{R^{6-\beta}}\int_{\mathbb{R}^{6}}\varphi_{n,t_1,R}(r,y)\varphi_{n,t_2,R}(r,z)|y-z|^{-\beta}dydz \xrightarrow{R \to \infty} \tau_{\beta}(t_1-r)(t_2-r). \label{2 limit}
\end{align}
\end{Lem}
\begin{proof}
We follow the arguments in \cite[Lemma 2.2]{Averaging2dSWE}. Recall that the tempered measure $|x|^{-\beta}dx$ is the Fourier transform of its spectral measure $c_{\beta}|\xi|^{\beta -3}d\xi$.
Taking into account Lemma \ref{Lemma fourier transform}, we can use the Fourier transform to obtain
\begin{align*}
    &\int_{\mathbb{R}^{6}}\varphi_{t_1,R}(r,y)\varphi_{t_2,R}(r,z)|y-z|^{-\beta}dydz = c_{\beta} \int_{\mathbb{R}^3}|\mathcal{F} \mathbf{1}_{B_R}(\xi)|^2 \mathcal{F} G(t_1-r)(\xi) \mathcal{F} G(t_2-r)(\xi)|\xi|^{\beta -3}d\xi\\
    &= c_{\beta}R^{-\beta} \int_{\mathbb{R}^3}\left|\mathcal{F} \mathbf{1}_{B_R}\left(\frac{\xi}{R}\right)\right|^2 \mathcal{F} G(t_1-r)\left(\frac{\xi}{R}\right) \mathcal{F} G(t_2-r)\left(\frac{\xi}{R}\right)|\xi|^{\beta -3}d\xi\\
    &= c_{\beta}R^{6-\beta} \int_{\mathbb{R}^3}\left|\mathcal{F} \mathbf{1}_{B_1}\left(\xi\right)\right|^2 \mathcal{F} G(t_1-r)\left(\frac{\xi}{R}\right) \mathcal{F} G(t_2-r)\left(\frac{\xi}{R}\right)|\xi|^{\beta -3}d\xi.
\end{align*}
Since \eqref{FG bound} holds and
\begin{align*}
    c_{\beta}\int_{\mathbb{R}^3}|\mathcal{F} \mathbf{1}_{B_1}\left(\xi\right)|^2|\xi|^{\beta -3}d\xi = \int_{B^2_1} |x-y|^{-\beta} dxdy = \tau_{\beta} < \infty
\end{align*}
by Lemma \ref{Lemma fourier transform}, we apply the Lebesgue dominated convergence theorem to get 
\begin{align*}
    &\frac{1}{R^{6-\beta}}\int_{\mathbb{R}^{6}}\varphi_{t_1,R}(r,y)\varphi_{t_2,R}(r,z)|y-z|^{-\beta}dydz\\
    &= c_{\beta} \int_{\mathbb{R}^3}\left|\mathcal{F} \mathbf{1}_{B_1}\left(\xi\right)\right|^2 \mathcal{F} G(t_1-r)\left(\frac{\xi}{R}\right) \mathcal{F} G(t_2-r)\left(\frac{\xi}{R}\right)|\xi|^{\beta -3}d\xi\\
    &\xrightarrow{R \to \infty} c_{\beta} (t_1-r) (t_2-r) \int_{\mathbb{R}^3}\left|\mathcal{F} \mathbf{1}_{B_1}\left(\xi\right)\right|^2 |\xi|^{\beta -3}d\xi = \tau_{\beta}(t_1-r)(t_2-r),
\end{align*}
and proves \eqref{1 limit}. \eqref{2 limit} follows by the same arguments.
Thus Lemma \ref{Lemma riesz limit} holds.
\end{proof}

By Lemmas \ref{Lemma varphi cov order} and \ref{Lemma riesz limit}, we can prove the following proposition.
\begin{Prop}
\label{Proposition Riesz limit}
Let $\gamma(x) = |x|^{-\beta}$ for some $0<\beta<2$. 
Then, for any $t_1, t_2 \in [0,T]$, 
\begin{align}
    \lim_{R \to \infty} R^{ \beta -6}\mathbb{E}[F_{n,R}(t_1)F_{n,R}(t_2)] &= \tau_{\beta} \int_0^{t_1 \land t_2}(t_1-r)(t_2-r)\eta_{n-1}^2(r)dr, \label{riesz n limit}\\
    \lim_{R \to \infty} R^{\beta-6}\mathbb{E}[F_{R}(t_1)F_{R}(t_2)] &= \tau_{\beta} \int_0^{t_1 \land t_2}(t_1-r)(t_2-r)\eta^2(r)dr. \label{riesz limit}
\end{align}
Here $\eta_{n}(r) = \mathbb{E}[\sigma(u_n(r,0))]$ and $\eta(r) = \mathbb{E}[\sigma(u(r,0))]$.
\end{Prop}
\begin{proof}
By Lemmas \ref{varphi property} and \ref{Lemma varphi cov order}, we have
\begin{align}
     &R^{\beta-6} \int_0^{t_1 \land t_2}dr\int_{\mathbb{R}^6}\varphi_{n,t_1,R}(r,z)\varphi_{n,t_2,R}(r,z')\gamma(z-z')\mathrm{Cov}(\sigma(u_{n-1}(r,z)), \sigma(u_{n-1}(r,z')))dzdz' \nonumber\\
    &\leqslant R^{\beta-6} \int_0^{t_1 \land t_2}dr\int_{\mathbb{R}^6}\varphi_{t_1,R+1}(r,z)\varphi_{t_2,R+1}(r,z')\gamma(z-z')\left|\mathrm{Cov}(\sigma(u_{n-1}(r,z)), \sigma(u_{n-1}(r,z')))\right|dzdz' \nonumber\\
    &\lesssim \Theta(T,n-1)^2 R^{\beta - 6} \times
    \begin{cases}
        (R+1)^{6-2\beta}  &(0<\beta < \frac{3}{2})\\
        (R+1)^3\log (R+1)  &(\beta = \frac{3}{2})\\
        (R+1)^3  &(\frac{3}{2}< \beta <2)
    \end{cases} \label{a2}\\
    &\xrightarrow{R \to \infty} 0. \nonumber
\end{align}
Hence we get from strict stationarity of $u_n$ and \eqref{2 limit} that
\begin{align*}
    &R^{ \beta -6}\mathbb{E}[F_{n,R}(t_1)F_{n,R}(t_2)]\\
    &= R^{\beta-6} \int_0^{t_1 \land t_2}dr\int_{\mathbb{R}^6}\varphi_{n,t_1,R}(r,z)\varphi_{n,t_2,R}(r,z')\gamma(z-z')\mathbb{E}[\sigma(u_{n-1}(r,z))\sigma(u_{n-1}(r,z'))]\\
    &= R^{\beta-6} \int_0^{t_1 \land t_2}dr\int_{\mathbb{R}^6}\varphi_{n,t_1,R}(r,z)\varphi_{n,t_2,R}(r,z')\gamma(z-z')\mathrm{Cov}(\sigma(u_{n-1}(r,z)), \sigma(u_{n-1}(r,z')))dzdz'\\
    &\quad + R^{\beta-6} \int_0^{t_1 \land t_2}dr\eta_{n-1}^2(r)\int_{\mathbb{R}^6}\varphi_{n,t_1,R}(r,z)\varphi_{n,t_2,R}(r,z')\gamma(z-z')dzdz'\\
    &\xrightarrow{R \to \infty} \tau_{\beta} \int_0^{t_1 \land t_2}(t_1-r)(t_2-r)\eta_{n-1}^2(r)dr 
\end{align*}
and \eqref{riesz n limit} is proved. 
Next we prove \eqref{riesz limit}. Taking into account \eqref{1 limit}, we see that 
\begin{align}
    &R^{\beta-6} \int_0^{t_1 \land t_2}dr\eta^2(r)\int_{\mathbb{R}^6}\varphi_{t_1,R}(r,z)\varphi_{t_2,R}(r,z')\gamma(z-z')dzdz' \nonumber\\
    &\xrightarrow{R \to \infty} \tau_{\beta} \int_0^{t_1 \land t_2}(t_1-r)(t_2-r)\eta^2(r)dr.\label{a3}
\end{align}
Moreover, by \eqref{iteration uniform bound}, \eqref{iteration uniform convergence}, and the Lipschitz continuity of $\sigma$, we see that 
\begin{align*}
    \lim_{n\to \infty}\sup_{t \in [0,T]} \sup_{(x,x') \in \mathbb{R}^6} |\mathrm{Cov}(\sigma(u(t,x)),\sigma(u(t,x'))) - \mathrm{Cov}(\sigma(u_n(t,x)),\sigma(u_n(t,x')))| = 0.
\end{align*}
Therefore, for any $\varepsilon>0$, we can take $n$ large enough to obtain
\begin{align}
    &R^{\beta-6} \int_0^{t_1 \land t_2}dr\int_{\mathbb{R}^6}\varphi_{t_1,R}(r,z)\varphi_{t_2,R}(r,z')\gamma(z-z')|\mathrm{Cov}(\sigma(u(r,z)), \sigma(u(r,z')))|dzdz' \nonumber\\
    &\leqslant R^{\beta-6} \int_0^{t_1 \land t_2}dr\int_{\mathbb{R}^6}\varphi_{t_1,R}(r,z)\varphi_{t_2,R}(r,z')\gamma(z-z')|\mathrm{Cov}(\sigma(u_{n}(r,z)), \sigma(u_{n}(r,z')))|dzdz' \nonumber\\
    &\quad + \varepsilon R^{\beta-6} \int_0^{t_1 \land t_2}dr\int_{\mathbb{R}^6}\varphi_{t_1,R}(r,z)\varphi_{t_2,R}(r,z')\gamma(z-z')dzdz'. \label{a1}
\end{align}
In \eqref{a1}, the first term on the right-hand side is bounded by \eqref{a2} and converges to 0 as $R\to \infty$.
Furthermore, it follows from \eqref{riesz bound} and Lemma \ref{varphi property} that
\begin{align*}
    &\varepsilon R^{\beta-6} \int_0^{t_1 \land t_2}dr\int_{\mathbb{R}^6}\varphi_{t_1,R}(r,z)\varphi_{t_2,R}(r,z')\gamma(z-z')dzdz'\\
    &\lesssim \varepsilon R^{\beta-6} \int_0^{t_1 \land t_2}dr\int_{\mathbb{R}^6}\varphi_{t_1,R}(r,z)\varphi_{t_2,R}(r,z')\gamma_0(z-z')dzdz'\\
    &\quad + \varepsilon R^{\beta-6} \int_0^{t_1 \land t_2}dr\int_{\mathbb{R}^6}\varphi_{t_1,R}(r,z)\varphi_{t_2,R}(r,z')\langle z-z' \rangle^{-\beta}dzdz'\\
    &\lesssim \varepsilon R^{\beta-6} \int_0^{t_1 \land t_2}dr\int_{\mathbb{R}^6}\varphi_{t_1,R}(r,z)\varphi_{t_2,R}(r,z')\gamma_0(z-z')dzdz'\\
    &\quad + \varepsilon R^{\beta-6} \int_{B_{R+T}^2}
    \langle z-z' \rangle^{-\beta}dzdz'.
\end{align*}
Hence, by \eqref{R_1 estimate} and \eqref{gamma ball order alpha}, we have for any $R \geqslant 2 \lor T$ that 
\begin{align*}
    &\varepsilon R^{\beta-6} \int_0^{t_1 \land t_2}dr\int_{\mathbb{R}^6}\varphi_{t_1,R}(r,z)\varphi_{t_2,R}(r,z')\gamma(z-z')dzdz' \lesssim \varepsilon,
\end{align*}
where the implicit constant is independent of $R$.
Thus we obtain 
\begin{align}
\label{a4}
    R^{\beta-6} \int_0^{t_1 \land t_2}dr\int_{\mathbb{R}^6}\varphi_{t_1,R}(r,z)\varphi_{t_2,R}(r,z')\gamma(z-z')|\mathrm{Cov}(\sigma(u(r,z)), \sigma(u(r,z')))|dzdz'  \xrightarrow{R \to \infty} 0.
\end{align}
Combining \eqref{a3} and \eqref{a4}, we get 
\begin{align*}
    &R^{ \beta -6}\mathbb{E}[F_{R}(t_1)F_{R}(t_2)]\\
    &= R^{\beta-6} \int_0^{t_1 \land t_2}dr\int_{\mathbb{R}^6}\varphi_{t_1,R}(r,z)\varphi_{t_2,R}(r,z')\gamma(z-z')\mathbb{E}[\sigma(u(r,z))\sigma(u(r,z'))]\\
    &= R^{\beta-6} \int_0^{t_1 \land t_2}dr\int_{\mathbb{R}^6}\varphi_{t_1,R}(r,z)\varphi_{t_2,R}(r,z')\gamma(z-z')\mathrm{Cov}(\sigma(u(r,z)), \sigma(u(r,z')))dzdz'\\
    &\quad + R^{\beta-6} \int_0^{t_1 \land t_2}dr\eta^2(r)\int_{\mathbb{R}^6}\varphi_{t_1,R}(r,z)\varphi_{t_2,R}(r,z')\gamma(z-z')dzdz'\\
    &\xrightarrow{R \to \infty} \tau_{\beta} \int_0^{t_1 \land t_2}(t_1-r)(t_2-r)\eta^2(r)dr, 
\end{align*}
which proves \eqref{riesz limit}.
\end{proof}

\subsection{Convergence of finite-dimensional distributions}
\label{section Convergence of finite dimensional distributions}
Taking into account Propositions \ref{Proposition L1 limit} and \ref{Proposition Riesz limit}, we set
\begin{align*}
    \Phi_{i,j} \coloneqq |B_1| \int_{\mathbb{R}^3}\mathrm{Cov}(u(t_i,x),u(t_j,0))dx, \quad \Phi^n_{i,j} \coloneqq |B_1| \int_{\mathbb{R}^3}\mathrm{Cov}(u_n(t_i,x),u_n(t_j,0))dx,\\
    \Psi_{i,j} \coloneqq \tau_{\beta} \int_0^{t_i \land t_j}(t_i-r)(t_j-r)\eta^2(r)dr, \quad \Psi^n_{i,j} \coloneqq \tau_{\beta} \int_0^{t_i \land t_j}(t_i-r)(t_j-r)\eta_{n-1}^2(r)dr.
\end{align*}
From \eqref{iteration uniform convergence} and the Lebesgue dominated convergence theorem, it is easy to check that 
\begin{align}
\label{Psi n limit}
    \lim_{n \to \infty} \Psi^n_{i,j} = \Psi_{i,j}. 
\end{align}
For $\Phi_{i,j}$ and $\Phi^n_{i,j}$, we have the following result.

\begin{Lem}
\label{Lemma Phi convergence}
When $\gamma \in L^1(\mathbb{R}^3)$, then,
\begin{align}
\label{Phi n limit}
    \lim_{n \to \infty} \Phi^n_{i,j} = \Phi_{i,j}.
\end{align}
\end{Lem}

\begin{proof}
From \eqref{iteration uniform bound} and \eqref{iteration uniform convergence}, we have
\begin{gather*}
    \lim_{n \to \infty}\mathrm{Cov}(u_n(t_i,x),u_n(t_j,0)) = \mathrm{Cov}(u(t_i,x),u(t_j,0)).
\end{gather*}
Hence it is sufficient to justify interchanging the limit and the integral.
For this purpose, we set 
\begin{align*}
    g^n_{i,j}(x) &\coloneqq \int_0^{t_i \land t_j}dr\int_{\mathbb{R}^6}\gamma(x-z-z')G_n(t_i -r, z)G_n(t_j -r, z')dzdz',\\
    g_{i,j}(x) &\coloneqq \int_0^{t_i \land t_j}dr\int_{\mathbb{R}^6}\gamma(x-z-z')G(t_i-r,dz)G(t_j -r, dz').
\end{align*}
Using the Fourier transform and similar arguments in \cite[Lemma 6.5]{swe3holder}, we can show that 
\begin{align*}
    g^n_{i,j}(x) &= \int_0^{t_i \land t_j}dr\int_{\mathbb{R}^3}\mu(d\xi)e^{-2\pi \sqrt{-1} x \cdot \xi}\mathcal{F} G(t_i-r)(\xi) \mathcal{F} G(t_j -r)(\xi)|\mathcal{F} \rho_n(\xi)|^2, \\
    g_{i,j}(x) &= \int_0^{t_i \land t_j}dr\int_{\mathbb{R}^3}\mu(d\xi)e^{-2\pi \sqrt{-1} x \cdot \xi}\mathcal{F} G(t_i-r)(\xi) \mathcal{F} G(t_j -r)(\xi).
\end{align*}
Because 
\begin{align*}
    \int_0^{t_i \land t_j}dr\int_{\mathbb{R}^3}\mu(d\xi)|\mathcal{F} G(t_i-r)(\xi)| |\mathcal{F} G(t_j -r)(\xi)| 
    \lesssim_T \int_{\mathbb{R}^3}\langle \xi \rangle^{-2}\mu(d\xi) < \infty,
\end{align*}
we conclude from the Lebesgue dominated convergence theorem that
\begin{align}
    \lim_{n \to \infty} g^n_{i,j}(x) &= \int_0^{t_i \land t_j}dr\int_{\mathbb{R}^3}\mu(d\xi)e^{-2\pi \sqrt{-1} x \cdot \xi}\mathcal{F} G(t_i-r)(\xi) \mathcal{F} G(t_j -r)(\xi) \nonumber\\
    &= g_{i,j}(x) \label{b1}
\end{align}
Furthermore, since
\begin{align*}
    |\mathrm{Cov}(u_n(t_i,x),u_n(t_j,0))| &\leqslant \int_0^{t_i \land t_j}dr\int_{\mathbb{R}^6}dzdz'\gamma(z-z')G_n(t_i -r, x-z)G_n(t_j -r, -z')\\
    &\qquad \quad \times |\mathbb{E}[\sigma(u_{n-1}(r,z))\sigma(u_{n-1}(r,z'))]|\\
    &\lesssim \int_0^{t_i \land t_j}dr\int_{\mathbb{R}^6}\gamma(x-z-z')G_n(t_i -r, z)G_n(t_j -r, z')dzdz'
\end{align*}
by \eqref{iteration uniform bound}, we have  
\begin{align}
\label{b2}
    |\mathrm{Cov}(u_n(t_i,x),u_n(t_j,0))| \lesssim g^n_{i,j}(x).
\end{align}
Furthermore, we have 
\begin{align}
\label{b3}
    \int_{\mathbb{R}^3}g^n_{i,j}(x)dx = \lVert \gamma \rVert_{L^1(\mathbb{R}^3)}\int_0^{t_i \land t_j}(t_i-r)(t_j-r)dr = \int_{\mathbb{R}^3}g_{i,j}(x)dx.
\end{align}
By \eqref{b1},\eqref{b2}, and \eqref{b3}, we apply a generalization of the Lebesgue dominated convergence theorem (see \cite[Theorem 2.8.8.]{MR2267655}) to obtain
\begin{align*}
    \lim_{n \to \infty} \int_{\mathbb{R}^3}\mathrm{Cov}(u_n(t_i,x),u_n(t_j,0))dx = \int_{\mathbb{R}^3}\mathrm{Cov}(u(t_i,x),u(t_j,0))dx,
\end{align*}
which proves the theorem.
\end{proof}

Here is the main result in this subsection.
\begin{Prop}
\label{Proposition fdd convergence}
Fix $m\geqslant1$, and let $\mathbf{F}_R = (F_R(t_1),..., F_R(t_m))$ for any $t_1,...,t_m \in [0,T]$.
\begin{enumerate}
    \item[(1)] Suppose $\gamma \in L^1(\mathbb{R}^3)$ and let $\mathbf{G_1} = (\mathcal{G}_1(t_1),...,\mathcal{G}_1(t_m))$ be a centered Gaussian random vector with covariance matrix $(\Phi_{i,j})_{1 \leqslant i,j \leqslant m}$.
    Then, as $R \to \infty$,
    \begin{align}
        R^{-\frac{3}{2}}\mathbf{F}_R \xrightarrow{d} \mathbf{G_1}. \label{d3}
    \end{align}
    \item[(2)] Suppose $\gamma(x) = |x|^{-\beta}$ for some $0< \beta <2$ and let $\mathbf{G_2} = (\mathcal{G}_2(t_1),...,\mathcal{G}_2(t_m))$ be a centered Gaussian random vector with covariance matrix $(\Psi_{i,j})_{1 \leqslant i,j \leqslant m}$. Then, as $R \to \infty$,
    \begin{align}
        R^{\frac{\beta}{2}-3}\mathbf{F}_R \xrightarrow{d} \mathbf{G_2}. \label{d4}
    \end{align}
\end{enumerate}
Here ``$\xrightarrow{d}$" denotes the convergence in law.
\end{Prop}
\begin{proof}
First we consider the case $\gamma \in L^1(\mathbb{R}^3)$. It is sufficient to show that for any $h \in C_0^{\infty}(\mathbb{R}^m ; \mathbb{R})$, 
\begin{align*}
    |\mathbb{E}[h(R^{-\frac{3}{2}}\mathbf{F}_R) - h(\mathbf{G_1})]| \xrightarrow{R \to \infty} 0.
\end{align*}
Let $\mathbf{F}_{n,R} = (F_{n,R}(t_1),..., F_{n,R}(t_m))$. Then, we have 
\begin{align}
    |\mathbb{E}[h(R^{-\frac{3}{2}}\mathbf{F}_R) - h(\mathbf{G_1})]| &\leqslant |\mathbb{E}[h(R^{-\frac{3}{2}}\mathbf{F}_R) - h(R^{-\frac{3}{2}}\mathbf{F}_{n,R})]|+ |\mathbb{E}[h(R^{-\frac{3}{2}}\mathbf{F}_{n,R}) - h(\mathbf{G_1})]| \nonumber\\
    &\leqslant \sup_{R>0} \left(\lVert \nabla h\rVert_{\infty}R^{-\frac{3}{2}}\mathbb{E}[|\mathbf{F}_R - \mathbf{F}_{n,R}|]\right) + |\mathbb{E}[h(R^{-\frac{3}{2}}\mathbf{F}_{n,R}) - h(\mathbf{G_1})]| \label{rhs},
\end{align}
where 
\begin{align*}
    \lVert \nabla h \rVert _{\infty} = \max_{1 \leqslant i \leqslant m} \sup_{x \in \mathbb{R}^m} \left| \frac{\partial h}{\partial x_i}(x) \right|.
\end{align*}
Given $\varepsilon >0$, from \eqref{c1} and \eqref{c2} we can estimate the first term on the right-hand side of \eqref{rhs} by
\begin{align*}
    &\lVert \nabla h\rVert_{\infty}R^{-\frac{3}{2}}\mathbb{E}[|\mathbf{F}_R - \mathbf{F}_{n,R}|] \lesssim R^{-\frac{3}{2}}\sum_{i = 1}^m \lVert F_R(t_i) - F_{n,R}(t_i)\rVert_2 \\
    &\lesssim \left( \sup_{(\theta,\eta) \in [0,T]\times\mathbb{R}^3}\lVert u_n(\theta,\eta) - u(\theta,\eta)\rVert_2 + 2\sqrt{\varepsilon} + \sup_{x \in K_{\varepsilon}} \left|\mathcal{F} \rho\left(\frac{x}{a_n}\right) - 1\right|  \right) \left(R^{-3}\int_{{B^2_R}}\gamma(z-z')dzdz'\right)^{\frac{1}{2}}\\
    &\lesssim \left( \sup_{(\theta,\eta) \in [0,T]\times\mathbb{R}^3}\lVert u_n(\theta,\eta) - u(\theta,\eta)\rVert_2 + 2\sqrt{\varepsilon} + \sup_{x \in K_{\varepsilon}} \left|\mathcal{F} \rho\left(\frac{x}{a_n}\right) - 1\right|  \right),
\end{align*}
and thus, by \eqref{iteration uniform convergence}, we can find $n$ large enough such that 
\begin{align}
\label{d1}
    \sup_{R>0} \left(\lVert \nabla h\rVert_{\infty}R^{-\frac{3}{2}}\mathbb{E}[|\mathbf{F}_R - \mathbf{F}_{n,R}|]\right) \lesssim \varepsilon.
\end{align}
Therefore, by \eqref{Phi n limit} we can find $n$ large enough such that \eqref{d1} holds true and 
\begin{align*}
    \max_{1 \leqslant i,j \leqslant m}|\Phi^n_{i,j} - \Phi_{i,j}|^2 < \varepsilon.
\end{align*}
From now on, we fix such an integer $n$.
Now we prove that 
\begin{align}
    |\mathbb{E}[h(R^{-\frac{3}{2}}\mathbf{F}_{n,R}) - h(\mathbf{G_1})]| \xrightarrow{R \to \infty} 0. \label{d2}
\end{align}
Since \eqref{fn delta} holds, we apply Proposition \ref{Proposition multivariate stein bound} to obtain
\begin{align*}
    |\mathbb{E}[h(R^{-\frac{3}{2}}\mathbf{F}_{n,R}) - h(\mathbf{G_1})]| \leqslant \frac{m}{2}\lVert \nabla^2 h\rVert_{\infty}\sqrt{\sum_{i,j}^m \mathbb{E}[|\Phi_{i,j} - R^{-3}\langle DF_{n,R}(t_i), V_{n,t_j,R}\rangle_{\mathcal{H}_T}|^2]}, 
\end{align*}
and thus it suffices to show that
\begin{align*}
    \lim_{R \to \infty} \mathbb{E}[|\Phi_{i,j} - R^{-3}\langle DF_{n,R}(t_i), V_{n,t_j,R}\rangle_{\mathcal{H}_T}|^2] = 0.
\end{align*}
Since, by the duality formula \eqref{duality formula} and Proposition \ref{Proposition L1 limit},
\begin{align*}
    R^{-3} \mathbb{E}[\langle DF_{n,R}(t_i), V_{n,t_j,R}\rangle_{\mathcal{H}_T}] = R^{-3} \mathbb{E}[F_{n,R}(t_i)F_{n,R}(t_j)] \xrightarrow{R \to \infty} \Phi^n_{i,j},
\end{align*}
 we deduce from Lemma \ref{key} that
\begin{align*}
    &\mathbb{E}[|\Phi_{i,j} - R^{-3}\langle DF_{n,R}(t_i), V_{n,t_j,R}\rangle_{\mathcal{H}_T}|^2] \\
    &= \Phi_{i,j}^2 - 2 \Phi_{i,j}\left(R^{-3} \mathbb{E}[\langle DF_{n,R}(t_i), V_{n,t_j,R}\rangle_{\mathcal{H}_T}]\right)\\
    &\quad + R^{-6} \left(\mathrm{Var}(\langle DF_{n,R}(t_i), V_{n,t_j,R}\rangle_{\mathcal{H}_T}) + \mathbb{E}[\langle DF_{n,R}(t_i), V_{n,t_j,R}\rangle_{\mathcal{H}_T}]^2\right) \\
    &\xrightarrow[]{R \to \infty} |\Phi_{i,j} - \Phi^n_{i,j}|^2 < \varepsilon,
\end{align*}
and \eqref{d2} is proved. 
Combining \eqref{d1} and \eqref{d2}, we obtain \eqref{d3}.

Next we consider the case $\gamma(x) = |x|^{-\beta}$.
Similar arguments also work for this case and we only sketch the proof.
For any $h \in C_0^{\infty}(\mathbb{R}^m ; \mathbb{R})$, we have
\begin{align*}
    |\mathbb{E}[h(R^{\frac{\beta}{2}-3}\mathbf{F}_R) - h(\mathbf{G_2})]| \leqslant \sup_{R>0} \left(\lVert \nabla h\rVert_{\infty}R^{\frac{\beta}{2}-3}\mathbb{E}[|\mathbf{F}_R - \mathbf{F}_{n,R}|]\right) + |\mathbb{E}[h(R^{\frac{\beta}{2}-3}\mathbf{F}_{n,R}) - h(\mathbf{G_2})]|.
\end{align*}
Given $\varepsilon >0$, the first term above can be estimated by 
\begin{align*}
    &\lVert \nabla h\rVert_{\infty}R^{\frac{\beta}{2}-3}\mathbb{E}[|\mathbf{F}_R - \mathbf{F}_{n,R}|]\\
    &\lesssim \left( \sup_{(\theta,\eta) \in [0,T]\times\mathbb{R}^3}\lVert u_n(\theta,\eta) - u(\theta,\eta)\rVert_2 + 2\sqrt{\varepsilon} + \sup_{x \in K_{\varepsilon}} \left|\mathcal{F} \rho\left(\frac{x}{a_n}\right) - 1\right|  \right) \left(R^{\beta -6}\int_{{B^2_R}}|z-z'|^{-\beta}dzdz'\right)^{\frac{1}{2}}\\
    &\lesssim \left( \sup_{(\theta,\eta) \in [0,T]\times\mathbb{R}^3}\lVert u_n(\theta,\eta) - u(\theta,\eta)\rVert_2 + 2\sqrt{\varepsilon} + \sup_{x \in K_{\varepsilon}} \left|\mathcal{F} \rho\left(\frac{x}{a_n}\right) - 1\right|  \right).
\end{align*}
Therefore,  by \eqref{iteration uniform convergence} and \eqref{Psi n limit}, we can find $n$ large enough such that 
\begin{align*}
    \sup_{R>0} \left(\lVert \nabla h\rVert_{\infty}R^{\frac{\beta}{2}-3}\mathbb{E}[|\mathbf{F}_R - \mathbf{F}_{n,R}|]\right) \lesssim \varepsilon, \quad 
    \max_{1 \leqslant i,j \leqslant m}|\Psi^n_{i,j} - \Psi_{i,j}|^2 < \varepsilon.
\end{align*}
Fix such an integer $n$. In view of Proposition \ref{Proposition multivariate stein bound}, we are reduced to proving 
\begin{align*}
    \lim_{R \to \infty} \mathbb{E}[|\Psi_{i,j} - R^{\beta-6}\langle DF_{n,R}(t_i), V_{n,t_j,R}\rangle_{\mathcal{H}_T}|^2] = 0.
\end{align*}
We can derive this convergence by Proposition \ref{Proposition Riesz limit} and Lemma \ref{key}, and thus \eqref{d4} follows.
This completes the proof.
\end{proof}

\subsection{Tightness}
\label{section Tightness}
Let us now prove the tightness.
Thanks to the tightness criterion of Kolmogorov-Chentsov (see \textit{e.g.}  \cite[Theorem 23.7]{MR4226142}), we only need to show the following moment estimates.

\begin{Prop}
\label{Proposition tightness}
For any $0 \leqslant s < t \leqslant T$, the following results hold.
\begin{enumerate}
    \item[(1)] When $\gamma \in L^1(\mathbb{R}^3)$, then, 
    \begin{align*}
        \mathbb{E}\left[\left|\frac{1}{R^{3/2}}F_R(t) - \frac{1}{R^{3/2}}F_R(s)\right|^2\right] \lesssim_{T} (t-s)^{2}.
    \end{align*}
    \item[(2)] When $\gamma(x) = |x|^{-\beta}$ for some $0< \beta <2$, then,
    \begin{align*}
        \mathbb{E}\left[\left|\frac{1}{R^{3-\beta/2}}F_R(t) - \frac{1}{R^{3-\beta/2}}F_R(s)\right|^2\right] \lesssim_{T} (t-s)^{2}.
    \end{align*}
\end{enumerate}
\end{Prop}
\begin{proof}
Using the isometry property of stochastic integral, we have
\begin{align*}
    &\mathbb{E}[|F_{R}(t)-F_{R}(s)|^{2}]
    = \mathbb{E}\left[\left(\int_{0}^{T} \int_{\mathbb{R}^{3}}\left(\varphi_{t,R}(r,y)-\varphi_{s, R}(r, y)\right) \sigma(u(r, y)) W(dr,dy)\right)^{2}\right]\\
    &= \int_{0}^{T} dr \int_{\mathbb{R}^{6}}(\varphi_{t, R}(r , y)-\varphi_{s, R}(r, y))(\varphi_{t, R}(r,z)-\varphi_{s, R}(r,z))\\
    & \qquad \qquad \qquad \times \mathbb{E}[\sigma(u(r, y)) \sigma(u(r , z))] \gamma(y-z)dy dz.
\end{align*}
In view of Lemma \ref{Lemma fourier transform}, we can use the Fourier transform to write
\begin{align*}
    &\int_{0}^{T} dr \int_{\mathbb{R}^{6}}(\varphi_{t, R}(r , y)-\varphi_{s, R}(r, y))(\varphi_{t, R}(r,z)-\varphi_{s, R}(r,z))\\
    & \qquad \qquad \qquad \times \mathbb{E}[\sigma(u(r, y)) \sigma(u(r , z))] \gamma(y-z)dy dz\\
    &= \int_0^Tdr\int_{\mathbb{R}^3}\mu_{r}^{\sigma(u)}(d\xi)|\mathcal{F}\varphi_{t,R}(r)(\xi) - \mathcal{F}\varphi_{s,R}(r)(\xi)|^2\\
    &= \int_s^tdr\int_{\mathbb{R}^3}\mu_{r}^{\sigma(u)}(d\xi)|\mathcal{F}\varphi_{t,R}(r)(\xi)|^2 + \int_0^sdr\int_{\mathbb{R}^3}\mu_{r}^{\sigma(u)}(d\xi)|\mathcal{F}\varphi_{t,R}(r)(\xi) - \mathcal{F}\varphi_{s,R}(r)(\xi)|^2\\
    &\eqqcolon \mathbf{T}_1 + \mathbf{T}_2.
\end{align*}
Since by \eqref{FG bound},
\begin{align*}
    |\mathcal{F}\varphi_{t,R}(r)(\xi)| = |\mathcal{F}\mathbf{1}_{B_R}(\xi)||\mathcal{F}G(t-r)(\xi)| \leqslant |\mathcal{F}\mathbf{1}_{B_R}(\xi)|(t-r),
\end{align*}
the term $\mathbf{T}_1$ is estimated by 
\begin{align}
    \mathbf{T}_1 
    &\leqslant \int_s^t dr (t-r)^{2}\int_{\mathbb{R}^3}\mu_{r}^{\sigma(u)}(d\xi)|\mathcal{F}\mathbf{1}_{B_R}(\xi)|^2\nonumber\\
    &= \int_s^t dr (t-r)^{2}\int_{{B^2_R}}\gamma(y-z)\mathbb{E}[\sigma(u(r, y-z)) \sigma(u(r , 0))]dydz\nonumber\\
    &\lesssim (t-s)^{3}\int_{{B^2_R}}\gamma(y-z)dydz. \label{t1}
\end{align}
Similarly, by \eqref{FG minus FG bound}, the term $\mathbf{T}_2$ is estimated by
\begin{align}
    \mathbf{T}_2 
    &\lesssim (t-s)^{2}\int_0^sdr\int_{\mathbb{R}^3}\mu_{r}^{\sigma(u)}(d\xi)|\mathcal{F}\mathbf{1}_{B_R}(\xi)|^2\nonumber\\
    &= (t-s)^{2}\int_0^sdr\int_{{B^2_R}}\gamma(y-z)\mathbb{E}[\sigma(u(r, y-z)) \sigma(u(r , 0))]dydz\nonumber\\
    &\lesssim_{T} (t-s)^{2}\int_{{B^2_R}}\gamma(y-z)dydz. \label{t2}
\end{align}
Hence we deduce from \eqref{gamma ball order}, \eqref{t1}, and \eqref{t2} that
\begin{align*}
    \mathbb{E}[|F_{R}(t)-F_{R}(s)|^{2}] &\lesssim_T (t-s)^{2}\int_{{B^2_R}}\gamma(y-z)dydz\\
    &\lesssim 
    \begin{cases}
        (t-s)^2R^3  &(\gamma \in L^1(\mathbb{R}^3))\\
        (t-s)^2R^{6-\beta}  &(\gamma(x) = |x|^{-\beta} \quad (0 < \beta <2))
    \end{cases}.
\end{align*}
This completes the proof.
\end{proof}

As mentioned at the beginning of Section \ref{FCLT}, the proof of Theorem \ref{main result2} is completed by combining Propositions \ref{Proposition fdd convergence} and \ref{Proposition tightness}.
w

\section{Appendix}
\label{Appendix}
In this section we gather some technical estimates needed for our proof.
Recall that $\langle x \rangle = \sqrt{1 + |x|^2}$.

We first record the useful inequality known as Peetre's inequality. 
See \cite[Lemma 34.34, p.738]{BKDriver} for the proof.
\begin{Lem}[Peetre's inequality]
\label{Peetre inequality}
For any $k \in \mathbb{R}$ and $x,y \in \mathbb{R}^3$, we have
\begin{align*}
    \langle x + y \rangle^k \leqslant 2^{\frac{|k|}{2}}\langle x \rangle^k\langle y \rangle^{|k|}.
\end{align*}
\end{Lem}

The following two lemmas are easy to check and we state without the proof.
\begin{Lem}
\label{p1}
Let $\alpha \in \mathbb{R}$ and $R \geqslant 2$. Then, we have
\begin{align*}
    \int_{B_R}\langle x \rangle^{\alpha} dx &\lesssim 
    \begin{cases}
        R^{3+\alpha}  &(\alpha > -3)\\
        \log R  &(\alpha = -3)\\
        1  &(\alpha < -3)
    \end{cases},
    \\
    \int_{{B^2_R}}\langle x-y \rangle^{\alpha}dxdy &\lesssim 
    \begin{cases}
        R^{6+\alpha}  &(\alpha > -3)\\
        R^3\log R  &(\alpha = -3)\\
        R^3  &(\alpha < -3)
    \end{cases},
    \\
    \int_{{B^3_R}}\langle x-y \rangle^{\alpha}\langle y-z \rangle^{\alpha}dxdydz &\lesssim 
    \begin{cases}
        R^{9+2\alpha}  &(\alpha > -3)\\
        R^3(\log R)^2  &(\alpha = -3)\\
        R^3  &(\alpha < -3)
    \end{cases}.
\end{align*}
\end{Lem}

\begin{Lem}
\label{p2}
For any $\alpha \in \mathbb{R}$, we have
\begin{align*}
    \sup_{t \in [0,T]}\int_{\mathbb{R}^3}\langle x \rangle^{\alpha}G(t,dx) < \infty,\\
    \sup_n\sup_{t \in [0,T]}\int_{\mathbb{R}^3}\langle x \rangle^{\alpha}G_n(t,dx) < \infty.
\end{align*}
\end{Lem}

Using Lemmas \ref{p1} and \ref{p2}, we can prove the following estimate which is needed for the proof of Lemma \ref{key}.
\begin{Lem}
\label{Lemma estimate}
Let $f,g \in L^1(\mathbb{R}^3)$ be nonnegative functions, $\alpha >0$, and $R\geqslant 2$. 
Suppose that nonnegative functions $\gamma$ and $F$ satisfy 
\begin{align}
    F(x) \lesssim \theta_1 f(x) + \theta_2 \langle x \rangle^{-\alpha}, \label{F cond}\\
    \gamma(x) \lesssim \theta_3 g(x) + \theta_4 \langle x \rangle^{-\alpha}, \label{gamma cond}
\end{align}
where $\theta_i$ is a constant such that $\theta_i \in \{0,1\}$ for all $1 \leqslant i \leqslant 4$. Then, for any $0 \leqslant s \leqslant r \leqslant T$, we have 
\begin{align}
    &\int_{\mathbb{R}^{12}}dydy'dzdz' 
    \varphi_{n,t_1,R}(r,y)\varphi_{n,t_2,R}(s,z)\gamma(y-z)
    \varphi_{n,t_1,R}(r,y')\varphi_{n,t_2,R}(s,z')\gamma(y'-z')F(z-z') \nonumber\\
    &\lesssim 
    \begin{cases}
        \theta_1\theta_3R^3 + \theta_3(\theta_1\theta_4 + \theta_2)R^{6-\alpha} + \theta_4(\theta_1+\theta_2\theta_3)R^{9-2\alpha} + \theta_2\theta_4R^{12-3\alpha}  &(0 < \alpha < 3)\\
        \theta_1\theta_3R^3 + \theta_3(\theta_1\theta_4 + \theta_2)R^3(\log R)+ \theta_4(\theta_1+\theta_2\theta_3)R^3(\log R)^2 + \theta_2\theta_4 R^3(\log R)^3  &(\alpha = 3)\\
        (\theta_1 + \theta_2)(\theta_3 + \theta_3\theta_4 + \theta_4)R^3  &(\alpha > 3)
    \end{cases}. \label{order}
\end{align}
\end{Lem}
\begin{proof}
By \eqref{F cond}, we have
\begin{align*}
    &\int_{\mathbb{R}^{12}}dydy'dzdz' 
    \varphi_{n,t_1,R}(r,y)\varphi_{n,t_2,R}(s,z)\gamma(y-z)
    \varphi_{n,t_1,R}(r,y')\varphi_{n,t_2,R}(s,z')\gamma(y'-z')F(z-z')\\
    &\lesssim \theta_1\int_{\mathbb{R}^{12}}dydy'dzdz'\int_{{B^4_R}}dxdx'dwdw'G_n(t_1-r,x-y)G_n(t_2-s,x'-z)\gamma(y-z)\\
    &\qquad \times G_n(t_1-r,w-y')G_n(t_2-s,w'-z')\gamma(y'-z')f(z-z')\\
    &\quad + \theta_2\int_{\mathbb{R}^{12}}dydy'dzdz'\int_{{B^4_R}}dxdx'dwdw'G_n(t_1-r,x-y)G_n(t_2-s,x'-z)\gamma(y-z)\\
    &\quad \qquad \times G_n(t_1-r,w-y')G_n(t_2-s,w'-z')\gamma(y'-z')\langle z-z' \rangle^{-\alpha}\\
    &\eqqcolon \theta_1\mathbf{A_1} + \theta_2\mathbf{A_2}.
\end{align*}
For the term $\mathbf{A_1}$, we have from \eqref{gamma cond} that
\begin{align*}
    \mathbf{A_1} &\lesssim \theta_3^2\int_{\mathbb{R}^{12}}dydy'dzdz'\int_{{B^4_R}}dxdx'dwdw'G_n(t_1-r,x-y)G_n(t_2-s,x'-z)g(y-z)\\
    &\qquad \times G_n(t_1-r,w-y')G_n(t_2-s,w'-z')g(y'-z')f(z-z')\\
    &\quad + \theta_3\theta_4\int_{\mathbb{R}^{12}}dydy'dzdz'\int_{{B^4_R}}dxdx'dwdw'G_n(t_1-r,x-y)G_n(t_2-s,x'-z)g(y-z)\\
    &\quad \qquad \times G_n(t_1-r,w-y')G_n(t_2-s,w'-z')\langle y'-z' \rangle^{-\alpha}f(z-z')\\
    &\quad +\theta_3\theta_4\int_{\mathbb{R}^{12}}dydy'dzdz'\int_{{B^4_R}}dxdx'dwdw'G_n(t_1-r,x-y)G_n(t_2-s,x'-z)\langle y-z \rangle^{-\alpha}\\
    &\quad \qquad \times G_n(t_1-r,w-y')G_n(t_2-s,w'-z')g(y'-z')f(z-z')\\
    &\quad +\theta_4^2\int_{\mathbb{R}^{12}}dydy'dzdz'\int_{{B^4_R}}dxdx'dwdw'G_n(t_1-r,x-y)G_n(t_2-s,x'-z)\langle y-z \rangle^{-\alpha}\\
    &\quad \qquad \times G_n(t_1-r,w-y')G_n(t_2-s,w'-z')\langle y'-z' \rangle^{-\alpha}f(z-z')\\
    &\eqqcolon \theta_3^2\mathbf{A_{11}}+\theta_3\theta_4\mathbf{A_{12}}+\theta_3\theta_4\mathbf{A_{13}}+\theta_4^2\mathbf{A_{14}}.
\end{align*}
Similarly, for the term $\mathbf{A_2}$, we have from \eqref{gamma cond} that
\begin{align*}
    \mathbf{A_2} &\lesssim \theta_3^2\int_{\mathbb{R}^{12}}dydy'dzdz'\int_{{B^4_R}}dxdx'dwdw'G_n(t_1-r,x-y)G_n(t_2-s,x'-z)g(y-z)\\
    &\qquad \times G_n(t_1-r,w-y')G_n(t_2-s,w'-z')g(y'-z')\langle z-z' \rangle^{-\alpha}\\
    &\quad +\theta_3\theta_4\int_{\mathbb{R}^{12}}dydy'dzdz'\int_{{B^4_R}}dxdx'dwdw'G_n(t_1-r,x-y)G_n(t_2-s,x'-z)g(y-z)\\
    &\quad \qquad \times G_n(t_1-r,w-y')G_n(t_2-s,w'-z')\langle y'-z' \rangle^{-\alpha}\langle z-z' \rangle^{-\alpha}\\
    &\quad +\theta_3\theta_4\int_{\mathbb{R}^{12}}dydy'dzdz'\int_{{B^4_R}}dxdx'dwdw'G_n(t_1-r,x-y)G_n(t_2-s,x'-z)\langle y-z \rangle^{-\alpha}\\
    &\quad \qquad \times G_n(t_1-r,w-y')G_n(t_2-s,w'-z')g(y'-z')\langle z-z' \rangle^{-\alpha}\\
    &\quad +\theta_4^2\int_{\mathbb{R}^{12}}dydy'dzdz'\int_{{B^4_R}}dxdx'dwdw'G_n(t_1-r,x-y)G_n(t_2-s,x'-z)\langle y-z \rangle^{-\alpha}\\
    &\quad \qquad \times G_n(t_1-r,w-y')G_n(t_2-s,w'-z')\langle y'-z' \rangle^{-\alpha}\langle z-z' \rangle^{-\alpha}\\
    &\eqqcolon \theta_3^2\mathbf{A_{21}}+\theta_3\theta_4\mathbf{A_{22}}+\theta_3\theta_4\mathbf{A_{23}}+\theta_4^2\mathbf{A_{24}}.
\end{align*}
First we consider $\mathbf{A_{11}}$. Integrating in the order $dw, dw', dy', dz', dx, dy,dz,dx'$, we get
\begin{align*}
    \mathbf{A_{11}} 
    &\lesssim_T
    \int_{\mathbb{R}^{12}}dydy'dzdz'\int_{{B^2_R}}dxdx'G_n(t_1-r,x-y)G_n(t_2-s,x'-z)g(y-z)\\
    &\qquad \times g(y'-z')f(z-z')\\
    &\lesssim
    \lVert g\rVert_{L^1(\mathbb{R}^3)}\lVert f\rVert_{L^1(\mathbb
    R)^3}\int_{\mathbb{R}^{6}}dydz\int_{{B^2_R}}dxdx'G_n(t_1-r,x-y)G_n(t_2-s,x'-z)g(y-z)\\
    &\lesssim_T R^3.
\end{align*}
Next we estimate $\mathbf{A_{12}}$. Integrating in the order $dx,dx',dy,dz$, we obtain
\begin{align*}
    \mathbf{A_{12}} 
    &\lesssim_T 
    \lVert g\rVert_{L^1(\mathbb{R}^3)}\lVert f\rVert_{L^1(\mathbb
    R)^3}\int_{\mathbb{R}^{6}}dy'dz'\int_{{B^2_R}}dwdw'G_n(t_1-r,w-y')G_n(t_2-s,w'-z')\langle y'-z' \rangle^{-\alpha}.
\end{align*}
Furthermore, Using Peetre's inequality and integrating with respect to $dy', dz'$, we get
\begin{align*}
    \mathbf{A_{12}} &\lesssim \int_{\mathbb{R}^{6}}dy'dz'\int_{{B^2_R}}dwdw'G_n(t_1-r,w-y')\langle w-y' \rangle^{\alpha}\\
    &\qquad \qquad \qquad \quad \times G_n(t_2-s,w'-z')\langle w'-z' \rangle^{\alpha}\langle w-w' \rangle^{-\alpha}\\
    &\lesssim_T 
    \int_{{B^2_R}}dwdw'\langle w-w' \rangle^{-\alpha}\\
    &\lesssim 
    \begin{cases}
        R^{6-\alpha}  &(0<\alpha < 3)\\
        R^3\log R  &(\alpha = 3)\\
        R^3  &(\alpha > 3 )
    \end{cases}.
\end{align*}
$\mathbf{A_{13}}$ can be estimated in the same way as $\mathbf{A_{12}}$. 
Now we consider the term $\mathbf{A_{14}}$. 
Applying Peetre's inequality and integrating with respect to $dy,dy'$ lead to 
\begin{align*}
    \mathbf{A_{14}} 
    &\lesssim 
    \int_{\mathbb{R}^{12}}dydy'dzdz'\int_{{B^4_R}}dxdx'dwdw'\\
    &\quad \qquad \times G_n(t_1-r,x-y)\langle x-y \rangle^{\alpha}G_n(t_2-s,x'-z)\langle x'-z \rangle^{\alpha}\\
    &\quad \qquad \times G_n(t_1-r,w-y')\langle w-y' \rangle^{\alpha}G_n(t_2-s,w'-z')\langle w'-z' \rangle^{\alpha}\\
    &\quad \qquad \times \langle x-x' \rangle^{-\alpha}\langle w-w' \rangle^{-\alpha}f(z-z')\\
    &\lesssim_T 
    \int_{\mathbb{R}^{6}}dzdz'\int_{{B^4_R}}dxdx'dwdw'\\
    &\quad \qquad \times G_n(t_2-s,x'-z)\langle x'-z \rangle^{\alpha}G_n(t_2-s,w'-z')\langle w'-z' \rangle^{\alpha}\\
    &\quad \qquad \times \langle x-x' \rangle^{-\alpha}\langle w-w' \rangle^{-\alpha}f(z-z')\\
    &= \int_{\mathbb{R}^{6}}dzdz'\int_{{B^2_R}}dwdw'G_n(t_2-s,w'-z')\langle w'-z' \rangle^{\alpha}\langle w-w' \rangle^{-\alpha}f(z-z')\\
    &\quad \qquad \times \left\{ \int_{B_R}G_n(t_2-s,x'-z)\langle x'-z \rangle^{\alpha}\left(\int_{B_R(-x')}\langle x \rangle^{-\alpha}dx \right)dx' \right\}\\
    &\leqslant 
    \int_{\mathbb{R}^{6}}dzdz'\int_{{B^2_R}}dwdw'G_n(t_2-s,w'-z')\langle w'-z' \rangle^{\alpha}\langle w-w' \rangle^{-\alpha}f(z-z')\\
    &\quad \qquad \times \left( \int_{B_R}G_n(t_2-s,x'-z)\langle x'-z \rangle^{\alpha}dx' \right) \left(\int_{B_{2R}}\langle x \rangle^{-\alpha}dx \right).
\end{align*}
Therefore, integrating in the order $dx',dz,dz'$, we have
\begin{align*}
    \mathbf{A_{14}}
    &\lesssim_T 
    \lVert f\rVert_{L^1(\mathbb{R})}\int_{B_{2R}}\langle x \rangle^{-\alpha}dx\int_{{B^2_R}} \langle w-w' \rangle^{-\alpha}dwdw'\\
    &\lesssim 
    \begin{cases}
        R^{9-2\alpha}  &(0<\alpha < 3)\\
        R^3(\log R)^2  &(\alpha = 3)\\
        R^3  &(\alpha > 3 )
    \end{cases}.
\end{align*}
Next we estimate $\mathbf{A_{21}}$. Integrating in the order $dx,dw,dy,dy'$, we get
\begin{align*}
    \mathbf{A_{21}} 
    &\lesssim_T \lVert g\rVert_{L^1(\mathbb{R}^3)}^2
    \int_{\mathbb{R}^{6}}dzdz'\int_{{B^2_R}}dx'dw'G_n(t_2-s,x'-z) G_n(t_2-s,w'-z')\langle z-z' \rangle^{-\alpha}.
\end{align*}
Thus, using Peetre's inequality and integrating with respect to $dz, dz'$, we obtain
\begin{align*}
    \mathbf{A_{21}}
    &\lesssim 
    \int_{\mathbb{R}^{6}}dzdz'\int_{{B^2_R}}dx'dw'\langle x'-w' \rangle^{-\alpha}\\
    &\qquad \times G_n(t_2-s,x'-z)\langle x'-z \rangle^{\alpha} G_n(t_2-s,w'-z')\langle w'-z' \rangle^{\alpha}\\
    &\lesssim_T 
    \int_{{B^2_R}}dx'dw'\langle x'-w' \rangle^{-\alpha}\\
    &\lesssim 
    \begin{cases}
        R^{6-\alpha}  &(0<\alpha < 3)\\
        R^3\log R  &(\alpha = 3)\\
        R^3  &(\alpha > 3 )
    \end{cases}.
\end{align*}
Next we consider $\mathbf{A_{22}}$. Integrating in the order $dx,dy$, we have 
\begin{align*}
    \mathbf{A_{22}} 
    &\lesssim_T \lVert g\rVert_{L^1(\mathbb{R}^3)}
    \int_{\mathbb{R}^{9}}dy'dzdz'\int_{{B^3_R}}dx'dwdw'\langle y'-z' \rangle^{-\alpha}\langle z-z' \rangle^{-\alpha}\\
    &\quad \qquad \qquad \quad \times G_n(t_2-s,x'-z)G_n(t_1-r,w-y')G_n(t_2-s,w'-z').
\end{align*}
Then, applying Peetre's inequality repeatedly and integrating with respect to $dz,dy',dz'$, we get
\begin{align*}
    \mathbf{A_{22}} 
    &\lesssim 
    \int_{\mathbb{R}^{9}}dy'dzdz'\int_{{B^3_R}}dx'dwdw'\langle w-w' \rangle^{-\alpha}\langle x'-w' \rangle^{-\alpha}\\
    &\qquad \times G_n(t_2-s,x'-z)\langle x'-z \rangle^{\alpha}G_n(t_1-r,w-y')\langle w-y' \rangle^{\alpha}\\
    &\qquad \times 
    G_n(t_2-s,w'-z')\langle w'-z' \rangle^{2\alpha}\\
    &\lesssim_T 
    \int_{{B^3_R}}dx'dwdw'\langle w-w' \rangle^{-\alpha}\langle x'-w' \rangle^{-\alpha}\\
    &\lesssim 
    \begin{cases}
        R^{9-2\alpha}  &(0<\alpha < 3)\\
        R^3(\log R)^2  &(\alpha = 3)\\
        R^3  &(\alpha > 3 )
    \end{cases}.
\end{align*}
$\mathbf{A_{23}}$ can be estimated as well as $\mathbf{A_{22}}$. 
Finally, we estimate the term $\mathbf{A_{24}}$.
Using Peetre's inequality iteratively and integrating with respect to $dy,dz,dy',dz'$, we obtain
\begin{align*}
    \mathbf{A_{24}} 
    &\lesssim 
    \int_{\mathbb{R}^{12}}dydy'dzdz'\int_{{B^4_R}}dxdx'dwdw'\langle x-x' \rangle^{-\alpha}\langle x'-w' \rangle^{-\alpha}\langle w'-w \rangle^{-\alpha}\\
    &\quad \qquad \times 
    G_n(t_1-r,x-y)\langle x-y \rangle^{\alpha}G_n(t_2-s,x'-z)\langle x'-z \rangle^{2\alpha}\\
    &\quad \qquad \times 
    G_n(t_1-r,w-y')\langle w-y' \rangle^{\alpha}G_n(t_2-s,w'-z')\langle w'-z' \rangle^{2\alpha}\\
    &\lesssim_T 
    \int_{{B^4_R}}dxdx'dwdw'\langle x-x' \rangle^{-\alpha}\langle x'-w' \rangle^{-\alpha}\langle w'-w \rangle^{-\alpha}\\
    &\leqslant 
    \int_{B_{2R}}\langle x \rangle^{-\alpha}dx\int_{{B^3_R}}dx'dwdw'\langle x'-w' \rangle^{-\alpha}\langle w'-w \rangle^{-\alpha}\\
    &\lesssim 
    \begin{cases}
        R^{12-3\alpha}  &(0<\alpha < 3)\\
        R^3(\log R)^3  &(\alpha = 3)\\
        R^3  &(\alpha > 3 )
    \end{cases}.
\end{align*}
As a consequence, combining altogether we have \eqref{order}.

\end{proof}

\bibliographystyle{plain}
\bibliography{main}

\end{document}